\documentclass[english,11 pt]{amsart}%
\usepackage{amsthm}
\usepackage{mathrsfs}
\usepackage{a4wide}
\usepackage{verbatim}
\usepackage[T1]{fontenc}
\usepackage[latin9]{inputenc}
\usepackage{amssymb}
\usepackage{amsfonts}
\usepackage{babel}
\usepackage{amsmath}
\usepackage{graphicx}
\usepackage{color}%
\usepackage{mathtools}

\usepackage{newtxtext}
\usepackage{newtxmath}

\usepackage{marginnote}

\usepackage[normalem]{ulem}
\newcounter{corr}
\definecolor{violet}{rgb}{0.580,0.,0.827}
\newcommand{\corr}[3]{\typeout{Warning : a correction remains in page
\thepage}
				\stepcounter{corr}
				{\color{blue}\ifmmode\text{\,\sout{\ensuremath{#1}}\,}\else\sout{#1}\fi}
        {\color{red}#2}
        {\color{violet} 
					#3
					} 
	}

\usepackage{hyperref}

\input colordvi

\newcommand\del[1]{}
\newcommand\dela[1]{}

\newcommand\comdel[1]{}

\newcommand\lb{\left\langle}
\newcommand\rb{\right\rangle}

\hypersetup{colorlinks,%
citecolor=black,%
filecolor=black,%
linkcolor=black,%
urlcolor=black,%
pdftex}

\makeatletter
\numberwithin{equation}{section}
\numberwithin{figure}{section}
\@ifundefined{theoremstyle}{\usepackage{amsthm}}{}
\theoremstyle{plain}
\newtheorem{thm}{Theorem}[section]
\theoremstyle{remark}

\newtheorem*{acknowledgement*}{Acknowledgement}
\theoremstyle{plain}

\theoremstyle{plain}

\newcounter{casectr}

\theoremstyle{remark}

\theoremstyle{remark}

\theoremstyle{definition}

\theoremstyle{plain}

\theoremstyle{plain}
\newtheorem{cor}[thm]{Corollary}
\theoremstyle{plain}

\theoremstyle{definition}

\theoremstyle{definition}
\newtheorem{defn}[thm]{Definition}
\theoremstyle{definition}

\theoremstyle{definition}

\theoremstyle{plain}

\theoremstyle{plain}
\newtheorem{lem}[thm]{Lemma}
\theoremstyle{remark}
\newtheorem{notation}[thm]{Notation}
\theoremstyle{definition}
\newtheorem{problem}[thm]{Problem}
\theoremstyle{plain}
\newtheorem{prop}[thm]{Proposition}
\theoremstyle{remark}
\newtheorem*{rem*}{Remark}
\theoremstyle{remark}
\newtheorem{rem}[thm]{Remark}
\theoremstyle{remark}

\theoremstyle{plain}
\newtheorem{assumption}[thm]{Assumption}

\numberwithin{equation}{section}
\makeatother

\newcommand{\ud}{\,\mathrm{d}}

\newcommand{\lmd}{\lambda}

\newcommand{\R}{\mathbb{R}}

\newcommand{\eps}{\varepsilon}
\newcommand{\vp}{\varphi}
\newcommand{\B}{\mathscr{B}}

\newcommand{\esup}{\mathop{\mathrm{ess}\sup}}

\newcommand{\lspan}{\mathrm{linspan}}
\newcommand{\llangle}{\left\langle}
\newcommand{\rrangle}{\right\rangle}

\newcommand{\E}{\mathbb{E}}

\newcommand{\cl}{\mathcal}
\newcommand{\bb}{\mathbb}

\newcommand{{\HH}}{\mathbb{H}}
\newcommand{\V}{\mathbb{V}}

\newcommand{\til}{\widetilde}

\def\la{\left(}
\def\ra{\right)}

\def\gjn{G_{jn}^\prime\la\til M_n\ra\left[G_{jn}\la\til M_n\ra\right]}
\def\gj{\la G^\psi_{j}\ra^\prime\la\til M\ra\left[G^\psi_{j}\la\til M\ra\right]}

\begin{document}
	
\title{3D Stochastic Landau-Lifshitz-Gilbert Equations coupled with Maxwell's Equations with full energy}
\author[Z Brze\'zniak, B. Goldys and Liang Li]{Zdzis{\l}aw Brze\'zniak, Ben Goldys and Liang Li}
\maketitle
\today
\begin{abstract}
	We consider 3D stochastic Landau-Lifshitz-Gilbert equations coupled with the Maxwell equations with the full energy. We have proved the existence of the weak solution as well as some further regularities of it.
\end{abstract}


\tableofcontents

\section{Introduction}
Landau-Lifshitz-Gilbert equation (LLGE) coupled to Maxwell equations provides a fundamental mathematical model for physical properties of ferromagnetic materials,
and it has been intensely investigated by physicists since the seminal work by Landau and Lifshitz \cite{Landau} and Gilbert \cite{Gilbert}. The exact form of the equation is determined by the energy functional that may include varying number of terms, so that in fact we have to deal with the whole family of equations. The first mathematical analysis of the LLGE corresponding to the full energy functional and coupled to the time dependent Maxwell equations was provided by Visintin in \cite{Visintin}.
\par\noindent
 It has been noticed by physics community a long time ago, that
 there are phase transitions between different equilibrium states, such phase transitions are induced
 by thermal fluctuations of the effective magnetic field. To describe the phase transitions, 
the noise must be included into the deterministic LLGE, see \cite{neel,brown,berkow}. Including the noise into LLGE requires sophisticated tools from the theory of  quasi-linear stochastic PDEs that have been missing for some time and are still not well developed. A rigorous mathematical theory of stochastic LLGE was initiated in \cite{ZB&BG&TJ1} and intensely studied since then, see \cite{LB&ZB&AP&MN,LB&ZB&AP&MNbk,ZB&BG&TJ,ZB&LL1,Goldys&Le,GGL}. In all these papers a simplified version of the energy functional is considered and so far the stochastic LLG equation associated to the full energy functional and coupled to  the time dependent Maxwell equations has never been studied. This is a serious deficiency since coupling with the Maxwell equations is fundamental for many physical phenomena, such as emergence and movement of boundary vortices, and movement of the domain walls, see \cite{moser}. Even for deterministic systems, the case with time dependent Maxwell equations is not well understood and after the seminal paper \cite{Visintin} most of the effort was focused on the so-called quasi-static case. Recently, the interest in the full time-dependent case has been renewed, see for example \cite{dumas,joly,zamponi}. In the stochastic case, the only work in this direction, we are aware of, is the paper \cite{GLT} but it imposes strong simplifying assumptions on the noise and the energy functional.
 \par\noindent
 In this paper we are concerned with the stochastic Landau-Lifshitz-Gilbert equation coupled to time dependent Maxwell equations and we assume that the evolution of spins is driven by the full energy functional described below. To be more precise, given the time horizon $T>0$ and a bounded open domain $\cl D\subset R^3$, the magnetization field $M:[0,T]\times \cl D\longrightarrow \R^3$ satisfies the Landau-Lifshitz-Gilbert equation:
\begin{equation}\label{eq:LLeq}
  \frac{\ud M(t,x)}{\ud t}=\lmd_1M(t,x)\times \rho(t,x)-\lmd_2M(t,x)\times (M(t,x)\times \rho(t,x)),
\end{equation}
where $\lmd_1\in\R$ and $\lmd_2>0$, subject to the constraint
\begin{equation}\label{eq:M=M0}
  |M(t,x)|=|M_0(x)|.
\end{equation}where $\rho$ is the effective field defined by
\begin{equation}\label{eq:effmag}
  \rho=-\nabla_M\mathcal{E}.
\end{equation}
Here $\mathcal{E}$ is the total electro-magnetic energy including anisotropy energy, exchange energy, magnetic field energy and electronic energy.

In order to describe phase transitions between different equilibrium states induced by thermal fluctuations of the effective field $\rho$, we introduce the Gaussian noise into the Landau-Lifschitz equation to perturb $\rho$ and so have the following stochastic Landau-Lifschitz-Gilbert equation (SLLGE):
\setlength\arraycolsep{2pt}{
	\begin{eqnarray}\label{eq:sllgrho}
		&&\ud M(t)=\lmd_1M(t)\times\big[\rho(t)\ud t+ \sum_jh_j\circ \ud W_j(t)\big]\\
		&&-\lmd_2M(t)\times\Big(M(t)\times\big[\rho(t)\ud t+ \sum_jh_j\circ\ud W_j(t)\big]\Big),\nonumber
\end{eqnarray}}
where $\{W_j\}_j$ are independent real-valued Wiener processes and $\{h_j\}_j$ are coefficients with good enough regularities. The reason we choose Stratonovich type noise is that we want to keep the geometric property \eqref{eq:M=M0} of the SLLG equations. More detailed assumptions and discussion can be found in the statement of Problem \ref{S-LLG} and Remark \ref{rem:QW}.

Since the magnetic field energy and electronic energy are related to the magnetic field and electric field,
we also consider the magnetic field $H:[0,T]\times \R^3\longrightarrow \R^3$ and the electric field $E:[0,T]\times \R^3\longrightarrow \R^3$ in this paper. We denote
\begin{equation*}
    B:=H+\overline{M},
  \end{equation*}
  where
  \[
  \overline{M}(x):=\left\{\begin{array}{ll}
  M(x),&x\in \cl D;\\
  0,&x\notin \cl D.
  \end{array}\right.\]
Then $B$ and $E$ are related by the Maxwell's equation:
\begin{equation*}
   \ud B=\nabla\times E\ud t.
  \end{equation*}
\begin{equation*}
    \ud E=\nabla\times[B-\overline{M}]\ud t-[1_{\cl D}E+\overline{f}(t)]\ud t,
  \end{equation*}
  {
where $f$  is a map
\[f:[0,T]\times \R^3\longrightarrow \R^3,\]
which is a given non-inductive applied electromotive field.}

Summarising, the equation we are going to study in this paper has the following form:
\setlength\arraycolsep{2pt}{
  \begin{eqnarray*}
\ud M(t)&=&\left[\lmd_1 M\times \rho-\lmd_2 M\times(M\times \rho)\right]\ud t\\
  &&+\sum_{j=1}^\infty\left\{\left[ M\times h_j+  M\times(M\times h_j)\right]\circ\ud W_j(t)\right\}.
  \end{eqnarray*}}
\begin{equation*}
   \ud B(t)=\nabla\times E(t)\ud t.
  \end{equation*}
\begin{equation*}
    \ud E(t)=\nabla\times[B(t)-\overline{M}(t)]\ud s-[1_{\cl D}E(t)+\overline{f}(t)]\ud t.
  \end{equation*}
  \[\left.\frac{\partial M}{\partial \nu}\right|_{\partial \cl D}=0. \]
 \[M(0)=M_0,\quad B(0)=B_0,\quad E(0)=E_0.\]

 This paper is constructed as follows. In section 2, firstly we give all the formal definitions of all the energies and state the problem we consider. Secondly we give the definition of the solution of the stochastic differential equation. And at last we formulate the main result (Theorem \ref{thm:mainthm}) of the whole paper. In section 3, we construct a series of some auxiliary equations \eqref{eq:SVis 3.29R}, with all the elements in a finite dimensional space and prove the existence and uniqueness of the global solution of the finite dimensional equations. In section 4, we get some a'priori estimates of the series of solutions of equations \eqref{eq:SVis 3.29R}. In section 5, we show the laws of the finite dimensional solutions are tight on some spaces. In section 6, we construct a new probability space by the Skorohod Theorem in which there exist limit processes $\til{M},\til{B},\til{E}$ of the solutions of \eqref{eq:SVis 3.29R}. In section 7, we prove that the $\til{M},\til{B},\til{E}$ which we got in section 6 are actually the weak solution of our original problem. In section 8, we show some more regularities of the weak solution. Finally in section 9, we complete the proof of the main result, i.e. the Theorem \ref{thm:mainthm}. In the Appendix, we list the important lemmata which are used in this paper.
 
 By the end of this introduction, it may worth to mention that the uniqueness of 3-dimensional LLG equations is an open problem, we do not discuss it in this paper.

\section{Statement of the problem and formulation of the main result}
\begin{assumption}
Throughout this paper we assume $\cl D\subset \R^3$ to be a bounded open domain with $C^2$ boundary.
\end{assumption}
\begin{notation}
\begin{trivlist}
	\;
\item[(1)]  We use the following notations for the classical functional spaces:
  $$\mathbb{L}^p:=L^p(\cl D;\R^3) \textrm{ or }L^p(\cl D;\R^{3\times 3}),\quad \bb L^p(\R^3):=L^p(\R^3;\R^3)$$
  $$\mathbb{W}^{k,p}:=W^{k,p}(\cl D;\R^3),\;\mathbb{H}^{k}:=H^{k}(\cl D;\R^3)=W^{k,2}(\cl D;\R^3),\;\textrm{and }\mathbb{V}:=\mathbb{W}^{1,2},\;\mathbb{H}:=\mathbb{L}^2.$$
\item[(2)] The duality between a Banach space $X$ and its dual $X^*$ will be denoted by ${}_{X^*} \langle\cdot, \cdot\rangle_X$. The notations $ \langle\cdot, \cdot\rangle_{K}$ and $\| \cdot \|_{K}$ stand for the scalar product and its associated norm in a given Hilbert space $K$ respectively. The norm of a vector $x\in\R^d$ will be denoted by $|x|$ and the inner product in $\R^d$ will be denoted by $ \langle\cdot, \cdot\rangle$ for any $d$.
\item[(3)] For a function $\vp:\R^3\to\R$ we will write
\[\vp':=\nabla\vp,\quad\mathrm{and}\quad\vp^{\prime\prime}:=\nabla^2\vp\,.\]
\item[(4)] For a function $u:\cl D\rightarrow \R^3$, we denote
\[
  \overline{u}(x):=\left\{\begin{array}{ll}
  u(x),&x\in \cl D,\\
  0,&x\notin \cl D.
  \end{array}\right.\]
\item[(5)] For $u\in\mathbb{L}^2(\R^3)$, we define the distribution $\nabla \times u$ by
\[{}_{\mathscr{D}'}\llangle \nabla\times u, v\rrangle_{\mathscr{D}}=\llangle u,\nabla\times v\rrangle_\HH,\quad v\in C_0^\infty\left(\R^3,\R^3\right).\]
Then we define the Hilbert space
$$\mathbb{Y}:=\left\{u\in \mathbb{L}^2(\R^3):\nabla\times u\in \mathbb{L}^2(\R^3)\right\},$$
with the inner product
\[\llangle u,v\rrangle_\mathbb{Y}:=\llangle u,v\rrangle_{\mathbb{L}^2(\R^3)}+\llangle \nabla\times u,\nabla\times v\rrangle_{\mathbb{L}^2(\R^3)}.\]
\item[(6)] For a fixed $h\in\mathbb L^\infty$ and $\lmd_1\in\R,\,\lmd_2>0$ we define a mapping $G_h$ by
\[\mathbb L^4\ni u\longrightarrow G_h(u)=\lambda_1 u\times h-\lambda_2 u\times(u\times h)\in\mathbb L^2\,.\]
For a given sequence $\{ h_j\}_{j=1}^\infty\subset\mathbb L^\infty$  we will use the notation $G_j(u):=G_{h_j}(u)$.
\item[(7)] To avoid too long equations, we may simply use $u$ to denote $u(t,x)$.
\end{trivlist}
\end{notation}
\begin{defn}[Magnetic Induction]
Given a magnetization field $M:\cl D\longrightarrow\R^3$ and a magnetic field $H:\R^3\longrightarrow\R^3$, we define the magnetic induction as a vector field $B:\R^3\longrightarrow\R^3$ by
  \begin{equation}\label{eq:def B}
    B:=H+\overline{M}.
  \end{equation}
\end{defn}

\begin{defn} (The energy)
\begin{trivlist}
    \item[(i)] Suppose that $\vp\in C_0^2(\R^3;\R^+)$. For a magnetization field $M\in \mathbb{V}$, we define the anisotropy energy of $M$ by:
  \[\mathcal{E}_{an}(M):=\int_{\cl D}\vp(M(x))\ud x.\]
    \item[(ii)]  We define the exchange energy of $M$ by:
    \begin{equation}\label{eq:ex eng}
    \mathcal{E}_{ex}(M):=\frac{1}{2}\int_{\cl D}|\nabla M(x)|^2\ud x=\frac{1}{2}\|\nabla M\|_{\mathbb L^2}^2.
    \end{equation}
    \item[(iii)] For  a magnetic field $H\in \bb{L}^2(\R^3)$, we define the Zeeman energy by:
    \begin{equation}\label{eq:fi eng}
    \mathcal{E}_{ze}(H):=\frac{1}{2}\int_{\R^3}|H(x)|^2\ud x=\frac{1}{2}\|H\|_{\mathbb{L}^2(\R^3)}^2=\frac{1}{2}\left\|B-\overline M\right\|_{\mathbb{L}^2(\R^3)}^2.
    \end{equation}
\end{trivlist}

 Finally, given an electric field $E\in \mathbb{L}^2(\R^3)$, a magnetization field $M\in \mathbb{V}$ and a magnetic field $H\in \mathbb{L}^2(\R^3)$, (hence the magnetic induction  $B\in \mathbb{L}^2(\R^3)$) we define the total electro-magnetic energy by
 \begin{equation}\label{eq_te}
\mathcal E(M,B,E)=\int_{\mathcal D}\varphi(M(x))\,dx+\frac{1}{2}\|\nabla M\|^2_{\mathbb H}+\frac{1}{2}\left\|B-\overline M\right\|^2_{\mathbb L^2(\R^3)}+\frac{1}{2}\|E\|^2_{\mathbb L^2(\R^3)}
\end{equation}
\end{defn}

To investigate some properties of the total energy $\mathcal{E}$, we need the following Lemma, the proof of which is straightforward. 
\begin{lem}\label{lem:dEdM}
  For $M\in \mathbb{V}$, if we define $\Delta M\in \mathbb{V}^*$ by
  \begin{equation}
    {}_{\mathbb{V}^*}\llangle\Delta M,u\rrangle_{\mathbb{V}}:=-\llangle\nabla M,\nabla u\rrangle_{\mathbb L^2},\qquad\forall u\in \mathbb{V}.
  \end{equation}
  Then the total energy $\mathcal{E}:\mathbb{V}\times\mathbb{L}^2(\R^3)\times\mathbb{L}^2(\R^3)\longrightarrow\R$ defined in \eqref{eq_te} has partial derivatives of 2nd order with respect to $M, B$ and $E$ well defined and:
  \begin{equation}\label{eq:dEdM}
    \frac{\partial \mathcal{E}}{\partial M}(M,B,E)=\vp'(M)-(1_{\cl D}B-M)-\Delta M,\qquad\textrm{in } \mathbb{V}^*\,,
  \end{equation}
   for $u,v\in \mathbb{V}$,
  \begin{equation}\label{eq:d2EdM2}
  \frac{\partial^2\mathcal{E}}{\partial M^2}(M,B,E)(u,v)=\int_{\cl D}\vp''(M(x))(u(x),v(x))\ud x+\llangle u,v\rrangle_{\mathbb{V}}\,,
  \end{equation}
    \begin{equation}\label{eq:dEdB}
  \frac{\partial \mathcal{E}}{\partial B}(M,B,E)=B-{\overline{M}}\,,
\end{equation}
\begin{equation}\label{eq:dEdE}
  \frac{\partial \mathcal{E}}{\partial E}(M,B,E)=E\,.
\end{equation}
\end{lem}

Now we can define the effective field, which is the partial derivative of the total energy. 
\begin{defn}[Effective field]
	We define the effective field $\rho\in \bb V^*$ as
	\begin{equation}\label{eq:rho5}
		\rho:=\vp'(M)-(1_{\cl D}B-M)-\Delta M,\qquad\textrm{in } \mathbb{V}^*.
	\end{equation}
\end{defn}


We are now ready to formulate the problem we are going to study in this paper.
 
\begin{problem}\label{S-LLG}
 Let $\left(\Omega,\mathcal{F},\mathbb{F}=(\mathcal{F}_t)_{t\geq 0},\mathbb{P}\right)$ be a filtered probability space, and let $W=\{W_j\}_{j=1}^\infty$ be a set of independent,  real valued, $\mathbb{F}-$\dela{-adapted}Wiener processes.
Let 
$$M_0\in  \V \textrm{ with }|M_0(x)|=1\textrm{ for all }x\in\cl D   ;$$
$$B_0\in\mathbb{L}^2(\R^3);\quad\nabla\cdot B_0=0,\quad \textrm{in } \mathscr{D}'(\R^3;\R);\qquad E_0\in\mathbb{L}^2(\R^3);$$
\begin{equation}\label{ass_h}
\begin{aligned}
c_h^2=\sum_{j=1}^\infty\left\|h_j\right\|_{\mathbb{L}^\infty}+\sum_{j=1}^\infty\left\|h_j\right\|^2_{\V}&<\infty.\end{aligned}
\end{equation}
$$f\in L^2(0,T;\mathbb{H});\qquad \vp\in C_0^2(\R^3;\R^+);$$
$$ \lmd_1\in\R,\quad \lmd_2>0\,.$$

Our aim is to show that the following system of stochastic PDEs has a solution in the sense made precise below:
 \begin{equation}\label{eq:SLLG}
 \left\{
  \begin{aligned}
\ud M(t)&=\left[\lmd_1 M\times \rho-\lmd_2 M\times(M\times \rho)\right]\ud t+\sum_{j=1}^\infty G_j(M)\circ\ud W_j(t)\,,\\
   \ud B(t)&=\nabla\times E(t)\ud t\,,\\
    \ud E(t)&=\nabla\times[B(t)-\overline{M(t)}]\ud s-[1_{\cl D}E(t)+\overline{f(t)}]\ud t\,.
  \end{aligned}\right.\end{equation}
  with the boundary conditions
  \[
  \left.\frac{\partial M(t)}{\partial \nu}\right|_{\partial \cl D}=0,\qquad t\ge0, \textrm{ where $\nu$ is the exterior normal vector on $\partial \cl D$.}
  \]
  and the initial conditions
\[
 M(0)=M_0,\quad B(0)=B_0,\quad E(0)=E_0.
 \]
\end{problem}
 
The Stratonovich equation in \eqref{eq:SLLG} can be rewritten as an It\^o equation
	\setlength\arraycolsep{2pt}{
	\begin{eqnarray}\label{eq:SLLG1}
\ud M&=&\left[\lmd_1 M\times \rho-\lmd_2 M\times(M\times \rho)+\frac{1}{2}\sum_{j=1}^\infty G_j^\prime(M)G_j(M)\right]\ud t+\sum_{j=1}^\infty G_j(M)\ud W_j\\
&=&\Bigg\{\lmd_1 M\times \rho-\lmd_2 M\times(M\times \rho)+\frac{1}{2}\sum_{j=1}^\infty\bigg[\lmd_1^2\left[(M\times h_j)\times h_j\right]+\lmd_1\lmd_2\left[\big[M\times(M\times h_j)\big]\times h_j\right]\nonumber\\
&&+\lmd_2^2 \left[M\times\big[M\times(M\times h_j)\times h_j\big]\right]+\lmd_1\lmd_2 \left[M\times\big[(M\times h_j)\times h_j\big]\right]\nonumber\\
&&+\lmd_2^2\left[ M\times(M\times h_j)\times(M\times h_j)\right]\bigg]\Bigg\}\ud t+\sum_{j=1}^\infty\bigg\{\lmd_1 \left[M\times h_j\right]+\lmd_2 \left[ M\times(M\times h_j)\right]\bigg\}\ud W_j\nonumber.
\end{eqnarray}

\begin{rem}\label{rem:QW}
	We can understand the noise as a $Q$-Wiener process $W_h(t):=\sum_{j=1}^\infty W_j(t)h_j\sim N(0,tQ)$ on $\HH$ for some operator $Q$ which is nonnegative, symmetric and with finite trace.
	
	In fact, $W_h(0)=0$ a.s. is obvious.
	
	By \cite[Proposition 3.18]{DP&ZSE}, $W_jh_j$ can be viewed as a random variable taking values on $C([0,T];\HH)$ for each $j$.
	And by our assumption of $h_j$ as in \eqref{ass_h} and using Doob's maximal inequality, we can show $\{\sum_{j=1}^n W_jh_j\}_n$ is a Cauchy sequence in $L^2(\Omega; C([0,T];\HH))$ 
	 (We put the proof in the Appendix, Proposition \ref{prop:sumWjhjCh}). Therefore its limit $W_h\in L^2(\Omega;C([0,T];\HH))$. Hence $W_h$ has continuous trajectory almost surely.
	
	The independence of increment of $W_h$ follows from the fact that $W_j$ are independent for all different $j$ and they all have independent increments for each $j$.
	
	So it only remains to check the distribution of $W_h$ on $\HH$. Since $W_h$ is the sum of independent normal random variables with mean $0$, it has normal distribution with mean $0$ as well.
	
	Next we try to find its covariance operator. For any $u,v\in \HH$, we have
	\setlength\arraycolsep{2pt}{
		\begin{eqnarray*}	
			\E\left(\llangle W_h(t),u\rrangle_\HH\llangle W_h(t),v\rrangle_\HH\right)&=&\E\left(\llangle \sum_{j=1}^\infty W_j(t)h_j,u\rrangle_\HH\llangle\sum_{j=1}^\infty W_j(t)h_j,v\rrangle_\HH\right)\\
			&=&t \sum_{j=1}^\infty\llangle h_j,u\rrangle_\HH \llangle h_j,v\rrangle_\HH.
		\end{eqnarray*}}
	So the covariance operator $Q$ is uniquely determined by
	\[\llangle Qu,v\rrangle_\HH=\sum_{j=1}^\infty\llangle h_j,u\rrangle_\HH \llangle h_j,v\rrangle_\HH,\qquad u,v\in\HH.\]
	We can check that the operator $Q$ is nonnegative, symmetric and with finite trace.
	
	So $W_h$ is really a $Q$-Wiener process for some nonnegative, symmetric and with finite trace operator $Q$. Hence it has a representation
	\[W_h(t)=\sum_{j=1}^\infty \bar W_j(t)\bar h_j,\]
	Where $\bar W_j$ are independent $1$-dimensional Brownian motions and $\{\bar h_j\}_j$ is an ONB of $\HH$ which consists of eigenvectors of $Q$.
	
	Therefore we can actually assume that $\{h_j\}_j$ is an ONB of $\HH$.
\end{rem}
\begin{defn}[Weak martingale solution of equation \eqref{eq:SLLG}]\label{defn:sol}
Given $T>0$, a \textit{weak martingale solution} of equation \eqref{eq:SLLG} is a set consisting of a filtered probability space $(\til{\Omega},\til{\mathcal{F}},\til{\mathbb{F}},\til{\mathbb{P}})$,
an $\infty$-dimensional $\til{\mathbb{F}}$-Wiener process
$\til{W}=(\til{W}_j)_{j=1}^\infty$ and $\til{\mathbb{F}}$-progressively measurable processes

\[\til M:[0,T]\times \til\Omega\longrightarrow \mathbb{V}\cap \bb L^\infty,\;\til B:[0,T]\times\til\Omega\longrightarrow \mathbb{L}^2(\R^3),\;\til E:[0,T]\times\til\Omega\longrightarrow \mathbb{L}^2(\R^3)\]
such that  for all the test functions $u\in \mathbb{V}\cap \bb L^\infty$, $v\in \mathbb{Y}$ and $t\in[0,T]$, we have the following equalities holding $\til{\mathbb{P}}$-a.s.:
\setlength\arraycolsep{2pt}{
  \begin{eqnarray}\label{eq:dM}
    &&\int_{\cl D}\llangle \til{M}(t)-M_0, u\rrangle\ud x\label{eq:Vis 3.5}\\
    &=&\int_0^t\int_{\cl D}\Bigg\{\llangle \til{B}-\til{M}-\vp'(\til{M}),\lmd_1u\times \til{M}-\lmd_2(u\times \til{M})\times \til{M}\rrangle\nonumber\\
    &&-\sum_{i=1}^3\llangle \nabla_i \til{M},\lmd_1\nabla_iu\times \til{M}-\lmd_2\left(\nabla_iu\times \til{M}+u\times\nabla_i \til{M}\right)\times \til{M}\rrangle\Bigg\}\ud x\ud s\nonumber\\
    &&+\sum_{j=1}^\infty\int_0^t\left\langle G_j\left(\til M\right),u\rrangle\circ\ud \til W_j(s);\nonumber
  \end{eqnarray}}
  \begin{equation}\label{eq:dB}
        \int_{\R^3}\llangle \til{B}(t)-B_0,v\rrangle\ud x=-\int_0^t\int_{\R^3}\llangle \til{E},\nabla\times v\rrangle \ud x\ud s;
  \end{equation}
  \begin{equation}\label{eq:dE}
    \int_{\R^3}\llangle \til{E}(t)-E_0, v\rrangle\ud x=\int_0^t\int_{\R^3}\llangle \til{B}-\overline{\til M},\nabla\times v\rrangle\ud x\ud s-\int_0^t\int_{\cl D}\llangle \til{E}+f,v\rrangle\ud x\ud s.
  \end{equation}
\end{defn}

Next we would like to formulate the main result of this paper:
\begin{thm}\label{thm:mainthm}

 There exists a weak martingale solution of Problem \ref{S-LLG} with the following stronger regularity properties:
  \begin{trivlist}
    \item[(i)]
    \begin{equation}\label{eq:MregV}
    \til M\in L^{2r}(\til\Omega;L^\infty(0,T;\V)),\quad \forall r>0; \quad \til B,\til E\in L^2(\til\Omega;L^2(0,T;\bb L^2(\R^3)));
    \end{equation}
    \begin{equation}\label{eq:Mxrhoreg}
    \til M\times \til\rho\in L^{2}(\til{\Omega};L^2(0,T;\mathbb{H})),\quad\textrm{where $\til{\rho}:=-\vp'(\til{M})+1_{\cl D}\til{B}-\til{M}+\Delta \til{M}$.}
    \end{equation}
\begin{equation}
	\Delta \til M\in L^1(\til \Omega;L^1(0,T;\bb L^1)).
\end{equation}
 for any $T>0$.
    \item[(ii)] For every $t\in[0,\infty)$, the equation
  \setlength\arraycolsep{2pt}{
  \begin{eqnarray}
    \til{M}(t)&=&M_0+\int_0^t\left\{\lmd_1 \til{M}\times \til{\rho}-\lmd_2 \til{M}\times(\til{M}\times \til{\rho})\right\}\ud s\label{eq:MtinH}+\sum_{j=1}^\infty\int_0^tG_j(\til{M})\circ\ud \til{W}_j(s),\nonumber
  \end{eqnarray}}
holds in $L^2(\til{\Omega};\mathbb{H})$.
   \begin{equation}\label{eq:Vis 3.7}
    \til{B}(t)=B_0-\int_0^t\nabla\times \til{E}\ud s\in \mathbb{Y}^*,\quad\til{\mathbb{P}}-a.s.
  \end{equation}
  \begin{equation}\label{eq:Vis 3.6}
    \til{E}(t)=E_0+\int_0^t\nabla\times[\til{B}-\overline{\til{M}}]\ud s-\int_0^t[1_{\cl D}\til{E}+\overline{f}]\ud s\in \mathbb{Y}^*,\;\til{\mathbb{P}}-a.s.
  \end{equation}
 \item[(iii)]
  \begin{equation}\label{eq:|M|=1}
  |\til{M}(t,x)|=1, \quad\textrm{for Lebesgue a.e. } (t,x)\in [0,\infty)\times \cl D \textrm{ and }\til{\mathbb{P}}-a.s.
  \end{equation}
    \item[(iv)] For every $\theta\in \left(0,\frac{1}{2}\right)$,
  \begin{equation}\label{eq:MregtC}
  \til{M}\in C^\theta([0,T];\mathbb{H}),\qquad \til{\mathbb{P}}-a.s.
  \end{equation}
  \end{trivlist}
\end{thm}
\begin{rem}
  In \eqref{eq:Mxrhoreg} $\til\rho$ is a distribution, but  $\til M\times \til\rho\in L^{2}(\til{\Omega};L^2(0,T;\mathbb{H}))$. Precise definition is provided in Notation \ref{nt:Mxrho}.
 \end{rem}
 \begin{rem}
 Equality \eqref{eq:Vis 3.7} implies $\nabla\cdot B(t)=0$, for all $t\in[0,T]$.
\end{rem}

\section{Galerkin approximation}
In this section we start to prove the existence of the martingale solution to Problem \ref{S-LLG}.  We begin with the classical Galerkin approximation.
Let $A$ denoteus the negative Laplace operator in $\cl D$ with
 the homogeneous Neumann boundary condition:
\[D(A):=\left\{u\in \mathbb{H}^{2}:\,\frac{\partial u}{\partial \nu}\Big|_{\partial \cl D}=0\right\},\qquad A:=-\Delta\,, \]
where $\nu$ stands for the outer normal to the boundary of $\cl D$. The operator
$A$ is self-adjoint and there exists an orthonormal basis $\{e_k:\,k\ge 1\}\subset C^\infty\left({\cl D};\R^3\right)\cap D(A)$ of $\HH$ that consists of eigenvectors of $A$. We set $\mathbb{H}_n=\lspan\{e_1,e_2,\ldots,e_n\}$ and denote by $\pi_n$ the orthogonal projection from $\HH$ to $\mathbb{H}_n$. We also note that $\mathbb{V}=D\left(A_1^{1/2}\right)$ for $A_1:=I+A$, and $\|u\|_{\mathbb{V}}=\left\|A_1^{1/2}u\right\|_{\mathbb{H}}$ for $u\in \mathbb{V}$.\\
The following properties of the operator $A$ will be frequently used later:
for any $u\in D(A)$ and $v\in \V$,
  \[\llangle Au,v\rrangle_\mathbb{H}=\int_{\cl D}\Big(\nabla u(x),\nabla v(x)\Big)_{\R^{3\times 3}}\ud x,\]
and
  \begin{equation}\label{eq:2.9}
   \llangle u\times Au,v\rrangle_\HH=\sum_{i=1}^3\left\langle \nabla_iu,\nabla_i v\times u\right\rangle_\HH.
  \end{equation}

Let $\{y_n\}_{n=1}^\infty\subset C_0^\infty(\R^3;\R^3)$ be an orthonormal basis of $\mathbb{L}^2(\R^3)$.
We define $\mathbb{Y}_n:=\lspan\{y_1,\ldots,y_n\}$ and the orthogonal projections
$$\pi_n^\mathbb{Y}:\mathbb{L}^2(\R^3)\longrightarrow \mathbb{Y}_n\quad\textrm{and}\quad\pi_n^\mathbb{Y}|_\mathbb{Y}:\mathbb{Y}\longrightarrow \mathbb{Y}_n,\qquad n\in \bb N.$$

On $\mathbb{H}_n$ and $\mathbb{Y}_n$ we consider the scalar product inherited from $\HH$ and $\mathbb{Y}$ respectively.
Let us denote by $\mathcal{E}_n$ the restriction of the total energy functional $\mathcal{E}$ to the finite dimensional space $\mathbb{H}_n\times \mathbb{Y}_n\times \mathbb{Y}_n$, i.e.
\[\mathcal{E}_n:\mathbb{H}_n\times \mathbb{Y}_n\times \mathbb{Y}_n\longrightarrow \R,\]
\begin{equation}\label{eq:calEn}
  \mathcal{E}_n(M,B,E)=\int_{\cl D}\vp(M(x))\ud x+\frac{1}{2}\|\nabla M\|_{\mathbb{H}}^2\nonumber+\frac{1}{2}\left\|B-\pi_n^{\bb Y}\overline{M}\right\|_{\mathbb{L}^2(\R^3)}^2+\frac{1}{2}\|E\|_{\mathbb{L}^2(\R^3)}^2.
\end{equation}

The proof of the following Lemma is straightforward by the definition of Fr{\^e}chet derivative, so we only state the result.
\begin{lem}
The function $\mathcal{E}_n$ is of class $C^2$ and for $M\in \mathbb{H}_n$, $B,E\in \mathbb{Y}_n$ we have:
\begin{trivlist}
\item[(i)]
  \begin{equation}\label{eq:dEnM}
  (\nabla_M\mathcal{E}_n)(M,B,E)=\pi_n\big[\vp'(M)-1_{\cl D}(B-\pi_n^{\bb Y}\overline{M})\big]-\Delta M\,,
\end{equation}
\item[(ii)]
\begin{equation}\label{eq:dEnB}
  (\nabla_B\mathcal{E}_n)(M,B,E)=B-\pi_n^{\bb Y}\overline{M}\,,
\end{equation}
\item[(iii)]
\begin{equation}\label{eq:dEnE}
  (\nabla_E\mathcal{E}_n)(M,B,E)=E\,,
\end{equation}
\item[(iv)]
\begin{equation}\label{eq:d2EnM}
  \frac{\partial^2\mathcal{E}_n}{\partial M^2}(M,B,E)(u,v)=\int_{\cl D}\vp''(M(x))(u(x),v(x))\ud x+\llangle u,v\rrangle_{\mathbb{V}},\qquad u,v\in\V.
\end{equation}
\end{trivlist}
\end{lem}
\begin{notation}
Let us define the function $\rho_n:\mathbb{H}_n\times \mathbb{Y}_n\times \mathbb{Y}_n\longrightarrow \mathbb{H}_n$ which corresponds to $\rho$ by:
\begin{equation}\label{eq:rhon}
  \rho_n:=-(\nabla_M\mathcal{E}_n)(M_n,B_n,E_n)=\pi_n\big[-\vp'(M_n)+1_{\cl D}(B_n-\pi_n^{\bb Y}\overline{M}_n)\big]+\Delta M_n\in \mathbb{H}_n.
\end{equation}

We will also need a function $\psi:\R^3\longrightarrow\R$ such that $\psi\in C^1\left(\R^3\right)$,
\[\psi(x)=\left\{\begin{array}{ll}
1,&\quad |x|\leq3,\\
0,&\quad |x|\geq5,\end{array}\right.\]
and  $|\nabla\psi|\leq1$.
\begin{rem}
	The $\psi$ defined above is used to  truncate $M$ in order to make sure we can get the estimates in Proposition \ref{prop:estimates} below. The setting of $|\nabla\psi|\leq1$ is also necessary, for instance in the proof of Lemma \ref{lem:4 ineqM}. By Lemma \ref{lem:M=1}, we will prove that $|M(t,x)|=1$ for almost every $x\in \cl D$, therefore we can remove this $\psi$ by the end.
\end{rem}
It also will be convenient to define mappings

 $F_n:\mathbb{H}_n\times \mathbb{Y}_n\times \mathbb{Y}_n\longrightarrow \mathbb{H}_n$ and $G_{nj}:\mathbb{H}_n\longrightarrow \mathbb{H}_n$, $j=1,2,\ldots$, by
\begin{equation}\label{eq:Fn}
		F_n(M_n,B_n,E_n):=\lmd_1\pi_n\left[ M_n\times \rho_n\right]-\lmd_2 \pi_n\left[M_n\times(M_n\times \rho_n)\right]	+\frac{1}{2}\sum_{j=1}^\infty G_{jn}^\prime(M_n)\left[G_{jn}(M_n)\right],
\end{equation}
\begin{equation}\label{eq:Gn}
	G_{jn}(M_n):=\lmd_1 \pi_n\left[M_n\times h_j\right]+\lmd_2 \pi_n\left[\psi(M_n) M_n\times(M_n\times h_j)\right].
\end{equation}

where
\setlength\arraycolsep{2pt}{
		\begin{eqnarray}
	&&G^\prime_{jn}\la M_n\ra\left[G_{jn}\la M_n\ra\right]:=\lmd_1^2\pi_n\left[\pi_n(M_n\times h_j)\times h_j\right]\label{eq:Gjn'}\\
	&&\qquad+\lmd_1\lmd_2\pi_n\left[\psi(M_n)\big[M_n\times(M_n\times h_j)\big]\times h_j\right]
	+\lmd_2^2 \pi_n\left[\psi(M_n)M_n\times\big[(M_n\times(M_n\times h_j))\times h_j\big]\right]\nonumber\\
	&&\quad+\lmd_1\lmd_2\pi_n\left[ \psi(M_n)M_n\times\big[(M_n\times h_j)\times h_j\big]\right]\nonumber
	+\lmd_2^2\pi_n\left[\pi_n\big[\psi(M_n)M_n\times(M_n\times h_j)\big]\times(M_n\times h_j)\right]
\end{eqnarray}}
note that because of the $\psi$, \eqref{eq:Gjn'} is only a notation, not the Fr\'echet derivative of $G_{jn}$.

Similar as \eqref{eq:Gjn'}, we will also use the following notations
\begin{equation}
	G_j^\psi(M):=\lmd_1 M\times h_j+\lmd_2\psi(M)M\times(M\times h_j),
\end{equation}
and
\setlength\arraycolsep{2pt}{
	\begin{eqnarray}
		&&(G_j^\psi)^\prime\la M\ra\left[G_j^\psi\la M\ra\right]:=\lmd_1^2\left[(M\times h_j)\times h_j\right]\label{eq:Gjpsi'}\\
		&&\qquad+\lmd_1\lmd_2\left[\psi(M)\big[M\times(M\times h_j)\big]\times h_j\right]
		+\lmd_2^2 \left[\psi(M)M\times\big[(M\times(M\times h_j))\times h_j\big]\right]\nonumber\\
		&&\quad+\lmd_1\lmd_2\pi_n\left[ \psi(M)M\times\big[(M\times h_j)\times h_j\big]\right]\nonumber
		+\lmd_2^2\left[\big[\psi(M)M\times(M\times h_j)\big]\times(M\times h_j)\right]
	\end{eqnarray}}
\end{notation}
\begin{rem}
	It may looks like there are too many $\pi_n$s in \eqref{eq:Gjn'}, but all of them are necessary. It is not only we want all the terms of \eqref{eq:Gjn'} are in $\HH_n$, but we also want to get the a'priori estimates in Proposition \ref{prop:estimates}.
\end{rem}
To solve Problem \ref{S-LLG}, we first consider the following system of equations in $\mathbb H_n$, $\mathbb Y_n$ and $\mathbb Y_n$:

\begin{problem}\label{problem:SLLGn}
	Let us consider the following $n$-dimensional system:
\begin{equation}\label{eq:SVis 3.29R}
\left\{
  \begin{aligned}
  \ud M_n(t)=&F_n(M_n(t),B(t),E(t))\ud t+\sum_{j=1}^\infty G_{jn}\la M_n(t)\ra\ud W_j\\
\ud E_n(t)=&-\pi_n^\mathbb{Y}\big[1_{\cl D}(E_n(t)+\overline f(t))\big]\ud t+\pi_n^\mathbb{Y}\big[\nabla\times (B_n(t)-\pi_n^{\bb Y}\overline{M}_n(t))\big]\ud t\\
 \ud B_n(t)=&-\pi_n^\mathbb{Y}\big[\nabla\times E_n(t)\big]\ud t
  \end{aligned}\right.
  \end{equation}
  with the initial conditions
  \begin{equation}  \label{eq:ndinicon}
  	M_n(0)=\pi_nM_0,\qquad E_n(0)=\pi_n^\mathbb{Y}E_0,\qquad B_n(0)=\pi_n^\mathbb{Y}B_0\,.
  	\end{equation}
\end{problem}

\begin{lem}\label{lem:ndimsolp}
   There exists a unique global strong solution $(M_n, B_n, E_n)$ of Problem \ref{problem:SLLGn}. In particular, 
   $(M_n,B_n, E_n)\in C([0,\infty);\HH_n\times\mathbb{Y}_n\times \mathbb{Y}_n)$, $\mathbb{P}$-almost surely.  
\end{lem}
\begin{proof}
  We define mappings
  \[\begin{aligned}
   \widehat{F}_n:\mathbb{H}_n\times \mathbb{Y}_n\times \mathbb{Y}_n&\longrightarrow\mathbb{H}_n\times \mathbb{Y}_n\times \mathbb{Y}_n\\
   \widehat{G}_{jn}:\mathbb{H}_n\times \mathbb{Y}_n\times \mathbb{Y}_n&\longrightarrow\mathbb{H}_n\times \mathbb{Y}_n\times \mathbb{Y}_n
   \end{aligned}\]
   putting
   \begin{equation}\label{eq:hatFn}
   \widehat{F}_n\left(u,v,w\right)=\left(\begin{array}{c}
      F_n(u,v,w)\\
      -\pi_n^\mathbb{Y}[1_{\cl D}(w+\overline f)]+\pi_n^\mathbb{Y}[\nabla\times (v-\pi_n^{\bb Y}\overline{u})]\\
      -\pi_n^\mathbb{Y}[\nabla\times w]
    \end{array}\right)
    \end{equation}
    and
    \begin{equation}\label{eq:hatGn}
    \widehat{G}_{jn}\left(u,v,w\right)=\left(\begin{array}{c}
      G_{jn}(u)\\
      0\\
      0
    \end{array}\right)
    \end{equation}
 Then system \eqref{eq:SVis 3.29R} takes the form of a stochastic differential equation
 \[\ud X_n=\widehat{F}_n\left(X_n\right)\ud t+\sum_{j=1}^\infty \widehat{G}_{jn}\left(X_n\right)\ud W_j\]
 where $X_n=\left( M_n, B_n,E_n\right)$.
The mapping $\widehat F_n$  defined in \eqref{eq:hatFn} is Lipschitz on balls. For the mapping $\widehat G_{jn}$ defined in \eqref{eq:hatGn}, note that $G_{jn}$ are Lipschitz  and  we have
\[\begin{aligned}
\sum_{j=1}^\infty\left\|G_{jn}(u)-G_{jn}(v)\right\|_\HH\le c_h^2||u-v||_\HH\,,
\end{aligned}\]
where $c_h$ was defined in \eqref{ass_h}.  Hence we have checked the condition of the Lipschitz on balls.

Next let us check that the system also satisfy the one sided linear growth condition.

For $\hat F_n$ are of one sided linear growth, we need to show that there exists $K>0$ such that for all $(u,v,w)\in\HH_n\times\bb Y_n\times\bb Y_n$, we have:
\[\langle (u, v,w),\hat F_n(u,v,w)\rangle_{\HH_n\times \bb Y_n\times\bb Y_n}\le K(1+\|u\|^2_{\HH_n\times \bb Y_n\times\bb Y_n}).\]
$\hat F_n$ has the vector expression \eqref{eq:hatFn}, we only need to show the one sided linear growth property for each component. 
For the first component, we have
\[\langle u, F_n(u,v,w)\rangle_\HH=\frac{1}{2}\llangle u, \sum_{j=1}^\infty G_{jn}'(u)[G_{jn}(u)]\rrangle_\HH.\]
Now let us consider each term in $\llangle u, \sum_{j=1}^\infty G_{jn}'(u)[G_{jn}(u)]\rrangle_\HH$, we will repeatedly using the facts: $\langle a\times b,a\rangle=0$, $\langle a\times b,c\rangle=\langle a,b\times c\rangle$, $a\times b=-b\times a$, definition of the function $\psi$ in Notation 3.2 and equation (2.11) in the assumption part of Problem 2.7. For the first term in the right hand side of (3.9), we have
\[\left|\sum_{j=1}^\infty\langle u, \pi_n[\pi_n(u\times h_j)\times h_j]\rangle_\HH\right|=\left|\sum_{j=1}^\infty\langle u,\pi_n(u\times h_j)\times h_j\rangle_\HH\right|\le \|u\|_\HH^2\sum_{j=1}^\infty\|h_j\|_{L^\infty}^2\le c_h\|u\|_\HH^2.\]
For the second term in the right hand side of (3.9), we have
\[\langle u, \pi_n[\psi(u)[u\times (u\times h_j)]\times h_j]\rangle_\HH=-\langle u\times h_j,\psi(u)[u\times (u\times h_j)]\rangle_\HH=0.\]
For the third term in the right hand side of (3.9), we have
\begin{align*}
	&\left|\sum_{j=1}^\infty\langle u, \pi_n[\pi_n[\psi(u)u\times(u\times h_j)]]\times(u\times h_j)\rangle_\HH\right|\\
	=& \sum_{j=1}^\infty\left|\langle u\times(u\times h_j),\pi_n[\psi(u)u\times(u\times h_j)]\rangle_\HH\right|\\
	=&\sum_{j=1}^\infty\|\psi(u)u\times (u\times h_j)\|_{\HH_n}^2\\
	\le&C\sum_{j=1}^\infty \|h_j\|_\HH^2\le C c_h^2,
\end{align*}
for some constant $C>0$.
Similarly we can check that for the fourth term in the right hand side of (3.9), we have
\[\langle u,\pi_n[\psi(u)u\times [(u\times h_j)\times h_j]]\rangle_\HH=0.\]
And for the fifth  term in the right hand side of (3.9), we have
\[\left|\sum_{j=1}^\infty\langle u,\pi_n[\pi_n[\psi(u)u\times (u\times h_j)]\times(u\times h_j)]\rangle_\HH\right|\le C c_h^2,\]
for some constant $C>0$.
Therefore we have proved that the first component in $\hat F_n$ satisfies the one sided linear growth condition. 

Now let us consider the second component in $\hat F_n$ as in (3.14). For the first term in the second component, we have
\[|\langle v,\pi_n^{\bb Y}[1_{\cl D}(w+\bar f)]\rangle_{\bb Y_n}|\le \|v\|_{\bb Y_n}(\|w\|_{\bb Y_n}+\|f\|_\HH)\le C(1+\|(v,w)\|_{\bb Y_n\times\bb Y_n}^2),\]
for some constant $C>0$.
For the second term in the second component, we have
\[\langle v,\pi_n^{\bb Y}[\nabla\times(v-\pi_n^{\bb Y}\bar u)]\rangle_{\bb Y_n}\le \|v\|_{\bb Y_n}^2+\|v\|_{\bb Y_n}\|u\|_{\bb Y_n}\le 2\|(v,w)\|_{\bb Y_n\times\bb Y_n}^2.\]
Therefore  the second component in $\hat F_n$ also satisfies the one sided linear growth condition.  And it is obvious that the third component in $\hat F_n$ satisfies the one sided linear growth condition. Hence  $\hat F_n$ satisfies the one sided linear growth condition. 

It remains to check if we have 
\[Tr(\sigma(u)\sigma^*(u))\le K(1+\|u\|_{\HH_n}^2),\qquad u\in\HH_n,\]
where $\sigma(u):\HH\longrightarrow \HH_n$ is defined by
\[\sigma(u)(g)=\lmd_1\pi_n(u\times g)-\lmd_2\pi_n\psi(u)u\times(u\times g),\qquad g\in \HH. \]
Hence
\[\sigma^*(u)(h)=-\lmd_1u\times h-\lmd_2\psi(u)u\times(u\times h),\qquad h\in\HH_n.\]
Therefore
\begin{align*}
	&Tr(\sigma(u)\sigma^*(u))=\sum_{k=1}^n\langle \sigma(u)\sigma^*(u)e_k,e_k\rangle_\HH=\sum_{k=1}^n\|\sigma^*(u)e_k\|_\HH^2\\
	\le& 2\sum_{k=1}^n\left(\|u\times e_k\|_\HH^2+\|\psi(u)u\times (u\times e_k)\|_\HH^2\right)\\
	\le&2\|u\|_\HH^2\sum_{k=1}^n\|e_k\|_{L^\infty}^2+C \sum_{k=1}^n\|e_k\|_{L^\infty}^2\\
	\le&K (1+\|u\|_{\HH_n}^2),
\end{align*}
for some constants $C>0$ and $K>0$. So the proof of the one sided linear growth condition is complete.

Therefore the claim follows by standard arguments, see for example Theorem 3.1 in \cite{SA&ZB&JW}.
\end{proof}

\section{A priori estimates}
Next we will get some a priori estimates of the solution to 
equation \eqref{eq:SVis 3.29R}.  
\begin{prop}\label{prop:estimates}
For any $T>0$, $p>0$ and $b>\frac{1}{4}$, there exists a constant $C=C(p,b)>0$ independent of $n$ such that:
\begin{equation}\label{eq:S 3.10R}
  \|M_n\|_{L^\infty(0,T;\mathbb{H})}\leq \|M_0\|_{\mathbb{H}}\,,
\end{equation}
  \begin{equation}\label{eq:SVis 3.32R}
  \mathbb{E}\|B_n-\pi_n^{\bb Y}\overline{M}_n\|_{L^\infty(0,T;\mathbb{L}^{2}(\R^3))}^{p}\leq C\,,
\end{equation}
\begin{equation}\label{eq:SVis 3.33R}
  \mathbb{E}\|E_n\|_{L^\infty(0,T;\mathbb{L}^{2}(\R^3))}^{p}\leq C\,,
\end{equation}
\begin{equation}\label{eq:SVis 3.34R}
  \mathbb{E}\| M_n\|_{L^\infty(0,T;\mathbb{V})}^{p}\leq C\,,
\end{equation}
\begin{equation}\label{eq:SVis 3.35R}
  \mathbb{E}\|M_n\times \rho_n\|_{L^2(0,T;\mathbb{H})}^{p}\leq C\,,
\end{equation}
\begin{equation}\label{eq:SVis 3.321R}
  \mathbb{E}\|B_n\|_{L^\infty(0,T;\mathbb{L}^{2}(\R^3))}^{p}\leq C\,,
\end{equation}
\begin{equation}\label{eq:SLLG 3.13R}
  \mathbb{E}\left(\int_0^T\left\|M_n(t)\times\left(M_n(t)\times \rho_n(t)\right)\right\|^2_{\mathbb{L}^\frac{3}{2}}\ud t\right)^\frac{p}{2}\leq C\,,
\end{equation}
\begin{equation}\label{eq:SLLG 3.14R}
  \mathbb{E}\left\|\pi_n\left[M_n(t)\times\left(M_n(t)\times \rho_n(t)\right)\right]\right\|^2_{L^2(0,T;\mathbb X^{-b})}\leq C\,,
\end{equation}
\begin{equation}\label{eq:SVis 3.38R}
  \mathbb{E}\left\|\frac{\ud E_n}{\ud t}\right\|^p_{L^\infty(0,T;\mathbb{Y}^*)}\leq C\,,
\end{equation}
\begin{equation}\label{eq:SVis 3.39R}
  \mathbb{E}\left\|\frac{\ud B_n}{\ud t}\right\|^p_{L^\infty(0,T;\mathbb{Y}^*)}\leq C\,,
\end{equation}
where $\mathbb X^{-b}$ is the dual space of $X^b=D(A^b)$.
\end{prop}
\begin{proof}[Proof of \eqref{eq:S 3.10R}]
By the It\^{o} formula and straightforward calculus we have
\[\begin{aligned}
   \ud \|M_n\|_{\mathbb{H}}^2&=\sum_{j=1}^\infty2\llangle M_n,G_{nj}(M_n)\rrangle_{\mathbb{H}}\ud W_j+\left[2\llangle M_n,F_n(M_n)\rrangle_{\mathbb{H}}+\sum_{j=1}^\infty\|G_{nj}(M_n)\|_{\mathbb{H}}^2\right]\ud t\\
   &=0
  \end{aligned}\]
 hence
  \[\|M_n(t)\|_{\mathbb{H}}^2=\|M_n(0)\|_{\mathbb{H}}^2=\|\pi_nM_0\|_{\mathbb{H}}^2\leq\|M_0\|_{\mathbb{H}}^2,\qquad t\geq 0.\]
\end{proof}
\begin{proof}[Proof of \eqref{eq:SVis 3.32R}, \eqref{eq:SVis 3.33R}, \eqref{eq:SVis 3.34R}, \eqref{eq:SVis 3.35R}]
By the It$\hat{\textrm{o}}$ formula we get:
\setlength\arraycolsep{2pt}{
  \begin{eqnarray*}
    &&\ud \mathcal{E}_n(M_n(t),B_n(t),E_n(t))\\
    &=&\Bigg[\frac{\partial\mathcal{E}_n}{\partial M_n}\big(F_n(M_n)(t)\big)+\frac{1}{2}\sum_{j=1}^\infty\frac{\partial^2\mathcal{E}_n}{\partial M_n^2}\big(G_{jn}(M_n)(t),G_{jn}(M_n)(t)\big) -\frac{\partial\mathcal{E}_n}{\partial B_n}\big(\pi_n^\mathbb{Y}(\nabla\times E_n(t))\big)\\
    &&+\frac{\partial\mathcal{E}_n}{\partial E_n}\Big(\pi_n^\mathbb{Y}[\nabla\times (B_n(t)-\pi_n^{\bb Y}\overline{M}_n(t))]-\pi_n^\mathbb{Y}[1_{\cl D}(E_n(t)+\overline{f}(t))]\Big)\Bigg]\ud t\\
    &&+\sum_{j=1}^\infty\frac{\partial \mathcal{E}_n}{\partial M_n}\big(G_{jn}(M_n)(t)\big)\ud W_j(t).
  \end{eqnarray*}}
Then by \eqref{eq:dEnM}-\eqref{eq:d2EnM} and \eqref{eq:rhon}, we have
\setlength\arraycolsep{2pt}{
  \begin{eqnarray}
    &&\mathcal{E}_n(t)-\mathcal{E}_n(0)\nonumber\\
    \hspace{1cm}&=&\int_0^t\Bigg\{-\llangle\rho_n(s),F_n(M_n)(s)\rrangle_{\mathbb H}\label{eq:Ent-En0}\\
    &&+\frac{1}{2}\sum_{j=1}^\infty\left\langle\vp''(M_n(s))G_{jn}(M_n)(s),G_{jn}(M_n)(s)\right\rangle_{\mathbb H}\nonumber\\
    &&+\frac{1}{2}\sum_{j=1}^\infty\|\nabla G_{jn}(M_n(s))\|_{\mathbb{H}}^2-\llangle B_n(s)-\pi_n^\mathbb{Y}\overline{M}_n(s),\pi_n^\mathbb{Y}(\nabla\times E_n(s))\rrangle_{\mathbb{L}^2(\R^3)}\nonumber\\
    &&+\llangle E_n(s),\pi_n^\mathbb{Y}[\nabla\times (B_n(s)-\pi_n^{\bb Y}\overline{M}_n(s))]-\pi_n^\mathbb{Y}[1_{\cl D}(E_n(s)+\overline{f}(s))]\rrangle_{\mathbb{L}^2(\R^3)}\Bigg\}\ud s\nonumber\\
    &&-\sum_{j=1}^\infty\int_0^t\llangle \rho_n(s),G_{jn}(M_n)(s)\rrangle\ud W_j(s).\nonumber
  \end{eqnarray}}

Now let's consider each term in the equality \eqref{eq:Ent-En0}.\\
For the term on the left hand side of \eqref{eq:Ent-En0},
\begin{equation}\label{eq:lhsEnt-En0}
\begin{aligned}
\mathcal{E}_n(t)-\mathcal{E}_n(0)&=\int_{\cl D}\vp(M_n(t,x))\ud x-\int_{\cl D}\vp(M_n(0,x))\ud x\\
& + \frac{1}{2} \left\|\nabla M_n(t)\right\|^2_{\mathbb{H}}+\frac{1}{2}\left\|B_n(t)-\pi_n^{\bb Y}\overline{M}_n(t)\right\|^2_{\mathbb{L}^2(\R^3)}+\frac{1}{2}\left\|E_n(t)\right\|^2_{\mathbb{L}^2(\R^3)}\\
 &- \frac{1}{2} \left\|\nabla M_n(0)\right\|^2_{\mathbb{H}}-\frac{1}{2}\left\|B_n(0)-\pi_n^{\bb Y}\overline{M}_n(0)\right\|^2_{\mathbb{L}^2(\R^3)}-\frac{1}{2}\left\|E_n(0)\right\|^2_{\mathbb{L}^2(\R^3)}\,.
  \end{aligned}
  \end{equation}
For the 1st term on the right hand side of \eqref{eq:Ent-En0},  by \eqref{eq:Fn},
 \[\begin{aligned}
 -\llangle \rho_n,F_n(M_n)\rrangle_{\mathbb{H}}=&-\lmd_1\llangle \rho_n,\pi_n[M_n\times \rho_n]\rrangle_{\mathbb{H}}+\lmd_2\llangle\rho_n,\pi_n[M_n\times (M_n\times \rho_n)]\rrangle_{\mathbb{H}}\\
 &-\frac{1}{2}\sum_{j=1}^\infty\lb\rho_n,G_{jn}^\prime\left(M_n\right)\left(G_{jn}\left(M_n\right)\right)\rb\,.
 \end{aligned}\]
  Since
  \[\llangle \rho_n,\pi_n[M_n\times \rho_n]\rrangle_{\mathbb{H}}=\llangle \rho_n,M_n\times \rho_n\rrangle_{\mathbb{H}}=0\,,\]
  and
  \[\llangle\rho_n,\pi_n[M_n\times (M_n\times \rho_n)]\rrangle_{\mathbb{H}}=\llangle\rho_n,M_n\times (M_n\times \rho_n)\rrangle_{\mathbb{H}}=-\|M_n\times \rho_n\|^2_{\mathbb{H}}\,,\]
we find that
 \[-\llangle \rho_n,F_n(M_n)\rrangle_{\mathbb{H}}=-\lmd_2\|M_n\times \rho_n\|^2_{\mathbb{H}}-\frac{1}{2}\sum_{j=1}^\infty \lb\rho_n,G_{jn}^\prime\left( M_n\right)\left[G_{jn}\left( M_n\right)\right]\rb\,.\]

For the 4th and 5th terms on the right hand side of \eqref{eq:Ent-En0},
we notice that
\setlength\arraycolsep{2pt}{
  \begin{eqnarray*}
    &&-\llangle B_n(s)-\pi_n^\mathbb{Y}\overline{M}_n(s),\pi_n^\mathbb{Y}(\nabla\times E_n(s))\rrangle_{\mathbb{L}^2(\R^3)}+\llangle E_n(s),\pi_n^\mathbb{Y}[\nabla\times (B_n(s)-\pi_n^{\bb Y}\overline{M}_n(s))]\rrangle_{\mathbb{L}^2(\R^3)}\\
    &=&  -\llangle B_n(s)-\pi_n^\mathbb{Y}\overline{M}_n(s),\nabla\times E_n(s)\rrangle_{\mathbb{L}^2(\R^3)}+\llangle E_n(s),\nabla\times (B_n(s)-\pi_n^{\bb Y}\overline{M}_n(s))\rrangle_{\mathbb{L}^2(\R^3)}=0.
    \end{eqnarray*}}
Therefore,
\setlength\arraycolsep{2pt}{
  \begin{eqnarray}
    &&-\llangle B_n(s)-\pi_n^\mathbb{Y}\overline{M}_n(s),\pi_n^\mathbb{Y}(\nabla\times E_n(s))\rrangle_{\mathbb{L}^2(\R^3)}\nonumber\\
    &&+\llangle E_n(s),\pi_n^\mathbb{Y}[\nabla\times (B_n(s)-\pi_n^{\bb Y}\overline{M}_n(s))]-\pi_n^\mathbb{Y}[1_{\cl D}(E_n(s)+\overline{f}(s))]\rrangle_{\mathbb{L}^2(\R^3)}\label{eq:rhs45Ent-En0}\\
    &=&-\llangle E_n(s),1_{\cl D}(E_n(s))+\overline{f}(s)\rrangle_{\mathbb{L}^2(\R^3)}=-\|1_{\cl D}E_n\|_{\mathbb{H}}^2-\llangle f,1_{\cl D}E_n\rrangle_{\mathbb{H}}.\nonumber
  \end{eqnarray}}

By \eqref{eq:lhsEnt-En0} and \eqref{eq:rhs45Ent-En0}, equality \eqref{eq:Ent-En0} takes the form
\setlength\arraycolsep{2pt}{
  \begin{eqnarray}
    &&\int_{\cl D}\vp(M_n(t,x))\ud x+ \frac{1}{2} \left\|\nabla M_n(t)\right\|^2_{\mathbb{H}}+\frac{1}{2}\left\|B_n(t)-\pi_n^{\bb Y}\overline{M}_n(t)\right\|^2_{\mathbb{L}^2(\R^3)}+\frac{1}{2}\left\|E_n(t)\right\|^2_{\mathbb{L}^2(\R^3)}\label{eq:SVis 3.31R}\\
    &&+\lmd_2\int_0^t|M_n\times \rho_n|^2\ud s+\frac{1}{2}\sum_{j=1}^\infty\int_0^t\lb\rho_n,G_{jn}^\prime\left(M_n\right)\left[G_{jn}\left(M_n\right)\right]\rb
   \ud s\nonumber\\
    &&-\frac{1}{2}\sum_{j=1}^\infty\int_0^t\left|G_{jn}\left(M_n\right)\right|^2 \ud s-\frac{1}{2}\sum_{j=1}^\infty\int_0^t\left|\nabla G_{jn}\left(M_n\right)\right|^2 \ud s\nonumber\\
    &&-\frac{1}{2}\sum_{j=1}^\infty\int_0^t\lb \vp''(M_n)G_{jn}\left(M_n\right),G_{jn}\left(M_n\right)\rb\ud s\nonumber\\
    &&+\sum_{j=1}^\infty\int_0^t\llangle \rho_n, G_{jn}\left(M_n\right)\rrangle\ud W_j(s)\nonumber\\
    &=&\int_{\cl D}\vp(M_n(0,x))\ud x+ \frac{1}{2} \left\|\nabla M_n(0)\right\|^2_{\mathbb{H}} \nonumber\\
    &&+\frac{1}{2}\left\|B_n(0)-\pi_n^\mathbb{Y}(\overline{M}_n(0))\right\|^2_{\mathbb{L}^2(\R^3)}+\frac{1}{2}\left\|E_n(0)\right\|^2_{\mathbb{L}^2(\R^3)}
    ,\qquad \forall t\in(0,T).\nonumber
    \end{eqnarray}}
Now let us consider some terms in the equality \eqref{eq:SVis 3.31R}.\\
By \eqref{eq:rhon} we have
\setlength\arraycolsep{2pt}{
  \begin{eqnarray*}
   &&\llangle \rho_n, G_{jn}\left(M_n\right)\rrangle\\
    &=&-\llangle \pi_n[\vp'(M_n)],G_{jn}\left(M_n\right)\rrangle+\llangle \Delta M_n, G_{jn}\left(M_n\right)\rrangle\\
    &&+\llangle \pi_n[B_n-\pi_n^{\bb Y}\overline{M}_n],G_{jn}\left(M_n\right)\rrangle\,.
  \end{eqnarray*}}

We also have
 \[\begin{aligned}
  \llangle\Delta M_n,G_{jn}\left(M_n\right)\rrangle&=
  - \llangle \nabla M_n,\nabla G_{jn}\left(M_n\right)\rrangle
    \\
    &=-\lambda_1 \llangle \nabla M_n,\nabla K_{jn}\left(M_n\right)\rrangle \\
    &\leq\|\nabla M_n\|^2_{L^2}\|h_j\|_{\mathbb{L}^\infty(\cl D)}^2+2\|\nabla M_n\|_{L^2}\|M_n\|_{\mathbb{L}^6(D)}\|h_j\|_{\mathbb{L}^\infty(\cl D)}\|\nabla h_j\|_{\mathbb{L}^3(D)}\\
    \leq&\|M_n\|_{\mathbb{V}}^2\|h_j\|_{\mathbb{L}^\infty(\cl D)}^2+2\|M_n\|_{\mathbb{V}}^2\|h_j\|_{\mathbb{L}^\infty(\cl D)}\|\nabla h_j\|_{\mathbb{L}^3(D)}.
  \end{aligned}\]
 Next we have
 \[\left|\llangle \pi_n[B_n-\pi_n^{\bb Y}\overline{M}_n],G_{jn}\left(M_n\right)\rrangle\right|\le
  C\la\left\|h_j\right\|_{\mathbb L^\infty}^2+\left\|1_{\cl D}\left[B_n-\pi_n^{\bb Y}\overline{M}_n\right]\right\|^2_{\mathbb{H}}\ra\]
Since we assume that $\vp'$ is bounded we obtain
\setlength\arraycolsep{2pt}{
  \begin{eqnarray*}
    &&\left|\llangle \pi_n[\vp'(M_n)],G_{jn}\la M_n\ra\rrangle_{\mathbb H}\right|\leq C\|h_j\|_{\mathbb{L}^\infty(\cl D)}\,.
    \end{eqnarray*}}
   
Note that we also have,
\begin{equation*}\label{eq:SVis 3.31R41}
\left|\int_0^t\int_{\cl D}\llangle f, E_n\rrangle\ud x\ud s\right|\leq \frac{1}{2}\int_0^t\int_{\cl D}\left(|f|^2+|E_n|^2\right)\ud x\ud s.
\end{equation*}

Hence by \eqref{eq:S 3.10R} and \eqref{eq:SVis 3.31R} we infer that there exists a constant $C(\alpha,\beta,\cl D)>0$ independent of $n$ such that

\setlength\arraycolsep{2pt}{
  \begin{eqnarray}
    &&\qquad\qquad\frac{1}{2}\|B_n(t)-\pi_n^{\bb Y}\overline{M}_n(t)\|^2_{\mathbb{L}^2(\R^3)}+\frac{1}{2}\|E_n(t)\|^2_{\mathbb{L}^2(\R^3)}\label{eq:goingtoeststocterm}\\
    &&+\lmd_2\int_0^t\|M_n(s)\times \rho_n(s)\|^2_{\mathbb{H}}\ud s+\int_{\cl D}\vp(M_n(t))\ud x+\frac{1}{2} \left\| M_n(t)\right\|_{\mathbb{V}}^2\nonumber\\
    &\leq&\frac{1}{2}\|B_n(0)-\pi_n^{\bb Y}\overline{M}_n(0)\|^2_{\mathbb{L}^2(\R^3)}+\frac{1}{2}\|E_n(0)\|^2_{\mathbb{L}^2(\R^3)}+\frac{1}{2}\int_0^t\|f(s)\|_\HH^2\ud s\nonumber\\
    &&+\int_{\cl D}\vp(M_n(0,x))\ud x +\frac{1}{2}\left\|\nabla M_n(0)\right\|^2_\HH\nonumber\\
    &&+Cc_h\int_0^t\left(\|M_n(s)\|_{\mathbb{V}}^2+\left\|1_{\cl D}\left[B_n(s)-\pi_n^{\bb Y}\overline{M}_n(s)\right]\right\|_{\mathbb{H}}^2\right)\ud s+Cc_ht\nonumber\\
    &&+\left|\sum_{j=1}^\infty\int_0^t \llangle\rho_n, G_{jn}\la M_n\ra\rrangle\ud W_j(s)\right|\nonumber\,.
\end{eqnarray}}
We are going to estimate the stochastic term in the above inequality \eqref{eq:goingtoeststocterm}. We will show first that the infinite sum of stochastic integrals
\begin{equation}\label{eq_ito}
\sum_{j=1}^\infty\int_0^t \llangle\rho_n, G_{jn}\la M_n\ra\rrangle\ud W_j(s)
\end{equation}
 is well defined. We have
\setlength\arraycolsep{2pt}{
  \begin{eqnarray}
  &&\left|\llangle \rho_n,G_{jn}\left(M_n\right)\rrangle_\HH\right|\nonumber\\
  &\le&\left|\llangle -\vp'(M_n)+1_{\cl D}(B_n-\pi_n^{\bb Y}\overline M_n),G_{jn}\left(M_n\right)\rrangle_\HH\right|+\left|\llangle\Delta M_n,G_{jn}\left(M_n\right)\rrangle_\HH\right|\label{eq:eststerm1}\\
  &\le&C\left(\|h_j\|_{\mathbb L^\infty}+\|1_{\cl D}(B_n-\pi_n^{\bb Y}\overline M_n)\|_\HH\|h_j\|_{\mathbb L^\infty}+\|\nabla M_n\|^2_\HH\left( \|\nabla h_j\|_{\bb H}+\|h_j\|_{\bb L^\infty}\right)\right)\nonumber\\
  &\le& C\left( \|\nabla h_j\|_{\bb H}+\|h_j\|_{\bb L^\infty}\right)
  \left(1+\|1_{\cl D}(B_n-\pi_n^{\bb Y}\overline M_n)\|_\HH+\|\nabla M_n\|^2_\HH\right)\nonumber
  \end{eqnarray}}
 
hence
\[\begin{aligned}
\E\sum_{j=1}^\infty&\int_0^t \llangle\rho_n, G_{jn}\la M_n\ra\rrangle^2_{\bb H}\ud s\le c_hC\int_0^t\left(1+\|1_{\cl D}(B_n-\pi_n^{\bb Y}\overline M_n)\|_\HH+\|\nabla M_n\|^2_\HH\right)^2\ud s
\end{aligned}\]
and the It\^o integral \eqref{eq_ito} is a well defined square-integrable martingale.
Secondly, we do some preparation before using the Burkholder-Davis-Gundy inequality on the stochastic term of \eqref{eq:goingtoeststocterm}. Taking supremum over $r\in[0,t]$ on both sides of \eqref{eq:goingtoeststocterm} we obtain
\setlength\arraycolsep{2pt}{
  \begin{eqnarray*}
  &&\frac{1}{2}\sup_{r\in[0,t]}\left(\|B_n(r)-\pi_n^\mathbb{Y}(\overline{M}_n(r))\|^2_{\mathbb{L}^2(\R^3)}+\|E_n(r)\|_{\mathbb{L}^2(\R^3)}^2\right) \\
    &&+\lmd_2\int_0^t\|M_n(s)\times \rho_n(s)\|^2_{\mathbb{H}}\ud s+\sup_{r\in[0,t]}\left(\int_{\cl D}\vp(M_n(r))\ud x+\frac{1}{2} \left\| M_n(r)\right\|_{\mathbb{V}}^2\right)\\
    &\leq&\frac{1}{2}\|B_n(0)-\pi_n^{\bb Y}\overline{M}_n(0)\|^2_{\mathbb{L}^2(\R^3)}+\frac{1}{2}\|E_n(0)\|^2_{\mathbb{L}^2(\R^3)}+\frac{1}{2}\int_0^t\|f(s)\|_\HH^2\ud s\\
    &&+\int_{\cl D}\vp(M_n(0,x))\ud x +\frac{1}{2}\left\|\nabla M_n(0)\right\|^2_\HH\\
    &&+Cc_h\int_0^t\left(\|M_n(s)\|_{\mathbb{V}}^2+\left\|1_{\cl D}\left[B_n(s)-\pi_n^{\bb Y}\overline{M}_n(s)\right]\right\|_{\mathbb{H}}^2\right)\ud s+Cc_ht\\
    &&+\sup_{r\in[0,t]}\left|\sum_{j=1}^\infty\int_0^r\llangle\rho_n,G_{jn}\left(M_n\right)\rrangle_\HH\ud W_j(s)\right|.\nonumber
\end{eqnarray*}}
Let $p\geq 2$. Then using the Jensen inequality we find that for some constant $C$ which includes the initial data, we have
\setlength\arraycolsep{2pt}{
  \begin{eqnarray}
  &&\label{eq:estwW5}\\
    &&\mathbb{E}\Bigg(\sup_{r\in[0,t]}\left(\|[B_n-\pi_n^{\bb Y}\overline{M}_n](r)\|^2_{\mathbb{L}^2(\R^3)} +\|E_n(r)\|^2_{\mathbb{L}^2(\R^3)} +\| M_n(r)\|^2_\mathbb{V}\right)+2\lmd_2\int_0^t\|M_n\times \rho_n\|^2_{\HH}\ud s\Bigg)^p\nonumber\\
    &\leq&Cc_h^pt^{p-1}\bb E\int_0^t\left(\| M_n(s)\|^2_\mathbb{V}+\|[B_n-\pi_n^{\bb Y}\overline{M}_n](s)\|^2_{\mathbb{L}^2(\R^3)}\right)^p\ud s\nonumber\\
    &&+3^{p-1}\mathbb{E}\left(\sup_{r\in[0,t]}\left|\sum_{j=1}^\infty\int_0^r\llangle\rho_n,G_{jn}\left(M_n\right)\rrangle_\HH\ud W_j(s)\right|\right)^p+Cc_h^pt^p.\nonumber
  \end{eqnarray}}

Finally, by the Burkholder-Davis-Gundy inequality, the Jensen's inequality again and \eqref{eq:eststerm1}, there exists a $n$-independent constants $K=K(p)>0$ and $C>0$ such that:

\setlength\arraycolsep{2pt}{
  \begin{eqnarray}
    &&\mathbb{E}\left(\sup_{r\in[0,t]}\left|\sum_{j=1}^\infty\int_0^r\llangle\rho_n,G_{jn}\left(M_n\right)\rrangle_\HH\ud W_j\right|\right)^p\nonumber\\
    &\leq&K\mathbb{E}\left|\sum_{j=1}^\infty\int_0^t\left\langle \rho_n,G_{jn}\left(M_n\right)\right\rangle_{\mathbb H}^2\ud s\right|^\frac{p}{2}\nonumber\\
    &\leq&Ct^{\frac{p}{2}-1}\mathbb{E}\int_0^t\sup_{r\in[0,s]}\left(\|M_n(r)\|_\V^{2p}+\|1_{\cl D}(B_n-\pi_n^{\bb Y}\overline M_n)(r)\|_\HH^{2p}\right)\ud s \label{eq:estoW5}
  \end{eqnarray}}

Hence by \eqref{eq:estwW5} and \eqref{eq:estoW5} there exists $C>0$ independent of $n$ such that,

\setlength\arraycolsep{2pt}{
  \begin{eqnarray*}
    &&\mathbb{E}\sup_{r\in[0,t]}\Bigg(\|[B_n-\pi_n^{\bb Y}\overline{M}_n](r)\|^2_{\mathbb{L}^2(\R^3)} +\|E_n(r)\|^2_{\mathbb{L}^2(\R^3)} +\| M_n(r)\|^2_\mathbb{V}+\int_0^t\|M_n\times \rho_n\|^2_{\HH}\ud\tau\Bigg)^p\\
    &&\leq C(t^{p-1}+t^{\frac{p}{2}-1})\int_0^t\bb E\sup_{r\in[0,s]}\left(\|M_n(r)\|_{\mathbb{V}}^{2p}+\|[B_n-\pi_n^{\bb Y}\overline{M}_n](r)\|_{\mathbb{L}^2(\R^3)}^{2p}\right)\ud s+Ct^p
\end{eqnarray*}}

Hence by the Gronwall inequality, with $C=CT^p e^{C(T^p+T^\frac{p}{2})}$, we get the following four a'priori estimates,
 \begin{equation*}
  \mathbb{E}\|B_n-\pi_n^{\bb Y}\overline{M}_n\|_{L^\infty(0,T;\mathbb{L}^{2}(\R^3))}^{2p}\leq C,
\end{equation*}

\begin{equation*}
  \mathbb{E}\|E_n\|_{L^\infty(0,T;\mathbb{L}^{2}(\R^3))}^{2p}\leq C,
\end{equation*}

\begin{equation*}
  \mathbb{E}\|  M_n\|_{L^\infty(0,T;\mathbb{V})}^{2p}\leq C,
\end{equation*}

\begin{equation*}
  \mathbb{E}\|M_n\times \rho_n\|_{L^2(0,T;\mathbb{H})}^{2p}\leq C.
\end{equation*}

And since $L^{2p}(\Omega)\hookrightarrow L^{q}(\Omega)$ continuously for all $q<2p$, these four inequalities imply the inequalities:
\eqref{eq:SVis 3.32R}, \eqref{eq:SVis 3.33R}, \eqref{eq:SVis 3.34R}, \eqref{eq:SVis 3.35R} for all $p>0$.
\end{proof}

We continue with the proof of Proposition \ref{prop:estimates}.
\begin{proof}[Proof of \eqref{eq:SVis 3.321R}]
For fixed $p\geq 1$, we have
  \setlength\arraycolsep{2pt}{
  \begin{eqnarray*}
    &&\mathbb{E}\|B_n\|_{L^\infty(0,T;\mathbb{L}^2(\R^3))}^{2p}\leq 2^p\left(\mathbb{E}\|[B_n-\pi_n^{\bb Y}\overline{M}_n]\|_{L^\infty(0,T;\mathbb{L}^2(\R^3))}^{2p}+\mathbb{E}\|M_n\|_{L^\infty(0,T;\mathbb{H})}^{2p}\right).
  \end{eqnarray*}}
  By the a'priori estimates \eqref{eq:SVis 3.32R} and \eqref{eq:SVis 3.34R}, there exists some $C>0$ independent of $n$ such that
  \[\mathbb{E}\|B_n\|_{L^\infty(0,T;\mathbb{L}^2(\R^3))}^{2p}\leq C.\]
  Together with the fact $L^{2p}(\Omega)\hookrightarrow L^{q}(\Omega)$ continuously for all $q<2p$, we complete the proof of \eqref{eq:SVis 3.321R}.
\end{proof}

\begin{proof}[Proof of \eqref{eq:SLLG 3.13R}]
By the Soblev imbedding theorem, there is a constant $C$ such that
  \[\big\|M_n\big\|_{\mathbb{L}^6}\leq C\big\|M_n\big\|_{\V},\]
  therefore by the H$\ddot{\textrm{o}}$lder inequality, we have
  \setlength\arraycolsep{2pt}{
  \begin{eqnarray*}
    &&\Big\|M_n(t)\times\big(M_n(t)\times\rho_n(t)\big)\Big\|_{\mathbb{L}^\frac{3}{2}}\leq \big\|M_n(t)\big\|_{\mathbb{L}^6}\big\|M_n(t)\times\rho_n(t)\big\|_{\mathbb{L}^2}\leq C\big\|M_n(t)\big\|_{\V}\big\|M_n(t)\times\rho_n(t)\big\|_{\mathbb{L}^2}.
  \end{eqnarray*}}
  Hence, by the Cauchy-Schwartz inequality,
  \setlength\arraycolsep{2pt}{
  \begin{eqnarray*}
    &&\mathbb{E}\left[\left(\int_0^T\Big\|M_n(t)\times\big(M_n(t)\times\rho_n(t)\big)\Big\|^2_{\mathbb{L}^{\frac{3}{2}}}\ud t\right)^\frac{p}{2}\right]\\
    &\leq&C^p\mathbb{E}\left[\sup_{r\in [0,T]}\big\|M_n(r)\big\|^p_{\V}\left(\int_0^T\big\|M_n(t)\times\rho_n(t)\big\|^2_{\mathbb{L}^2}\ud t\right)^\frac{p}{2}\right]\\
    &\leq&C^p\left(\mathbb{E}\left[\sup_{t\in[0,T]}\big\|M_n(t)\big\|_{\V}^{2p}\right]\right)^\frac{1}{2} \left(\mathbb{E}\left[\left(\int_0^T\big\|M_n(t)\times\rho_n(t)\big\|^2_{\mathbb{L}^2}\ud t\right)^p\right]\right)^\frac{1}{2}.
  \end{eqnarray*}}

  Then by \eqref{eq:SVis 3.34R} and \eqref{eq:SVis 3.35R}, we get \eqref{eq:SLLG 3.13R}.
\end{proof}

\begin{proof}[Proof of \eqref{eq:SLLG 3.14R}]
Since $\|\cdot\|_{X^b}=\|A_1^b\cdot\|_{\HH}=\|\cdot\|_{\HH^{2b}}$,
$X^b\hookrightarrow \mathbb{L}^3$ compactly for $b>\frac{1}{4}$. Hence $\mathbb{L}^\frac{3}{2}$ is compactly embedded in $\mathbb X^{-b}$. Thus there is a constant $C$ independent of $n$ such that

\setlength\arraycolsep{2pt}{
  \begin{eqnarray*}
    &&\hspace{-1.5cm}{\mathbb{E}\int_0^T\left\|\pi_n\left[M_n(t)\times\left(M_n(t)\times \rho_n(t)\right)\right]\right\|^2_{\mathbb X^{-b}}\ud t}\\
    &\leq&\mathbb{E}\int_0^T\left\|\left[M_n(t)\times\left(M_n(t)\times \rho_n(t)\right)\right]\right\|^2_{\mathbb X^{-b}}\ud t\\
    &&\hspace{1.5cm}{\leq C\mathbb{E}\int_0^T\left\|\left[M_n(t)\times\left(M_n(t)\times \rho_n(t)\right)\right]\right\|^2_{\mathbb{L}^\frac{3}{2}}\ud t.}
  \end{eqnarray*}}
  Then by \eqref{eq:SLLG 3.13R}, we get \eqref{eq:SLLG 3.14R}.
\end{proof}

\begin{proof}[Proof of \eqref{eq:SVis 3.38R} and \eqref{eq:SVis 3.39R}]
By the second equation in \eqref{eq:SVis 3.29R}, we have
\setlength\arraycolsep{2pt}{
  \begin{eqnarray*}
   && \mathbb{E}\left\|\frac{\ud E_n}{\ud t}\right\|^p_{L^\infty(0,T;\mathbb{Y}^*)}=\mathbb{E}\|\pi_n^\mathbb{Y}(\nabla\times [B_n-\pi_n^{\bb Y}\overline{M}_n])-\pi_n^\mathbb{Y}\big[1_{D}(E_n+\overline f)\big]\|^p_{L^\infty(0,T;\mathbb{Y}^*)}\\
    &\leq&C_p\mathbb{E}\sup_{t\in(0,T)}\|\nabla\times [B_n(t)-\pi_n^{\bb Y}\overline{M}_n(t)]\|^p_{\mathbb{Y}^*}+C_p\mathbb{E}\sup_{t\in(0,T)}\|1_{D}(E_n(t)+f(t))\|^p_{\mathbb{Y}^*}\\
    &\leq&C_p\mathbb{E}\sup_{t\in(0,T)}\sup_{y\neq 0}\left(\left|\frac{\llangle B_n(t)-\pi_n^{\bb Y}\overline{M}_n(t),\nabla\times y\rrangle_{\mathbb{L}^2(\R^3)}}{\|y\|_\mathbb{Y}}\right|^p+\left|\frac{\llangle 1_{D}(E_n+\overline f),y\rrangle_{\mathbb{L}^2(\R^3)}}{\|y\|_\mathbb{Y}}\right|^p\right)\\
    &\leq&C_p\mathbb{E}\|[B_n-\pi_n^{\bb Y}\overline{M}_n]\|^p_{L^\infty(0,T;\mathbb{L}^2(\R^3))} +C_p\mathbb{E}\|E_n\|^p_{L^\infty(0,T;\mathbb{L}^2(\R^3))}+C_p\|f\|^p_{\mathbb{L}^2(0,T;\HH)}.
  \end{eqnarray*}}
Hence, since $f\in L^2(0,T;\mathbb{H})$,  by \eqref{eq:SVis 3.32R} and \eqref{eq:SVis 3.33R}, we get \eqref{eq:SVis 3.38R} and similarly \eqref{eq:SVis 3.39R}.
\end{proof}
After so many pages of long calculation, the proof of Proposition \ref{prop:estimates} has been finished. Next let us consider the estimate of the stochastic term in the finite dimensional system \eqref{eq:SVis 3.29R}.

\begin{lem}\label{lem:StermWanorm}
  For $a\in [0,\frac{1}{2})$ and  $p\geq 2$, there exists a constant $C\geq0$ such that for all $n\in\mathbb{N}$,
  \begin{equation}\label{eq:SLLG 4.1R}
 \mathbb{E}\Bigg\|\sum_{j=1}^\infty\int_0^\cdot G_{jn}\left(M_n(s)\right)\ud W_j(s)\Bigg\|^p_{W^{a,p}(0,T;\mathbb{H})}\leq C\,.
   \end{equation}
\end{lem}
To prove Lemma \ref{lem:StermWanorm}, we will use the Lemma 2.1 from Flandoli and Gatarek's paper \cite{Flandoli}, which is stated as Lemma \ref{lem:FG2.1} in the Appendix.
\begin{proof}[Proof of Lemma \ref{lem:StermWanorm}]
  By Lemma \ref{lem:FG2.1}, there exists constant $C_1>0$, such that
 \[\begin{aligned}
  \mathbb{E}\Bigg\|\sum_{j=1}^\infty\int_0^\cdot G_{jn}\left(M_n\right)\ud W_j(s)\Bigg\|^p_{W^{a,p}(0,T;\mathbb{H})}
    &\leq C_1\mathbb{E}\left(\int_0^T\sum_{j=1}^\infty\left\|G_{jn}\left(M_n\right)\right\|^2_{\mathbb{H}}\ud t\right)^\frac{p}{2}\\
    &\leq 2^{p-1}C_1\la \sum_{j=1}^\infty\|h_j\|^2_{\mathbb{L}^\infty}\ra^{p/2}\mathbb{E}\int_0^T\la 1+\|M_n\|^p_{\mathbb{H}}\ra\ud t\\
    &\le C,
  \end{aligned}\]
  where the last inequality followed by \eqref{eq:SVis 3.34R}. This completes the proof of the estimate \eqref{eq:SLLG 4.1R}.
\end{proof}

\begin{rem}
  From now on we will always assume $a\in [0,\frac{1}{2})$, $b>\frac{1}{4}$ and $p\geq 2$.
\end{rem}

\begin{lem}
  For $a\in [0,\frac{1}{2})$, $b>\frac{1}{4}$, $p\geq 2$, there exists $C>0$ such that for all $n\in\mathbb{N}$,
  \begin{equation}\label{eq:SLLG 4.2R}
    \mathbb{E}\left\|M_n\right\|^2_{W^{a,p}(0,T;\mathbb X^{-b})}\leq C.
  \end{equation}
\end{lem}
\begin{proof}
  By \eqref{eq:SVis 3.29R},
\setlength\arraycolsep{2pt}{
  \begin{eqnarray*}
  &&\mathbb{E} \left\|M_n\right\|^2_{W^{a,p}(0,T;\mathbb X^{-b})}\\
  &=& \mathbb{E}\Bigg\|\int_0^\cdot\pi_n\Bigg\{\lmd_1 M_n\times \rho_n-\lmd_2 M_n\times(M_n\times \rho_n)+\frac{1}{2}\sum_{j=1}^\infty G_{jn}^\prime\la M_n\ra\left[G_{jn}\la M_n\ra\right]\Bigg\}\ud s\nonumber\\
  &&+\sum_{j=1}^\infty\int_0^\cdot G_{jn}\la M_n\ra\ud W_j\Bigg\|^2_{W^{a,p}(0,T;\mathbb X^{-b})}.
   \end{eqnarray*} }

By our assumption, $a\in [0,\frac{1}{2})$, so $H^1(0,T;\mathbb X^{-b})\hookrightarrow W^{a,p}(0,T;\mathbb X^{-b})$ compactly for all $p>0$. And since $\mathbb{H}\hookrightarrow\mathbb X^{-b}$ continuously, there is a constant $C$ independent of $n$ such that
\[\begin{aligned}
  \mathbb{E} \left\|M_n\right\|^2_{W^{a,p}(0,T;\mathbb X^{-b})}\leq &C\,\mathbb{E}\Bigg\|\int_0^t\pi_n\Bigg\{\lmd_1 M_n\times \rho_n+\frac{1}{2}\sum_{j=1}^\infty G_{jn}^\prime\la M_n\ra\left[G_{jn}\la M_n\ra\right]\Bigg\}\ud s\Bigg\|^2_{H^1(0,T;\mathbb{H})}\\
 &+C\,\mathbb{E}\left\|\int_0^t\lmd_2\pi_n\left[ M_n\times(M_n\times \rho_n)\right]\ud s\right\|^2_{H^1(0,T;\mathbb X^{-b})}\\
 &+C\,\mathbb{E}\Bigg\|\sum_{j=1}^\infty\int_0^t G_{jn}\la M_n\ra\ud W_j\Bigg\|^2_{W^{a,p}(0,T;\mathbb{H})}
  \end{aligned}\]
To prove \eqref{eq:SLLG 4.2R}, it is enough to consider each term on the right hand side of the above inequality.
By \eqref{eq:SVis 3.35R}, \eqref{eq:SLLG 3.14R} and \eqref{eq:SLLG 4.1R}, we can conclude \eqref{eq:SLLG 4.2R}.
\end{proof}
\section{Tightness results}
In this subsection we will use the a'priori estimates \eqref{eq:S 3.10R}-\eqref{eq:SVis 3.39R} to show that the laws $\{\mathcal{L}(M_n,B_n,E_n):n\in\mathbb{N}\}$ are tight on a suitable path space. Then we will use Skorohod's theorem to obtain another probability space and an almost surely convergent sequence defined on this space whose limit is a weak martingale solution of the Problem \ref{S-LLG}.

To do so, we will need some  compact embedding results from Flandoli and  Gatarek's paper \cite{Flandoli}, which stated in the Appendix as Lemma \ref{lem:compact emb 5}-Lemma \ref{lem:compact emb 6}.
We will also need the following Lemma about tightness.
\begin{lem}\label{lem:tight}
  Let $X,Y$ be separable Banach  spaces and $(\Omega,\mathcal{F},\mathbb{P})$ be a probability space, we assume that $i:X\hookrightarrow Y$ is compact and the random variables $u_n:\Omega\longrightarrow X$, $n\in\mathbb{N}$, satisfy the following condition: there is a constant $C>0$, such that
  \[\mathbb{E}\big(\|u_n\|_X\big)\leq C,\qquad n\in \bb N.\]
 Then the family of laws $\big\{\mathcal{L}(i\circ u_n)\big\}_{n\in\mathbb{N}}$ is tight on $Y$.
\end{lem}
\begin{proof}
	Let us arbitrarily fix $\eps>0$.
	
	Since $\mathbb{E}\big(\|u_n\|_X\big)\leq C$ for all $n$, we have 
	\[\bb P\left( ||u_n||_X\le \frac{C}{\eps}\right)\ge 1-\eps,\qquad n\in\bb N.\]
	On the other hand, since the embedding map $i:X\hookrightarrow Y$ is compact, 
	\[K_\eps:=\overline{\left\{i\circ u_n(\omega):||u_n(\omega)||_X\le \frac{C}{\eps},\omega\in\Omega\right\}}\]
	is compact in $Y$. And we have
	\[\cl{L}(i\circ u_n)(K_\eps)\ge\bb P\left(||u_n||_X\le \frac{C}{\eps} \right)\ge 1-\eps,\qquad n\in\bb N.\]
	Therefore $\big\{\mathcal{L}(i\circ u_n)\big\}_{n\in\mathbb{N}}$ is tight on $Y$ and the proof is complete.
\end{proof}

Now let's state and prove our tightness results.
\begin{lem}\label{lem:LMntight}
  For any  $p\geq 2$, $q\in [2,6)$ and $b>\frac{1}{4}$ the set of laws $\{\mathcal{L}(M_n):n\in\mathbb{N}\}$ on the Banach space
  \[L^p(0,T;\mathbb{L}^q)\cap C([0,T];\mathbb X^{-b})\]
  is tight.
\end{lem}
\begin{proof}
Firstly, let us prove $\{\mathcal{L}(M_n):n\in\mathbb{N}\}$ is tight on $L^p(0,T;\mathbb{L}^q)$ for all $p\ge 2$ and $q\in [2,6)$. To this end fix $p\ge 2$, $a\in (0,\frac{1}{2})$, $b>\frac{1}{4}$ and $q\in [2,6)$. Since $q<6$ and the embedding $\V=D(A^{\frac{1}{2}})\hookrightarrow X^\gamma=D(A^\gamma)$ is compact for $\gamma<\frac{1}{2}$, we can choose $\gamma\in (\frac{3}{4}-\frac{3}{2q},\frac{1}{2})$, such that, Lemma \ref{lem:compact emb 5} yields a compact embedding
\[L^p(0,T;\V)\cap W^{a,p}(0,T;\mathbb X^{-b})\hookrightarrow L^p(0,T;X^\gamma)\,.\]

Therefore
\setlength\arraycolsep{2pt}{
  \begin{eqnarray*}
&&\mathbb{P}\Big(\|M_n\|_{L^p(0,T;\V)\cap W^{a,p}(0,T;\mathbb X^{-b})}>r\Big)=\mathbb{P}\Big(\|M_n\|_{L^p(0,T;\V)}+\|M_n\|_{W^{a,p}(0,T;\mathbb X^{-b})}>r\Big)\\
&\leq& \mathbb{P}\left(\|M_n\|_{L^p(0,T;\V)}>\frac{r}{2}\right)+\mathbb{P} \left(\|M_n\|_{W^{a,p}(0,T;\mathbb X^{-b})}>\frac{r}{2}\right)\leq\frac{4}{r^2}\mathbb{E}\left(\|M_n\|^2_{L^p(0,T;\V)}+\|M_n\|^2_{W^{a,p}(0,T;\mathbb X^{-b})}\right).
\end{eqnarray*}}

Let $X_T:=L^p(0,T;\V)\cap W^{a,p}(0,T;\mathbb X^{-b})$. By estimates \eqref{eq:SLLG 4.2R} and \eqref{eq:SVis 3.34R}, there exists a constant $C$, such that
\[\mathbb{P}\big(\|M_n\|_{X_T}>r\big)\leq\frac{C}{r^2},\qquad\forall r,n.\]
hence
\[\mathbb{E}\big(\|M_n\|_{X_T}\big)\leq 1+\int_1^\infty \frac{C}{r^2}\ud r=1+C,\quad\forall n\in\mathbb{N}.\]
By Lemma \ref{lem:tight}, the family of laws $\big\{\mathcal{L}(M_n):n\in\mathbb{N}\big\}$ is tight on  $L^p(0,T;X^\gamma)$.
For $\gamma>\frac{3}{4}-\frac{3}{2q}$, we have $X^\gamma=\mathbb{H}^{2\gamma}(D)\hookrightarrow \mathbb{L}^q$ continuously.
Hence $L^p(0,T;X^\gamma)\hookrightarrow L^p(0,T;\bb L^q)$ continuously and $\big\{\mathcal{L}(M_n):n\in\mathbb{N}\big\}$ is also tight on $L^p(0,T;\bb L^q)$.

Secondly, we prove the laws $\{\mathcal{L}(M_n):n\in\mathbb{N}\}$ are tight on $C([0,T];\mathbb X^{-b})$ for all $b>\frac{1}{4}$. To do this, we fix some $b>\frac{1}{4}$ and choose $b'\in (\frac{1}{4},b)$.  Since $b'<b$, by Lemma \ref{lem:compact emb 6} we have $W^{a,p}(0,T;X^{-b'})\hookrightarrow C([0,T];\mathbb X^{-b})$ compactly for $a\in(0,\frac{1}{2})$ and $p>2$ satisfying $a>\frac{1}{p}$. Therefore by estimate \eqref{eq:SLLG 4.2R} and Lemma \ref{lem:tight} again, we conclude that $\big\{\mathcal{L}(M_n):n\in\mathbb{N}\big\}$ is tight on $C([0,T];\mathbb X^{-b})$.

Therefore $\big\{\mathcal{L}(M_n):n\in\mathbb{N}\big\}$ is tight on $L^p(0,T;\bb L^q)\cap C([0,T];\mathbb X^{-b})$ and the proof is complete.
\end{proof}
To prove the tightness results about $\{\mathcal{L}(E_n)\}$ and $\{\mathcal{L}(B_n)\}$, we need the version (\cite{ZB&EM}, Def. 3.7) of Aldous Condition (\cite{Aldous}), i.e. Definition \ref{def:Aldous} and the tightness criterion Lemma \ref{lem:tightcri}.

\begin{lem}\label{lem:LEBntight}
  The sets of laws $\{\mathcal{L}(E_n)\}$ and $\{\mathcal{L}(B_n)\}$ are tight on the space $L_w^2(0,T;\mathbb{L}^2(\R^3))$.
\end{lem}
\begin{proof}
  Here we will only prove the result about $\{\mathcal{L}(E_n)\}$, the proof about $\{\mathcal{L}(B_n)\}$ is exactly the same.\\
  In order to use Lemma \ref{lem:tightcri}, let us set $H=\mathbb{L}^2(\R^3)$ and choose an auxiliary Hilbert space $U$ such that the embedding $U\hookrightarrow \mathbb{Y}$ is compact.
  (The existence of such $U$ is actually worth to justify, we put it in the Lemma \ref{lem:UcomebY} in the Appendix.)  
  
  Since the embedding $\mathbb{Y}\hookrightarrow \mathbb{L}^2(\R^3)$ is bounded, the embedding $U\hookrightarrow \mathbb{L}^2(\R^3)$ is also compact.

  Next we will check the condition (a) and (b) in Lemma \ref{lem:tightcri}.

  Firstly, let us observe that by estimate \eqref{eq:SVis 3.33R}, condition (a) of the Lemma \ref{lem:tightcri} is satisfied.

  Secondly, we will check the Aldous condition (Definition \ref{def:Aldous}) in the space $U^*$. To this end, fix $\eps>0$, $\eta>0$ and a sequence of $\bb F$-stopping times $\{\tau_n\}$.
  The embedding $\mathbb{Y}^*\hookrightarrow U^*$ is compact so it is bounded and thus there exists a constant $C_1>0$ such that $\|\cdot\|_{\mathbb{Y}^*}\geq C_1 \|\cdot\|_{U^*}$. Hence together with the Chebyshev inequality and estimate \eqref{eq:SVis 3.38R}, we have
  \setlength\arraycolsep{2pt}{
  \begin{eqnarray*}
    &&\mathbb{P}\left(\|E_n(\tau_n+\theta)-E_n(\tau_n)\|_{U^*}\geq\eta\right)\leq\mathbb{P}\left(\|E_n(\tau_n+\theta)-E_n(\tau_n)\|_{\mathbb{Y}^*}\geq C_1\eta\right)\\
    &\leq&\frac{1}{C_1\eta}\mathbb{E}\left(\|E_n(\tau_n+\theta)-E_n(\tau_n)\|_{\mathbb{Y}^*}\right)\leq\frac{1}{C_1\eta}\mathbb{E}\int_{\tau_n}^{\tau_n+\theta}\left\|\frac{\ud E_n(s)}{\ud s}\right\|_{\mathbb{Y}^*}\ud s\leq\frac{C\theta}{C_1\eta},\quad \theta>0.
  \end{eqnarray*}}
  Hence for $\delta\leq \frac{C_1}{C}\eps\eta$, we have
  \[\sup_{n\in\mathbb{N}}\sup_{0\leq\theta\leq \delta}\mathbb{P}\left(\|E_n(\tau_n+\theta)-E_n(\tau_n)\|_{U^*}\geq\eta\right)\leq\eps.\]
  The Aldous condition \eqref{eq:Ald con} has been verified.\\
  Therefore by Lemma \ref{lem:tightcri}, the laws $\{\mathcal{L}(E_n)\}$ are tight on $C([0,T];U^*)\cap L_w^2(0,T;{H})$ and the lemma follows.
\end{proof}

By the previous tightness results and the Prokhorov Theorem, we have the following result of weakly convergence of laws.
\begin{prop}\label{prop:conv to mu}
  There exists a subsequence $\{(M_{n_k},B_{n_k},E_{n_k})\}$ of $\{(M_n,B_n,E_n)\}$, such that \dela{for each $j=1,2,\ldots$,} the laws $\mathcal{L}(M_{n_k},B_{n_k},E_{n_k},W_h)$ \dela{(where $W_j$ is the $1$-dimensional Brownian motion as in the setting of Problem \ref{S-LLG})} converge weakly to a probability measure $\mu$ on $[L^p(0,T;\mathbb{L}^q)\cap C([0,T];\mathbb X^{-b})]\times L_w^2(0,T;\mathbb{L}^2(\R^3))\times L_w^2(0,T;\mathbb{L}^2(\R^3))\times C([0,T];\HH),$
  where $p\in [2,\infty),$ $q\in [2,6)$ and $b>\frac{1}{4}$.
\end{prop}

\section{Construction of new probability space and processes}
Now we are going to use the Skorokhod Theorem to construct our new probability space and processes as the weak solution of Problem \ref{S-LLG}.
\begin{lem}\label{prop:tilMn}
For  $p\in [2,\infty),$ $q\in [2,6)$, $b>\frac{1}{4}$,
  there exists a probability space $(\til{\Omega},\til{\mathcal{F}}, \til{\mathbb{P}})$ and a sequence
  $\{(\til{M}_n,\til{E}_{n},\til{B}_{n},\til{W}_{hn})\}$, 

  of
  \setlength\arraycolsep{2pt}{
  \begin{eqnarray*}
    &&[L^p(0,T;\mathbb{L}^q(D))\cap C([0,T];\mathbb X^{-b})]\times L^2_w(0,T;\mathbb{L}^2(\R^3))\times L^2_w(0,T;\mathbb{L}^2(\R^3))\times C([0,T];\HH)
  \end{eqnarray*}}
  -valued random variables defined on $(\til{\Omega},\til{\mathcal{F}},\til{\mathbb{P}})$ such that
  \begin{trivlist}
    \item[(a)] On the product space
    \setlength\arraycolsep{2pt}{
  \begin{eqnarray*}
   &&[L^p(0,T;\mathbb{L}^q(D))\cap C([0,T];\mathbb X^{-b})]\times L^2_w(0,T;\mathbb{L}^2(\R^3))\times L^2_w(0,T;\mathbb{L}^2(\R^3))\times C([0,T];\HH)
     \end{eqnarray*}}
$$\mathcal{L}(M_n,E_{n},B_{n},W_h)=\mathcal{L}(\til{M}_n,\til{E}_{n},\til{B}_{n},\til{W}_{hn}),\qquad\forall n\in\mathbb{N}$$
    \item[(b)] There exists a random variable $(\til{M},\til{E},\til{B},\til{W}_h):$
    \setlength\arraycolsep{2pt}{
  \begin{eqnarray*}
  (\til{\Omega},\til{\mathcal{F}},\til{\mathbb{P}})&\longrightarrow& [L^p(0,T;\mathbb{L}^q)\cap C([0,T];\mathbb X^{-b})]\times L^2_w(0,T;\mathbb{L}^2(\R^3))\\
    &&\times L^2_w(0,T;\mathbb{L}^2(\R^3))\times C([0,T];\HH),
    \end{eqnarray*}}
    such that
    \begin{trivlist}
    \item[(i)] On the product space
    \setlength\arraycolsep{2pt}{
  \begin{eqnarray*}
    &&[L^p(0,T;\mathbb{L}^q(D))\cap C([0,T];\mathbb X^{-b})]\times L^2_w(0,T;\mathbb{L}^2(\R^3))\times L^2_w(0,T;\mathbb{L}^2(\R^3))\times C([0,T];\HH)
    \end{eqnarray*}}
    \[\mathcal{L}(\til{M},\til{E},\til{B},\til{W}_h)=\mu,\]
    where $\mu$ is same as in Proposition \ref{prop:conv to mu}. Moreover, the following convergence results hold $\til{\mathbb{P}}$-a.s. as $n\rightarrow \infty$,
    \item[(ii)] $\til{M}_n\longrightarrow \til{M}$ in $L^p(0,T;\mathbb{L}^q(D))\cap C([0,T];\mathbb X^{-b})$,
    \item[(iii)]
        $\til{E}_{n}\longrightarrow \til{E}$ in $L^2_w(0,T;\mathbb{L}^2(\R^3))$,
    \item[(iv)]    $\til{B}_{n}\longrightarrow \til{B}$ in $L^2_w(0,T;\mathbb{L}^2(\R^3))$.
    \item[(v)] $\til{W}_{hn}\longrightarrow \til{W}_h$ in $C([0,T];\HH)$.
    \end{trivlist}
  \end{trivlist}
\end{lem}

To prove Lemma \ref{prop:tilMn}, we need the standard Skorohod theorem \cite[Thm 11.7.2]{Dudley2002} for separable metric spaces  as well as the following Jakubowski's version of  Skorohod theorem:
\begin{lem}[\cite{Jakubowski, ZB&MO}, Thm A.1]\label{lem:SkorohodJaku}
  Let $X$ be a topological space such that there exists a sequence of continuous functions $f_m:X\longrightarrow\R$, $m=1,2,\ldots$ which separates points of $X$. Let us denote by $\mathscr{S}$ the $\sigma$-algebra generated by the maps $\{f_m\}$. Then
  \begin{trivlist}
    \item[(i)] every compact subset of $X$ is metrizable,
    \item[(ii)] if $\{\mu_m\}$ is a tight sequence  of probability measures on $(X,\mathscr{S})$, then there exists a subsequence $(m_k)$, a probability space $(\Omega,\mathcal{F},\mathbb{P})=([0,1],\mathcal{B}([0,1]),Leb.)$ with $X$-valued random variables $\xi_k$, $\xi$ such that $\mu_{m_k}$ is the law of $\xi_k$ and $\xi_k$ converges to $\xi$ almost surely. Moreover, the law of  $\xi$ is a Radon measure.
  \end{trivlist}
\end{lem}

\begin{proof}[Proof of Lemma \ref{prop:tilMn}]
 $L^p(0,T;\mathbb{L}^q(D))\cap C([0,T];\mathbb X^{-b})$ and $C([0,T];\HH)$ are separable metric spaces, so by the Skorohod Theorem for the separable metric spaces \cite[Thm 11.7.2]{Dudley2002}, there exists a probability space $(\Omega_1, \cl F_1, \bb P_1)$ and corresponding random variables take values in
 $$\left[L^p(0,T;\mathbb{L}^q(D))\cap C([0,T];\mathbb X^{-b})\right]\times C([0,T];\HH)$$
 such that the related results in Lemma \ref{prop:tilMn} hold.

 To prove the results relative to the space $L^2_w(0,T;\mathbb{L}^2)$ in Lemma \ref{prop:tilMn}, we will use the Proposition \ref{prop:conv to mu} and Lemma \ref{lem:SkorohodJaku}. Let us recall, that for any separable Hilbert space $H$, the elements of $H^*$ separate points in $H$, so the countable dense subset of $H^*$ also separate points in $H$. We also have that the Borel $\sigma$-algebras generated from strong and weak topologies are coincide, so $\{\mu_m\}$ is tight on $(H,\mathscr{S})$ equivalent to $\{\mu_m\}$ is tight on $(H,\mathscr{B}(H))$.

   Then the product probability space and the corresponding random variables of above two related results are the aims we are looking for and this completes the proof of Lemma \ref{prop:tilMn}.
\end{proof}

\begin{rem}
	We set $\til{\bb{F}}$ to be the filtration generated from $\til W_h$ and $\til M$ 
	 and $\til M_n$ for all $n$. i.e.
	\[\til{\bb{F}}=\sigma\left\{\til W_h(t),\til M(t),\til M_n(t):t\in[0,T],n=1,2,3,\ldots.\right\}\]
	 So now we have a filtered new probability space $(\til{\Omega},\til{\mathcal{F}},\til{\bb F}, \til{\mathbb{P}})$.
	 
		Since $(M_n,W_h)$ and $(\til M_n, \til W_h)$ have same distribution, and the increment $W_h(t)-W_h(s)$ is independent of $\sigma\{M_n(r):r\le s\}$,  we can see that $\til{W}_h(t)-\til{W}_h(s)$ is independent of $\til{\bb F}_s$ for all $t>s$.
\end{rem}

\begin{rem}\label{rem:tilWh}
	As stated in Lemma \ref{prop:tilMn}, $\til{W}_h$ has same distribution on $C([0,T];\HH)$ as $W_h$. Hence it can be proved that $\left\{i_t\circ\til{W}_h\right\}_{t\ge0}$ is also a $\til{\bb F}$-Wiener process on $\HH$
	 (See Lemma \ref{lem:distrWP} in the Appendix), where
	\[i_t:C([0,T];\HH)\ni f\mapsto f(t)\in \HH.\]
	And for convenience, we will use $\til W_h(t)$ to denote $i_t\circ \til W_h$.
	
	Since we assumed that $\{h_j\}_j$ is an ONB of $\HH$ as in Remark \ref{rem:QW}, $\til{W}_h$ has the following representation: 
	\[\til W_h(t)=\sum_{j=1}^\infty \til W_j(t)h_j,\qquad t\in[0,T],\]
	where
	\[\til W_j(t):=\frac{\llangle \til W_h(t),h_j\rrangle_\HH}{\|h_j\|^2_\HH}.\]
	It can be shown that $\til W_j(t)$ is $N(0,t)$ distributed for each $j$ and form a Gaussian family and so are independent for all $j=1,2,\cdots$.
	
	The map:
	\[t\mapsto \frac{\llangle \til W_h(t),h_j\rrangle_\HH}{\|h_j\|^2_\HH}=\til W_j(t)\]
	is continuous almost surely. So $\til W_j$ has continuous trajectory almost surely for every $j$.
	
	The independence of increments of $\til W_j$ for each $j$ follows from the independence of increments of $\til W_h$. Therefore $\til W_j, j=1,2,\cdots$ are independent $1$-dimensional $\til{\bb F}$-Brownian motions. 
	
	Similarly, we also have 
	\[\til W_{hn}(t)=\sum_{j=1}^\infty \til W_{jn}(t)h_j,\qquad t\in[0,T],\]
	for some independent $1$-dimensional $\til{\bb F}$-Brownian motions $\til W_{jn}, j=1,2,\cdots$.
\end{rem}

Let $\til{M}_n$, $\til{B}_n$ and $\til{E}_n$ be as in Lemma \ref{prop:tilMn}, we have the following result:
\begin{prop}\label{prop:LLCn}
The processes $\til{M}_n$, $\til{B}_n$ and $\til{E}_n$ have the following properties:
  \begin{trivlist}
    \item[(i)] $\til{M}_n\in C([0,T];\mathbb{H}_n)$ almost surely and $\mathcal{L}(\til{M}_n)=\mathcal{L}(M_n)$ on $C([0,T]; \mathbb{H}_n)$;
    \item[(ii)] $\til{E}_n\in C([0,T];\mathbb{Y}_n)$ almost surely and $\mathcal{L}(\til{E}_n)=\mathcal{L}(E_n)$ on $C([0,T]; \mathbb{Y}_n)$;
    \item[(iii)] $\til{B}_n\in C([0,T];\mathbb{Y}_n)$ almost surely and $\mathcal{L}(\til{B}_n)=\mathcal{L}(B_n)$ on $C([0,T]; \mathbb{Y}_n)$.
  \end{trivlist}
\end{prop}

\begin{proof}[Proof of Proposition \ref{prop:LLCn}]
  \begin{trivlist}
    \item[(i)] Since $C([0,T];\mathbb{H}_n)\subset L^p(0,T;\mathbb{L}^1(\cl D))\cap C([0,T];\mathbb X^{-b})$, if we take $\vp$ to be the embedding map, then by the Kuratowski Theorem \ref{thm:kuratowski}, the Borel sets in $C([0,T];\mathbb{H}_n)$ are the Borel sets in $L^p(0,T;\mathbb{L}^1(\cl D))\cap C([0,T];\mathbb X^{-b})$.  On the other hand, by Lemma \ref{prop:tilMn}, $\mathcal{L}(\til{M}_n)=\mathcal{L}(M_n)$ on $L^p(0,T;\mathbb{L}^1(\cl D))\cap C([0,T];\mathbb X^{-b})$, so $\mathcal{L}(\til{M}_n)=\mathcal{L}(M_n)$ on $C([0,T]; \mathbb{H}_n)$.
        By Lemma \ref{lem:ndimsolp}, $\mathbb{P}\{M_n\in C([0,T]; \mathbb{H}_n)\}=1$. Hence $\til{\mathbb{P}}\{\til{M}_n\in C([0,T]; \mathbb{H}_n)\}=1$.
    \item[(ii)] By the Kuratowski Theorem \ref{thm:kuratowski}, the Borel sets in $C([0,T];\mathbb{Y}_n)$ are Borel sets in $L^2(0,T;\mathbb{Y}_n)$. And since $L^2(0,T;\mathbb{Y}_n)$ is closed in $L^2(0,T;\mathbb{L}^2(\R^3))$, by the Lemma \ref{lem:wcloeqclo}, $L^2(0,T;\mathbb{Y}_n)$ is also closed in the space $L^2_w(0,T;\mathbb{L}^2(\R^3))$. Hence the Borel sets in $L^2(0,T;\mathbb{Y}_n)$ are also Borel sets in $L^2_w(0,T;\mathbb{L}^2(\R^3))$. Therefore the Borel sets in $C([0,T];\mathbb{Y}_n)$ are the Borel sets in $L^2_w(0,T;\mathbb{L}^2(\R^3))$.  By Lemma \ref{prop:tilMn}, $\mathcal{L}(\til{E}_n)=\mathcal{L}(E_n)$ on $L^2_w(0,T;\mathbb{L}^2(\R^3))$, so $\mathcal{L}(\til{E}_n)=\mathcal{L}(E_n)$ on $C([0,T]; \mathbb{Y}_n)$.
         By Lemma \ref{lem:ndimsolp},  $\mathbb{P}\{E_n\in C([0,T]; \mathbb{Y}_n)\}=1$. Hence $\til{\mathbb{P}}\{\til{E}_n\in C([0,T]; \mathbb{Y}_n)\}=1$.
    \item[(iii)] Exactly the same as the proof of (ii).
  \end{trivlist}
  This complete the proof of Proposition \ref{prop:LLCn}.
\end{proof}
}

The next result shows that the sequence $(\til{M}_n,\til{B}_n,\til{E}_n)$ satisfies the  similar a'priori estimates as $(M_n,B_n,E_n)$ in Proposition \ref{prop:estimates}.
\begin{prop}\label{prop_tilde}
Let us define
\[\til{\rho}_n:=\pi_n\big[-\vp'(\til{M}_n)+1_{\cl D}(\til{B}_n-\pi_n^\mathbb{Y}\overline{\til{M}}_n)\big]+\Delta \til{M}_n,\]
Then for all $p\geq0$, $b>\frac{1}{4}$, there exists $C>0$ such that for all $n\in\mathbb{N}$,
\begin{equation}\label{eq:S 3.10R'}
\|\til{M}_n\|_{L^\infty(0,T;\mathbb{H})}\leq \|M_0\|_{\mathbb{H}},\qquad \til{\mathbb{P}}-a.s.,
\end{equation}

  \begin{equation}\label{eq:SVis 3.32R'}
  \til{\mathbb{E}}\|\til{B}_n-\pi_n^\mathbb{Y}\overline{\til{M}}_n\|_{L^\infty(0,T;\mathbb{L}^{2}(\R^3))}^{p}\leq C,
\end{equation}

\begin{equation}\label{eq:SVis 3.33R'}
  \til{\mathbb{E}}\|\til{E}_n\|_{L^\infty(0,T;\mathbb{L}^{2}(\R^3))}^{p}\leq C,
\end{equation}

\begin{equation}\label{eq:SVis 3.34R'}
  \til{\mathbb{E}}\| \til{M}_n\|_{L^\infty(0,T;\mathbb{V})}^{p}\leq C,
\end{equation}

\begin{equation}\label{eq:SVis 3.35R'}
  \til{\mathbb{E}}\|\til{M}_n\times \til{\rho}_n\|_{L^2(0,T;\mathbb{H})}^{p}\leq C,
\end{equation}

\begin{equation}\label{eq:SVis 3.321R'}
  \til{\mathbb{E}}\|\til{B}_n\|_{L^\infty(0,T;\mathbb{L}^{2}(\R^3))}^{p}\leq C,
\end{equation}

\begin{equation}\label{eq:SLLG 3.13R'}
  \til{\mathbb{E}}\left(\int_0^T\left\|\til{M}_n(t)\times\left(\til{M}_n(t)\times \til{\rho}_n(t)\right)\right\|^2_{\mathbb{L}^\frac{3}{2}(\cl D)}\ud t\right)^\frac{p}{2}\leq C,
\end{equation}

\begin{equation}\label{eq:SLLG 3.14R'}
  \til{\mathbb{E}}\int_0^T\left\|\pi_n\left[\til{M}_n(t)\times\left(\til{M}_n(t)\times \til{\rho}_n(t)\right)\right]\right\|^2_{\mathbb X^{-b}}\ud t\leq C,
\end{equation}

\begin{equation}\label{eq:SVis 3.38R'}
  \til{\mathbb{E}}\left\|\frac{\ud \til{E}_n}{\ud t}\right\|^p_{L^\infty(0,T;\mathbb{Y}^*)}\leq C.
\end{equation}

\begin{equation}\label{eq:SVis 3.39R'}
  \til{\mathbb{E}}\left\|\frac{\ud \til{B}_n}{\ud t}\right\|^p_{L^\infty(0,T;\mathbb{Y}^*)}\leq C.
\end{equation}
\end{prop}
\begin{proof}
 
 Note that all the maps, $\pi_n\circ \vp'$, $\pi_n\circ 1_{\cl D}\circ \pi_n^{\bb Y}$, $\Delta$, all the cross products, the norms etc, are measurable maps on the corresponding spaces.
  Therefore by the Proposition \ref{prop:LLCn} and Proposition \ref{prop:estimates}, we get the estimates \eqref{eq:S 3.10R'}-\eqref{eq:SVis 3.39R'}.
\end{proof}


\begin{rem}
  From now on we will set $p=q=4$  and $b=\frac{1}{2}$ in Lemma \ref{prop:tilMn}. That will be enough to show the existence of the solution of the Problem \ref{S-LLG}.
\end{rem}

\begin{prop}
  As defined in Lemma \ref{prop:tilMn}, the $\til{M}$ satisfies the following estimates:
  \begin{equation}\label{eq:S 4.10}
    \esup_{t\in [0,T]}\|\til{M}(t)\|_{\mathbb{H}}\leq \|M_0\|_{\mathbb{H}}, \qquad \til{\mathbb{P}}-a.s.,
  \end{equation}
  And for some constant $C>0$,
  \begin{equation}\label{eq:S 4.11}
    \esup_{t\in [0,T]}\|\til{M}(t)\|_{\mathbb X^{-b}}\leq C\|M_0\|_{\mathbb{H}}, \qquad \til{\mathbb{P}}-a.s..
  \end{equation}
\end{prop}
\begin{proof}
The results follows from Lemma \ref{prop:tilMn} (b) (ii),  and $\mathbb{L}^4\hookrightarrow \mathbb{H}\hookrightarrow\mathbb X^{-b}$ continuously and the estimate  \eqref{eq:S 3.10R'}.
\end{proof}

We continue to investigate properties of the process $\til{M}$, the next result and it's proof are related to the estimate \eqref{eq:SVis 3.34R'}.
\begin{lem}\label{thm:MregV}
  The process $\til{M}$ defined in Lemma \ref{prop:tilMn}  satisfies the following estimate:
  \begin{equation}\label{eq:S 4.12}
    \til{\mathbb{E}}\big[\esup_{t\in [0,T]}\|\til{M}(t)\|_{\mathbb{V}}^{2r}\big]<\infty,\quad r\geq 0.
  \end{equation}
\end{lem}
\begin{proof}
Since $L^{2r}(\til{\Omega};L^\infty(0,T;\mathbb{V}))$ is isomorphic to $\big[L^\frac{2r}{2r-1}(\til{\Omega};L^1(0,T;\bb X^{-\frac{1}{2}}))\big]^\ast$, by the estimate \eqref{eq:SVis 3.34R'} and the Banach-Alaoglu Theorem we infer that the sequence $\{\til{M}_n\}$ contains a subsequence  (which will be denoted in the same way as the full sequence) and there exists an element
$v \in L^{2r}(\til{\Omega};L^\infty(0,T;\mathbb{V}))$ such that $\til{M}_n \to v$ weakly$^\ast$ in $L^{2r}(\til{\Omega};L^\infty(0,T;\mathbb{V}))$ as $n\rightarrow\infty$. So it remains to show that $\til M=v$.

We have
\[
\lim_{n\rightarrow\infty}\lb \til{M}_n , \varphi \rb = \lb v,  \varphi \rb, \qquad  \varphi \in L^\frac{2r}{2r-1}(\til{\Omega};(L^1(0,T;\bb X^{-\frac{1}{2}}))),
\]
which means that
\[\lim_{n\rightarrow\infty} \int_{\til{\Omega}} \int_0^T \lb \til{M}_n(t,\omega) , \varphi(t,\omega) \rb \ud t \ud\til{\mathbb{P}}(\omega) =
\int_{\til{\Omega}} \int_0^T \lb v(t,\omega) , \varphi(t,\omega) \rb \ud t\ud\til{\mathbb{P}}(\omega).
\]

On the other hand, if we fix $\varphi \in L^4(\til{\Omega};L^\frac{4}{3}(0,T;\mathbb{L}^\frac{4}{3}))$, we have
 \setlength\arraycolsep{2pt}{
 \begin{eqnarray*}
   &&\sup_n\int_{\til{\Omega}}\left|\int_0^T{}_{\mathbb{L}^4}\llangle \til{M}_n(t),\varphi(t)\rrangle_{\mathbb{L}^\frac{4}{3}}\ud t\right|^2\ud \til{\mathbb{P}}(\omega)\leq\sup_n\int_{\til{\Omega}}\left|\int_0^T\|\til{M}_n\|_{\mathbb{L}^4}\|\varphi\|_{\mathbb{L}^\frac{4}{3}}\ud t\right|^2\ud \til{\mathbb{P}}(\omega)\\
   &\leq&\sup_n\int_{\til{\Omega}}\|\til{M}_n\|_{L^\infty(0,T;\mathbb{L}^4)}^2\|\varphi\|_{L^1(0,T;\mathbb{L}^\frac{4}{3})}^2\ud \til{\mathbb{P}}(\omega)\leq \sup_n\|\til{M}_n\|^2_{L^4(\til{\Omega};L^\infty(0,T;\mathbb{L}^4))}\|\varphi\|^2_{L^4(\til{\Omega};L^1(0,T;\mathbb{L}^\frac{4}{3}))}<\infty.
 \end{eqnarray*}}

So the sequence
$\int_0^T{}_{\mathbb{L}^4}\llangle \til{M}_n(t),\varphi(t)\rrangle_{\mathbb{L}^\frac{4}{3}}\ud t$ is uniformly integrable on $\til{\Omega}$. Moreover, by the $\til{\mathbb{P}}$ almost surely convergence of $\til{M}_n$ to $\til{M}$ in $L^4(0,T;\mathbb{L}^4)$, we
infer that $\int_0^T{}_{\mathbb{L}^4}\llangle \til{M}_n(t),\varphi(t)\rrangle_{\mathbb{L}^\frac{4}{3}}\ud t$ converges to $\int_0^T{}_{\mathbb{L}^4}\llangle \til{M}(t),\varphi(t)\rrangle_{\mathbb{L}^\frac{4}{3}}\ud t$ $\til{\mathbb{P}}$ almost surely. Thus for $n\rightarrow\infty$, we have
\[ \int_{\til{\Omega}} \int_0^T{}_{\mathbb{L}^4}\llangle \til{M}_n(t,\omega),\varphi(t,\omega)\rrangle_{\mathbb{L}^\frac{4}{3}}\ud t \ud\til{\mathbb{P}}(\omega) \to
\int_{\til{\Omega}} \int_0^T{}_{\mathbb{L}^4}\llangle \til{M}(t,\omega),\varphi(t,\omega)\rrangle_{\mathbb{L}^\frac{4}{3}}\ud t \ud\til{\mathbb{P}}(\omega).
\]
Hence we deduce that
\[
\int_{\til{\Omega}} \int_0^T{}_{\mathbb{L}^4}\llangle v(t,\omega),\varphi(t,\omega)\rrangle_{\mathbb{L}^\frac{4}{3}}\ud t\ud\til{\mathbb{P}}(\omega)
=\int_{\til{\Omega}} \int_0^T{}_{\mathbb{L}^4}\llangle \til{M}(t,\omega),\varphi(t,\omega)\rrangle_{\mathbb{L}^\frac{4}{3}}\ud t\ud\til{\mathbb{P}}(\omega)
\]

By the arbitrariness of $\varphi$ and densness of $L^4(\til{\Omega};L^\frac{4}{3}(0,T;\mathbb{L}^\frac{4}{3}))$ in $L^\frac{2r}{2r-1}(\til{\Omega};L^1(0,T;\bb X^{-\frac{1}{2}}))$, we infer that $\til{M}=v$ and since $v$ satisfies \eqref{eq:S 4.12} we infer that $\til{M}$ also satisfies \eqref{eq:S 4.12}. In this way the proof of  \eqref{eq:S 4.12} is complete.
\end{proof}

We also investigate the following property of $\til{B}$ and $\til{E}$.
\begin{lem}\label{lem:BEreg}
	The processes $\til{B},\til{E}$ defined in Lemma \ref{prop:tilMn} have following regularities:
  \begin{equation}\label{eq:tilBL2}
    \til{\mathbb{E}}\int_0^T\|\til{B}(t)\|_{\mathbb{L}^2(\R^3)}^2\ud t<\infty.
  \end{equation}
  \begin{equation}\label{eq:tilEL2}
    \til{\mathbb{E}}\int_0^T\|\til{E}(t)\|_{\mathbb{L}^2(\R^3)}^2\ud t<\infty.
  \end{equation}
\end{lem}
\begin{proof}
The proof of \eqref{eq:tilBL2} and \eqref{eq:tilEL2} are similar to the proof of  \eqref{eq:S 4.12}.
\end{proof}

Next we will strengthen part (ii) and (iv) of Lemma \ref{prop:tilMn} (b) about the convergence.
\begin{prop}
  \begin{equation}\label{eq:EMn'L4}
    \lim_{n\rightarrow\infty}\til{\mathbb{E}}\int_0^T\|\til{M}_n(t)-\til{M}(t)\|_{\mathbb{L}^4}^4\ud t=0.
  \end{equation}
\end{prop}
\begin{proof}[Proof of \eqref{eq:EMn'L4}]
By the Lemma \ref{prop:tilMn}, $\til{M}_n(t)\longrightarrow \til{M}(t)$ in $L^4(0,T;\mathbb{L}^4)\cap C([0,T];\mathbb X^{-b})$ $\til{\mathbb{P}}$-almost surely, $\til{M}_n(t)\longrightarrow \til{M}(t)$ in $L^4(0,T;\mathbb{L}^4)$ $\til{\mathbb{P}}$-almost surely, that is
\[\lim_{n\rightarrow\infty}\int_0^T\|\til{M}_n(t)-\til{M}(t)\|_{\mathbb{L}^4}^4\ud t=0,\qquad \til{\mathbb{P}}-a.s.,\]
and by \eqref{eq:SVis 3.34R'} and \eqref{eq:S 4.12},
\[\sup_n\til{\mathbb{E}}\left(\int_0^T\|\til{M}_n(t)-\til{M}(t)\|_{\mathbb{L}^4}^4\ud t\right)^2\leq 2^7\sup_n\left(\|\til{M}_n\|^8_{L^4(0,T;\mathbb{L}^4)}+\|\til{M}\|^8_{L^4(0,T;\mathbb{L}^4)}\right)<\infty,\]
hence,
\[\lim_{n\rightarrow\infty}\til{\mathbb{E}}\int_0^T\|\til{M}_n(t)-\til{M}(t)\|_{\mathbb{L}^4}^4\ud t=\til{\mathbb{E}}\left(\lim_{n\rightarrow\infty}\int_0^T\|\til{M}_n(t)-\til{M}(t)\|_{\mathbb{L}^4}^4\ud t\right)=0.\]
This completes the proof.
\end{proof}

\begin{cor}\label{cor:Mn'tilMae}
	
	There is a subsequence  $\{\til M_{n_k}\}\subset\{\til M_n\}$ such that 
   $\til{M}_{n_k}\longrightarrow \til{M}$ almost everywhere in $\til{\Omega}\times[0,T]\times \cl D$ as $k\rightarrow\infty$.

\end{cor}

\begin{rem}
	For convenience, we will still denote the $\{\til M_{n_k}\}$ as in Corollary \ref{cor:Mn'tilMae} by $\{\til M_{n}\}$ in the rest part of this paper.
\end{rem}

\begin{prop}
  \begin{equation}\label{eq:EphiMn'L2}
    \lim_{n\rightarrow\infty}\til{\mathbb{E}}\int_0^T\|\pi_n\vp'(\til{M}_n(s))-\vp'(\til{M}(s))\|_{\mathbb{H}}^2\ud s=0.
  \end{equation}
\end{prop}
\begin{proof}[Proof of \eqref{eq:EphiMn'L2}]
  By Corollary \ref{cor:Mn'tilMae}, $\til{M}_n\longrightarrow \til{M}$ almost everywhere in $\til{\Omega}\times[0,T]\times D$. And since $\vp'$ is continuous, $$\lim_{n\rightarrow\infty}\left|\vp'(\til{M}_n)- \vp'(\til{M})\right|^2=0,$$
  almost everywhere in $\til{\Omega}\times[0,T]\times D$. Moreover, $\vp'$ is bounded, so there exists some constant $C>0$ such that $|\vp'(x)|\leq C$ for all $x\in \R^3$. Therefore for almost every $(\omega,s)\in \til{\Omega}\times [0,T]$,
  \[\int_{\cl D}\left|\vp'(\til{M}_n(\omega, s,x))- \vp'(\til{M}(\omega,s,x))\right|^4\ud x\leq 16C^4m(D)<\infty.\]
  Hence $\left|\vp'(\til{M}_n(\omega, s))- \vp'(\til{M}(\omega,s))\right|^2$ is uniformly integrable on $D$, so
  \[\lim_{n\rightarrow\infty} \left\|\vp'(\til{M}_n(\omega,s))-\vp'(\til{M}(\omega,s))\right\|^2_{\mathbb{H}}=0,\quad \til{\Omega}\times [0,T]-a.e..\]
  Therefore for almost every $(\omega,s)\in \til{\Omega}\times [0,T]$,
  \setlength\arraycolsep{2pt}{
  \begin{eqnarray*}
    &&\left\|\pi_n\vp'(\til{M}_n(\omega,s))-\vp'(\til{M}(\omega,s))\right\|^2_{\mathbb{H}}\\
    &\leq&2 \left\|\vp'(\til{M}_n(\omega,s))-\vp'(\til{M}(\omega,s))\right\|^2_{\mathbb{H}}+2 \left\|\pi_n\vp'(\til{M}(\omega,s))-\vp'(\til{M}(\omega,s))\right\|^2_{\mathbb{H}}\rightarrow 0.
  \end{eqnarray*}}
  Moreover since
  \setlength\arraycolsep{2pt}{
  \begin{eqnarray*}
    &&\til{\mathbb{E}}\int_0^T\left\|\pi_n\vp'(\til{M}_n(\omega,s))-\vp'(\til{M}(\omega,s))\right\|^4_{\mathbb{H}}\ud s\leq 16TC^4m(D)<\infty,
  \end{eqnarray*}}
  $\left\|\pi_n\vp'(\til{M}_n)-\vp'(\til{M})\right\|^2_{\mathbb{H}}$ is uniformly integrable on $\til{\Omega}\times [0,T]$. Hence
 \[ \lim_{n\rightarrow\infty}\til{\mathbb{E}}\int_0^T\|\pi_n\vp'(\til{M}_n(s))-\vp'(\til{M}(s))\|_{\mathbb{L}^2}^2\ud s=0.\]
 This completes the proof of \eqref{eq:EphiMn'L2}.
\end{proof}

\begin{prop}
For any $u\in L^2(0,T;\mathbb{H})$, we have
\begin{equation}\label{eq:EBn'L2w}
  \lim_{n\rightarrow\infty}\til{\mathbb{E}}\left|\int_0^T\llangle  u(s),\pi_n1_{\cl D}(\til{B}_n-\til{B})(s)\rrangle_{\mathbb{H}}\ud s\right|=0.
\end{equation}
\end{prop}
\begin{proof}[Proof of \eqref{eq:EBn'L2w}]
  By (iv) of Lemma \ref{prop:tilMn}, we have
  \[\lim_{n\rightarrow\infty}\left|\int_0^T\llangle  u(s),\pi_n1_{\cl D}(\til{B}_n-\til{B})(s)\rrangle_{\mathbb{H}}\ud s\right|=0,\quad \til{\mathbb{P}}-a.s..\]
  Moreover, by \eqref{eq:SVis 3.321R'} and \eqref{eq:tilBL2} we have
  \setlength\arraycolsep{2pt}{
  \begin{eqnarray*}
    &&\til{\mathbb{E}}\left|\int_0^T\llangle  u(s),\pi_n1_{\cl D}(\til{B}_n-\til{B})(s)\rrangle_{\mathbb{H}}\ud s\right|^2\\
    &\leq&2\|u\|^2_{L^2(0,T;\mathbb{H})}\til{\mathbb{E}}\left(\int_0^T\|1_{\cl D}\til{B}_n(s)\|_{\mathbb{H}}^2\ud s+\int_0^T\|1_{\cl D}\til{B}(s)\|_{\mathbb{H}}^2\ud s\right)<\infty.
  \end{eqnarray*}}
  Hence $\left|\int_0^T\llangle  u(s),\pi_n1_{\cl D}(\til{B}_n-\til{B})(s)\rrangle_{\mathbb{H}}\ud s\right|$ is uniformly integrable on $\til{\Omega}$, so
  \[\lim_{n\rightarrow\infty}\til{\mathbb{E}}\left|\int_0^T\llangle  u(s),\pi_n1_{\cl D}(\til{B}_n-\til{B})(s)\rrangle_{\mathbb{H}}\ud s\right|=0.\]
  The proof of \eqref{eq:EBn'L2w} has been complete.
\end{proof}

\begin{prop}
\begin{equation}\label{eq:4.17phi}
  \nabla_i \til{M}_n\longrightarrow \nabla_i \til{M}\textrm{ weakly in }L^2(\til{\Omega};L^2(0,T;\mathbb{L}^2)),\;i=1,2,3.
\end{equation}
\end{prop}
\begin{proof}
Let us fix $\varphi\in L^2(\til{\Omega};L^2(0,T;\mathbb{V}))$, by \eqref{eq:EMn'L4} $\til{M}_n\longrightarrow \til{M}$ in $L^2(\til{\Omega};L^2(0,T;\mathbb{H}))$, so we have:
\[\til{\mathbb{E}}\int_0^T\llangle \til{M},\nabla_i \varphi\rrangle_{\mathbb{H}}\ud x=\lim_{n\rightarrow\infty}\til{\mathbb{E}}\int_0^T\llangle \til{M}_n,\nabla_i \varphi\rrangle_{\mathbb{H}}\ud x=-\lim_{n\rightarrow\infty}\til{\mathbb{E}}\int_0^T\llangle \nabla_i \til{M}_n, \varphi\rrangle_{\mathbb{H}}\ud x.\]
  By the estimate \eqref{eq:SVis 3.34R'}, $\{\til{M}_n\}_{n=1}^\infty$ is bounded in $L^2(\til{\Omega};L^2(0,T;\mathbb{V}))$, so the limit of the right hand side of above equation exists. Hence the result follows.
\end{proof}

Next we will define $\til{M}\times\til{\rho}$ and show that the limits of the sequences  $\{\til{M}_n\times\til{\rho}_n\}_n$, $\{\til{M}_n\times(\til{M}_n\times\til{\rho}_n)\}_n$ and $\big\{\pi_n\big(\til{M}_n\times(\til{M}_n\times\til{\rho}_n)\big)\big\}_n$ are actually
$\til{M}\times\til{\rho}$, $\til{M}\times(\til{M}\times\til{\rho})$ and $\til{M}\times(\til{M}\times\til{\rho})$.

\begin{prop}
  For $p\geq 1$ and $b>\frac{1}{4}$, there exist $Z_1\in L^{2p}(\til{\Omega};L^2(0,T;\mathbb{H}))$, $Z_2\in L^2(\til{\Omega};L^2(0,T;\mathbb{L}^\frac{3}{2}))$ and $Z_3\in L^2(\til{\Omega};L^2(0,T;\mathbb X^{-b}))$, such that
  \begin{equation}\label{eq:Z1w}
    \til{M}_n\times\til{\rho}_n\longrightarrow Z_1\quad\textrm{weakly in } L^{2p}(\til{\Omega};L^2(0,T;\mathbb{H})),
  \end{equation}
  \begin{equation}\label{eq:4.15M}
    \til{M}_n\times(\til{M}_n\times\til{\rho}_n)\longrightarrow Z_2 \quad\textrm{weakly in } L^2(\til{\Omega};L^2(0,T;\mathbb{L}^\frac{3}{2})),
  \end{equation}
  \begin{equation}\label{eq:4.16M}
    \pi_n(\til{M}_n\times(\til{M}_n\times\til{\rho}_n))\longrightarrow Z_3\quad\textrm{weakly in } L^2(\til{\Omega};L^2(0,T;\mathbb X^{-b})).
  \end{equation}
\end{prop}
\begin{proof}
The spaces $L^{2p}(\til{\Omega};L^2(0,T;\mathbb{H}))$, $L^2(\til{\Omega};L^2(0,T;\mathbb{L}^\frac{3}{2}))$ and $L^2(\til{\Omega};L^2(0,T;\mathbb X^{-b}))$ are reflexive. Then by equations \eqref{eq:SVis 3.35R'}, \eqref{eq:SLLG 3.13R'}, \eqref{eq:SLLG 3.14R'} and
  by the Banach-Alaoglu Theorem, we get equations \eqref{eq:Z1w}, \eqref{eq:4.15M} and \eqref{eq:4.16M}.
\end{proof}

\begin{prop}
  \[Z_2=Z_3 \textrm{ in the space } L^2(\til{\Omega};L^2(0,T;\mathbb X^{-b})).\]
\end{prop}
\begin{proof}
  Notice that $(L^\frac{3}{2})^*=L^3$, and  $X^b=\mathbb{H}^{2b}$. $X^b\subset L^3$ for $b>\frac{1}{4}$, hence $L^\frac{3}{2}\subset\mathbb X^{-b}$, so $$L^2(\til{\Omega};L^2(0,T;\mathbb{L}^\frac{3}{2}))\subset L^2(\til{\Omega};L^2(0,T;\mathbb X^{-b})).$$
  Therefore $Z_2\in L^2(\til{\Omega};L^2(0,T;\mathbb X^{-b}))$ as well as $Z_3$.\\
  Since by definition $X^{b}=D(A^b)$ and $A$ is self-adjoint, we can define
  \[X^{b}_n:=\left\{\pi_nx=\sum_{j=1}^n\llangle x,e_j\rrangle_{\mathbb{H}}e_j: x\in \HH,\,\sum_{j=1}^\infty\lmd_j^{2b}\llangle x,e_j\rrangle_{\mathbb{H}}^2<\infty\right\}.\]
  Then $X^{b}=\bigcup_{n=1}^\infty X_n^{b}$, $L^2(\til{\Omega};L^2(0,T;X^{b}))=\bigcup_{n=1}^\infty L^2(\til{\Omega};L^2(0,T;X^{b}_n))$.

  Firstly, we prove the result for each $u_n\in L^2(\til{\Omega};L^2(0,T;X^{b}_n))$. To do this, let us fix $n$ and take  $u_n\in L^2(\til{\Omega};L^2(0,T;X^{b}_n))$, then  for any $m\geq n$, we have
  \setlength\arraycolsep{2pt}{
  \begin{eqnarray*}
  &&{}_{L^2(\til{\Omega};L^2(0,T;\mathbb X^{-b}))}\llangle \pi_m(\til{M}_m\times(\til{M}_m\times\til{\rho}_m)),u_n\rrangle_{L^2(\til{\Omega};L^2(0,T;X^{b}))}\\
  &=&\til{\mathbb{E}}\int_0^T \,_{\mathbb X^{-b}}\llangle \pi_m(\til{M}_m(t)\times(\til{M}_m(t)\times\til{\rho}_m(t))),u_n(t)\rrangle_{X^b}\ud t\\
  &=&\til{\mathbb{E}}\int_0^T \,\llangle \pi_m(\til{M}_m(t)\times(\til{M}_m(t)\times\til{\rho}_m(t))),u_n(t)\rrangle_{\mathbb{H}}\ud t\\
  &=&\til{\mathbb{E}}\int_0^T \,\llangle \til{M}_m(t)\times(\til{M}_m(t)\times\til{\rho}_m(t)),u_n(t)\rrangle_{\mathbb{H}}\ud t\\
  &=&\til{\mathbb{E}}\int_0^T \,_{\mathbb X^{-b}}\llangle \til{M}_m(t)\times(\til{M}_m(t)\times\til{\rho}_m(t)),u_n(t)\rrangle_{X^b}\ud t\\
  &=&_{L^2(\til{\Omega};L^2(0,T;\mathbb X^{-b}))}\llangle \til{M}_m\times(\til{M}_m\times\til{\rho}_m),u_n\rrangle_{L^2(\til{\Omega};L^2(0,T;X^{b}))}.
  \end{eqnarray*}}
  Hence let $m\rightarrow\infty$ on both sides of above equality, we have
  \[_{L^2(\til{\Omega};L^2(0,T;\mathbb X^{-b}))}\llangle Z_3,u_n\rrangle_{L^2(\til{\Omega};L^2(0,T;X^{b}))}=_{L^2(\til{\Omega};L^2(0,T;\mathbb X^{-b}))}\llangle Z_2,u_n\rrangle_{L^2(\til{\Omega};L^2(0,T;X^{b}))},\]
  $\forall u_n\in L^2(\til{\Omega};L^2(0,T;X^{b}_n))$.

  Secondly, for any $u\in L^2(\til{\Omega};L^2(0,T;X^{b})$, there exists $L^2(\til{\Omega};L^2(0,T;X^{b}_n))\ni u_n\longrightarrow u$ as $n\longrightarrow\infty$, hence for all $u\in L^2(\til{\Omega};L^2(0,T;X^{b})$, we have
  \setlength\arraycolsep{2pt}{
  \begin{eqnarray*}
  &&  _{L^2(\til{\Omega};L^2(0,T;\mathbb X^{-b}))}\llangle Z_3,u\rrangle_{L^2(\til{\Omega};L^2(0,T;X^{b}))}= \lim_{n\rightarrow\infty}\,_{L^2(\til{\Omega};L^2(0,T;\mathbb X^{-b}))}\llangle Z_3,u_n\rrangle_{L^2(\til{\Omega};L^2(0,T;X^{b}))}\\
    &=&\lim_{n\rightarrow\infty}\,_{L^2(\til{\Omega};L^2(0,T;\mathbb X^{-b}))}\llangle Z_2,u_n\rrangle_{L^2(\til{\Omega};L^2(0,T;X^{b}))}=_{L^2(\til{\Omega};L^2(0,T;\mathbb X^{-b}))}\llangle Z_2,u\rrangle_{L^2(\til{\Omega};L^2(0,T;X^{b}))}
  \end{eqnarray*}}
  Therefore $Z_2=Z_3\in L^2(\til{\Omega};L^2(0,T;\mathbb X^{-b}))$ and this concludes the proof.
\end{proof}

In next Lemma, we look into $Z_1$. 
\begin{lem}\label{lem:4.5M}
  $Z_1$ is the unique element in $L^{2}(\til{\Omega};L^2(0,T;\mathbb{H}))$ such that for any  $u\in L^4(\til{\Omega};L^4(0,T;\mathbb{W}^{1,4}))$,  the following equality holds
  \setlength\arraycolsep{2pt}{
  \begin{eqnarray*}
    &&\lim_{n\rightarrow\infty}\til{\mathbb{E}}\int_0^T\llangle \til{M}_n(s)\times \til{\rho}_n(s),u(s)\rrangle_{\mathbb{H}}\ud s=\til{\mathbb{E}}\int_0^T\llangle Z_1(s),u(s)\rrangle_{\mathbb{H}}\ud s\\
    &=&\til{\mathbb{E}}\int_0^T\llangle \til{M}(t)\times\big(\vp'(\til{M}(t))+1_{\cl D}(\til{B}-\til{\overline{M}})(t)\big),u(t)\rrangle_{\mathbb{H}}\ud t+\sum_{i=1}^3\til{\mathbb{E}}\int_0^T\llangle\nabla_i\til{M}(t),\til{M}(t)\times \nabla_iu(t)\rrangle_{\mathbb{H}}\ud t.
  \end{eqnarray*}}
\end{lem}
\begin{proof}
Let us recall that
\[\til{\rho}_n:=\pi_n\big[-\vp'(\til{M}_n)+1_{\cl D}(\til{B}_n-\pi_n^\mathbb{Y}\overline{\til{M}}_n)\big]+\Delta \til{M}_n,\]
so we take three parts to prove the desired result.

Firstly we show that
\[\lim_{n\rightarrow\infty}\til{\mathbb{E}}\int_0^T\llangle \til{M}_n(t)\times\Delta \til{M}_n(t),u(t)\rrangle_{\HH}\ud t=\sum_{i=1}^3\til{\mathbb{E}}\int_0^T\llangle \nabla_i\til{M}(t),\til{M}(t)\times \nabla_iu(t)\rrangle_{\HH}\ud t.\]
Proof of above equality:  for each $n\in \mathbb{N}$, we have
  \begin{equation}\label{eq:4.19M}
    \llangle \til{M}_n(t)\times \Delta \til{M}_n(t),u(t)\rrangle_{\mathbb{L}^2}=\sum_{i=1}^3\llangle\nabla_i\til{M}_n(t),\til{M}_n(t)\times \nabla_iu(t)\rrangle_{\mathbb{L}^2}
  \end{equation}
  for almost every $t\in [0,T]$ and $\til{\mathbb{P}}$ almost surely.
  Moreover, by the results: \eqref{eq:4.17phi}, \eqref{eq:SVis 3.34R'} and \eqref{eq:EMn'L4}, we have for $i=1,2,3$,
  \setlength\arraycolsep{2pt}{
  \begin{eqnarray*}
    &&\left|\til{\mathbb{E}}\int_0^T\llangle \nabla_i\til{M},\til{M}\times \nabla_iv\rrangle_{\mathbb{H}}\ud t-\til{\mathbb{E}}\int_0^T\llangle \nabla_i\til{M}_n,\til{M}_n\times \nabla_iv\rrangle_{\mathbb{H}}\ud t\right|\\
    &\leq&\left|\til{\mathbb{E}}\int_0^T\llangle \nabla_i\til{M}-\nabla_i\til{M}_n,\til{M}\times \nabla_iv\rrangle_{\mathbb{H}}\ud t\right|+\left|\til{\mathbb{E}}\int_0^T\llangle \nabla_i\til{M}_n,(\til{M}-\til{M}_n)\times \nabla_iv\rrangle_{\mathbb{H}}\ud t\right|\\
    &\leq&\left(\til{\mathbb{E}}\int_0^T\|\nabla_i\til{M}_n\|_{\mathbb{H}}^2\ud t\right)^\frac{1}{2}\left(\til{\mathbb{E}}\int_0^T\|\til{M}-\til{M}_n\|_{\mathbb{L}^4}^4\ud t\right)^\frac{1}{4}\left(\til{\mathbb{E}}\int_0^T\|\nabla_iv\|_{\mathbb{L}^4}^4\ud t\right)^\frac{1}{4}\\
    &&\hspace{+0.2truecm}+\left|\til{\mathbb{E}}\int_0^T\llangle \nabla_i\til{M}-\nabla_i\til{M}_n,\til{M}\times \nabla_iv\rrangle_{\mathbb{H}}\ud t\right|\rightarrow0,\qquad \textrm{as }n\rightarrow\infty.
  \end{eqnarray*}}

   Secondly we show that
   \[\lim_{n\rightarrow\infty}\til{\mathbb{E}}\int_0^T\llangle \til{M}_n(t)\times\pi_n\vp'(\til{M}_n(t)),u(t)\rrangle_{\mathbb{L}^2}\ud t=\til{\mathbb{E}}\int_0^T\llangle \til{M}(t)\times\vp'(\til{M}(t)),u(t)\rrangle_{\mathbb{L}^2}\ud t.\]
   Proof of the above equality: By \eqref{eq:EMn'L4} and \eqref{eq:EphiMn'L2}, we have
   \setlength\arraycolsep{2pt}{
  \begin{eqnarray*}
    &&\left|\til{\mathbb{E}}\int_0^T\llangle \til{M}_n(s)\times \pi_n\vp'(\til{M}_n(s))-\til{M}(s)\times\vp'(\til{M}(s)),u(s)\rrangle_{\mathbb{H}}\ud s\right|\\
    &\leq&\til{\mathbb{E}}\int_0^T\left|\llangle \big[\til{M}_n(s)-\til{M}(s)\big]\times u(s),\pi_n\vp'(\til{M}_n(s))\rrangle_{\mathbb{H}}\right|\ud s\\
    &&+\til{\mathbb{E}}\int_0^T\left|\llangle \til{M}(s)\times u(s),\pi_n\vp'(\til{M}_n(s))-\vp'(\til{M}(s))\rrangle_{\mathbb{H}}\right|\ud s\\
    &\leq&\left(\til{\mathbb{E}}\int_0^T\|\til{M}_n(s)-\til{M}(s)\|_{\mathbb{L}^4}^4\ud s\right)^\frac{1}{4} \left(\til{\mathbb{E}}\int_0^T\|u(s)\|_{\mathbb{L}^4}^4\ud s\right)^\frac{1}{4} \left(\til{\mathbb{E}}\int_0^T\|\vp'(\til{M}_n(s))\|_{\mathbb{L}^2}^2\ud s\right)^\frac{1}{2}\\
    &&+\left(\til{\mathbb{E}}\int_0^T\|\til{M}(s)\|_{\mathbb{L}^4}^4\ud s\right)^\frac{1}{4}\left(\til{\mathbb{E}}\int_0^T\|u(s)\|_{\mathbb{L}^4}^4\ud s\right)^\frac{1}{4}\left(\til{\mathbb{E}}\int_0^T\|\pi_n\vp'(\til{M}_n(s))-\vp'(\til{M}(s))\|_{\mathbb{L}^2}^2\ud s\right)^\frac{1}{2}\rightarrow0,
  \end{eqnarray*}}
  as $n\rightarrow\infty$.

 Finally, we will show that
 \[\lim_{n\rightarrow\infty}\til{\mathbb{E}}\int_0^T\llangle \til{M}_n(t)\times\pi_n1_{\cl D}(\til{B}_n-\pi_n^\mathbb{Y}\til{\overline{M}}_n)(t),u(t)\rrangle_{\mathbb{L}^2}\ud t=\til{\mathbb{E}}\int_0^T\llangle \til{M}(t)\times1_{\cl D}(\til{B}-\til{\overline{M}})(t),u(t)\rrangle_{\mathbb{L}^2}\ud t.\]
   Proof of the above equality: By \eqref{eq:EMn'L4} and \eqref{eq:EBn'L2w}, we have
 \setlength\arraycolsep{2pt}{
  \begin{eqnarray*}
    &&\left|\til{\mathbb{E}}\int_0^T\llangle \til{M}_n(s)\times \pi_n1_{\cl D}(\til{B}_n-\pi_n^\mathbb{Y}\til{\overline{M}}_n)(s)-\til{M}(s)\times1_{\cl D}(\til{B}-\til{\overline{M}})(s),u(s)\rrangle_{\mathbb{H}}\ud s\right|\\
    &\leq&\til{\mathbb{E}}\int_0^T\left|\llangle \big[\til{M}_n(s)-\til{M}(s)\big]\times u(s),\pi_n1_{\cl D}(\til{B}_n-\pi_n^\mathbb{Y}\til{\overline{M}}_n)(s)\rrangle_{\mathbb{H}}\right|\ud s\\
    &&+\til{\mathbb{E}}\left|\int_0^T\llangle \til{M}(s)\times u(s),\pi_n1_{\cl D}(\til{B}_n-\pi_n^\mathbb{Y}\til{\overline{M}}_n)(s)-1_{\cl D}(\til{B}-\til{\overline{M}})(s)\rrangle_{\mathbb{H}}\ud s\right|\\
    &\leq&\left(\til{\mathbb{E}}\int_0^T\|\til{M}_n(s)-\til{M}(s)\|_{\mathbb{L}^4}^4\ud s\right)^\frac{1}{4} \left(\til{\mathbb{E}}\int_0^T\|u(s)\|_{\mathbb{L}^4}^4\ud s\right)^\frac{1}{4} \left(\til{\mathbb{E}}\int_0^T\|1_{\cl D}(\til{B}_n-\pi_n^\mathbb{Y}\til{\overline{M}}_n)(s)\|_{\mathbb{L}^2}^2\ud s\right)^\frac{1}{2}\\
    &&+\left(\til{\mathbb{E}}\int_0^T\|\til{M}(s)\|_{\mathbb{L}^4}^4\ud s\right)^\frac{1}{4}\left(\til{\mathbb{E}}\int_0^T\|u(s)\|_{\mathbb{L}^4}^4\ud s\right)^\frac{1}{4}\left(\til{\mathbb{E}}\int_0^T\|\pi_n1_{\cl D}\pi_n^\mathbb{Y}\til{\overline{M}}_n(s)-1_{\cl D}\til{\overline{M}}(s)\|_{\mathbb{L}^2}^2\ud s\right)^\frac{1}{2}\\
    &&+\til{\mathbb{E}}\left|\int_0^T\llangle \til{M}(s)\times u(s),\pi_n1_{\cl D}(\til{B}_n-\til{B})(s)\rrangle_{\mathbb{H}}\ud s\right|\longrightarrow 0,\qquad \textrm{as }n\rightarrow\infty.
  \end{eqnarray*}}
So far we have shown that
  \setlength\arraycolsep{2pt}{
  \begin{eqnarray*}
 &&\lim_{n\rightarrow\infty}\til{\mathbb{E}}\int_0^T\llangle \til{M}_n(s)\times \til{\rho}_n(s),u(s)\rrangle_{\mathbb{H}}\ud s\\
 &=&\til{\mathbb{E}}\int_0^T\llangle \til{M}(t)\times\big(\vp'(\til{M}(t))+1_{\cl D}(\til{B}-\til{\overline{M}})(t)\big),u(t)\rrangle_{\mathbb{H}}\ud t+\sum_{i=1}^3\llangle\nabla_i\til{M}_n(t),\til{M}_n(t)\times \nabla_iu(t)\rrangle_{\mathbb{H}},
   \end{eqnarray*}}
   for all $u\in L^4(\til{\Omega};L^4(0,T;\mathbb{W}^{1,4}))$.
Since $L^4(\til{\Omega};L^4(0,T;\mathbb{W}^{1,4}))$ is dense in $ L^2(\til{\Omega};L^2(0,T;\mathbb{H}))$, we also have
\[
  \lim_{n\rightarrow\infty}\til{\mathbb{E}}\int_0^T\llangle \til{M}_n(s)\times \til{\rho}_n(s),u(s)\rrangle_{\mathbb{H}}\ud s=\til{\mathbb{E}}\int_0^T\llangle Z_1(s),u(s)\rrangle_{\mathbb{H}}\ud s,\quad u\in L^4(\til{\Omega};L^4(0,T;\mathbb{W}^{1,4})),
  \]
  and such $Z_1$ is unique in $L^{2}(\til{\Omega};L^2(0,T;\mathbb{H}))$.

  This completes the proof.
\end{proof}

\begin{notation}	\label{nt:Mxrho}
We will denote $\til{M}\times \til{\rho}:=Z_1$.
\end{notation}

\begin{rem}\label{rem:MxrhoinL2}
By the Notation \ref{nt:Mxrho}, the Lemma \ref{lem:4.5M} shows that $\til{M}\times \til{\rho}\in L^{2}(\til{\Omega};L^2(0,T;\mathbb{H}))$ and
\[\lim_{n\rightarrow\infty}\til{\mathbb{E}}\int_0^T\llangle \til{M}_n(s)\times \til{\rho}_n(s),u(s)\rrangle_{\mathbb{H}}\ud s=\til{\mathbb{E}}\int_0^T\llangle \til{M}(s)\times \til{\rho}(s),u(s)\rrangle_{\mathbb{H}}\ud s,\]
for all $u\in L^2(\til{\Omega};L^2(0,T;\mathbb{H}))$.
By \eqref{eq:S 4.12}, we also have
\[\til{M}\times (\til{M}\times \til{\rho})\in L^\frac{4}{3}(\til{\Omega};L^2(0,T;\mathbb{L}^\frac{3}{2}(\cl D))).\]
\end{rem}

\begin{lem}\label{lem:4.8M}
  For any  $\eta\in L^4(\til{\Omega};L^4(0,T;\mathbb{L}^4)$ we have
  \setlength\arraycolsep{2pt}{
  \begin{eqnarray}
    &&\lim_{n\rightarrow\infty}\til{\mathbb{E}}\int_0^T{}_{\mathbb{L}^\frac{3}{2}}\llangle \til{M}_n(s)\times(\til{M}_n(s)\times \til{\rho}_n(s)),\eta(s)\rrangle_{\mathbb{L}^3(D)}\ud s\nonumber\\
    &=&\til{\mathbb{E}}\int_0^T{}_{\mathbb{L}^\frac{3}{2}}\llangle Z_2(s),\eta(s)\rrangle_{\mathbb{L}^3(D)}\ud s\label{eq:4.25M}\\
    &=&\til{\mathbb{E}}\int_0^T{}_{\mathbb{L}^\frac{3}{2}}\llangle \til{M}(s)\times Z_1(s),\eta(s)\rrangle_{\mathbb{L}^3(D)}\ud s\label{eq:4.26M}
  \end{eqnarray}}
\end{lem}
\begin{proof}
  Let us denote $Z_{1n}:=\til{M}_n\times\til{\rho}_n$ for each $n\in \mathbb{N}$. $L^4(\til{\Omega};L^4(0,T;\mathbb{L}^4))\subset L^2(\til{\Omega};L^2(0,T;\mathbb{L}^3))$ which is the dual space of $L^2(\til{\Omega};L^2(0,T;\mathbb{L}^\frac{3}{2}))$. Hence \eqref{eq:4.15M} implies that \eqref{eq:4.25M} holds.

  Next we are going to prove \eqref{eq:4.26M}.

  By \eqref{eq:EMn'L4}, $\til{M}\in L^4(\til{\Omega};L^4(0,T;\mathbb{L}^4))$, hence by the H$\ddot{\textrm{o}}$lder inequality, we have
  \setlength\arraycolsep{2pt}{
  \begin{eqnarray*}
    &&\til{\mathbb{E}}\int_0^T\|\eta\times \til{M}\|_{\mathbb{L}^2}^2\ud t\le \til{\mathbb{E}}\int_0^T\|\eta\|_{\mathbb{L}^4}^2\|\til{M}\|_{\mathbb{L}^4}^2\ud t\\
    &\leq&\til{\mathbb{E}}\int_0^T\|\eta\|_{\mathbb{L}^4}^4\ud t+\til{\mathbb{E}}\int_0^T\|\til{M}\|_{\mathbb{L}^4}^4\ud t<\infty.
  \end{eqnarray*}}
   So $\eta\times \til{M}\in L^2(\til{\Omega};L^2(0,T;\mathbb{L}^2))$ and similarly $\eta\times \til{M}_n\in L^2(\til{\Omega};L^2(0,T;\mathbb{L}^2))$.

   By \eqref{eq:Z1w}, $Z_{1n}\in L^2(\til{\Omega};L^2(0,T;\mathbb{L}^2))$. And $\eta\times \til{M}_n\in L^2(\til{\Omega};L^2(0,T;\mathbb{L}^2))$. Hence
   \setlength\arraycolsep{2pt}{
  \begin{eqnarray}
  &&  \,\!_{\mathbb{L}^\frac{3}{2}}\llangle \til{M}_n\times Z_{1n},\eta\rrangle_{\mathbb{L}^3}=\int_{\cl D}\llangle \til{M}_n(x)\times Z_{1n}(x),\eta(x)\rrangle\ud x\nonumber\\
    &=&\int_{\cl D}\llangle Z_{1n}(x),\eta(x)\times \til{M}_n(x)\rrangle\ud x\label{eq:4.261M}=\llangle Z_{1n},\eta\times \til{M}_n\rrangle_{\mathbb{L}^2}
  \end{eqnarray}}
   By \eqref{eq:Z1w}, $Z_1\in L^2(\til{\Omega};L^2(0,T;\mathbb{L}^2))$. And $\eta\times \til{M}\in L^2(\til{\Omega};L^2(0,T;\mathbb{L}^2))$. So
   \setlength\arraycolsep{2pt}{
  \begin{eqnarray}
    &&\,\!_{\mathbb{L}^\frac{3}{2}}\llangle \til{M}\times Z_1,\eta\rrangle_{\mathbb{L}^3}=\int_{\cl D}\llangle \til{M}(x)\times Z_1(x),\eta(x)\rrangle\ud x\nonumber\\
    &=&\int_{\cl D}\llangle Z_1(x),\eta(x)\times \til{M}(x)\rrangle\ud x\label{eq:4.262M}=\llangle Z_1,\eta\times \til{M}\rrangle_{\mathbb{L}^2}
  \end{eqnarray}}
  By \eqref{eq:4.261M} and \eqref{eq:4.262M},
  \setlength\arraycolsep{2pt}{
  \begin{eqnarray*}
    \,\!_{\mathbb{L}^\frac{3}{2}}\llangle \til{M}_n\times Z_{1n},\eta\rrangle_{\mathbb{L}^3}-\,\!_{\mathbb{L}^\frac{3}{2}}\llangle \til{M}\times Z_1,\eta\rrangle_{\mathbb{L}^3}&=&\llangle Z_{1n},\eta\times \til{M}_n\rrangle_{\mathbb{L}^2}-\llangle Z_1,\eta\times \til{M}\rrangle_{\mathbb{L}^2}\\
    &=&\llangle Z_{1n}-Z_1,\eta\times \til{M}\rrangle_{\mathbb{L}^2}+\llangle Z_{1n},\eta\times(\til{M}_n-\til{M})\rrangle_{\mathbb{L}^2}.
  \end{eqnarray*}}
  By \eqref{eq:Z1w}, and since $\eta\times \til{M}\in L^2(\til{\Omega};L^2(0,T;\mathbb{L}^2))$,
  \[\lim_{n\rightarrow\infty}\til{\mathbb{E}}\int_0^T\llangle Z_{1n}(s)-Z_1(s),\eta(s)\times \til{M}(s)\rrangle_{\mathbb{L}^2}\ud s=0.\]
  By the Cauchy-Schwartz inequality,
  \setlength\arraycolsep{2pt}{
  \begin{eqnarray*}
    &&\llangle Z_{1n},\eta\times(\til{M}_n-\til{M})\rrangle_{\mathbb{L}^2}^2\leq\|Z_{1n}\|^2_{\mathbb{L}^2}\|\eta\times (\til{M}_n-\til{M})\|^2_{\mathbb{L}^2}\\
    &\leq&\|Z_{1n}\|^2_{\mathbb{L}^2}(\|\eta\|_{\mathbb{L}^4}^4+\|\til{M}_n-\til{M}\|_{\mathbb{L}^4}^4)\rightarrow 0,\qquad\textrm{as }n\longrightarrow\infty.
  \end{eqnarray*}}
  Hence
  \[\lim_{n\rightarrow\infty}\til{\mathbb{E}}\int_0^T\llangle Z_{1n}(s),\eta\times(\til{M}_n-\til{M})(s)\rrangle_{\mathbb{L}^2}\ud s=0.\]
  Therefore,
  \[\lim_{n\rightarrow\infty}\til{\mathbb{E}}\int_0^T\ \! _{\mathbb{L}^\frac{3}{2}}\llangle \til{M}_n(s)\times(\til{M}_n(s)\times\til{\rho}_n(s)),\eta(s)\rrangle_{\mathbb{L}^3}\ud s=\til{\mathbb{E}}\int_0^T\ \! _{\mathbb{L}^\frac{3}{2}}\llangle \til{M}(s)\times Z_1(s),\eta(s)\rrangle_{\mathbb{L}^3}\ud s.\]
  This completes the proof of the Lemma \ref{lem:4.8M}.
\end{proof}

\begin{rem}\label{rem:Z2=MxMxrho}
  By the notation \ref{nt:Mxrho}, the Lemma \ref{lem:4.8M} has proved that
  \[Z_2=\til{M}\times(\til{M}\times \til{\rho})\]
  in $L^2(\til{\Omega};L^2(0,T;\mathbb{L}^\frac{3}{2}))$.
  So
  \[\til{M}_n\times(\til{M}_n\times\til{\rho}_n)\longrightarrow \til{M}\times(\til{M}\times \til{\rho}) \quad\textrm{weakly in } L^2(\til{\Omega};L^2(0,T;\mathbb{L}^\frac{3}{2})).\]
\end{rem}

The next result will be used to show that the process $\til{M}$ satisfies the condition $|\til{M}(t,x)|=1$ for all $t\in [0,T]$, $x\in \cl D$ and $\til{\mathbb{P}}$-almost surely.

\begin{lem}
  For any bounded measurable function $\varphi:\cl D\longrightarrow \R$, we have
  \[\llangle Z_1(s,\omega),\varphi \til{M}(s,\omega)\rrangle_{\mathbb{H}}=0,\]
  for almost every $(s,\omega)\in [0,T]\times \til{\Omega}$.
\end{lem}
\begin{proof}
  Let $B\subset [0,T]\times \til{\Omega}$ be a measurable set and $1_B$ be the indicator function of $B$. Then
  \setlength\arraycolsep{2pt}{
  \begin{eqnarray*}
    &&\til{\mathbb{E}}\int_0^T\|1_B\varphi \til{M}_n(t)-1_B\varphi \til{M}(t)\|_{\mathbb{L}^2}\ud t=\til{\mathbb{E}}\int_0^T\|1_B\varphi[\til{M}_n(t)-\til{M}(t)]\|_{\mathbb{L}^2}\ud t\\
    &\leq&\|\varphi\|_{\mathbb{L}^\infty}\til{\mathbb{E}}\int_0^T\|\til{M}_n(t)-\til{M}(t)\|_{\mathbb{L}^2}\ud t\leq C\|\varphi\|_{\mathbb{L}^\infty}\til{\mathbb{E}}\int_0^T\|\til{M}_n(t)-\til{M}(t)\|_{\mathbb{L}^4}\ud t,
  \end{eqnarray*}}
  for some constant $C>0$. Hence by \eqref{eq:EMn'L4}, we have
  \[\lim_{n\rightarrow\infty}\til{\mathbb{E}}\int_0^T\|1_B\varphi \til{M}_n(t)-1_B\varphi \til{M}(t)\|_{\mathbb{L}^2}\ud t=0.\]
  Together with the fact that $\til{M}_n\times \til{\rho}_n$ converges to $Z_1$ weakly in $L^2(\til{\Omega};L^2(0,T;\mathbb{L}^2))$ we can infer that
  \[0=\lim_{n\rightarrow\infty}\til{\mathbb{E}}\int_0^T1_B(s)\llangle \til{M}_n(s)\times\til{\rho}_n(s),\varphi \til{M}_n(s)\rrangle_{\mathbb{L}^2}\ud s=\til{\mathbb{E}}\int_0^T1_B(s)\llangle Z_1(s),\varphi \til{M}(s)\rrangle_{\mathbb{L}^2}\ud s.\]
  This complete the proof.
\end{proof}

\section{The existence of a weak solution}
In this section, we will prove that the process $(\til{M},\til{B},\til{E})$ from Lemma \ref{prop:tilMn} is a weak solution of the Problem \ref{S-LLG}.

To explain how we will prove the result, let us define
\begin{equation}\label{eq:xinsol}
\begin{aligned}
\xi_n(t)&:=M_n(t)-M_n(0)\\
&\quad-\int_0^t\Bigg\{\lambda_1\pi_n\left[ M_n\times \rho_n\right]-\lmd_2 \pi_n\left[M_n\times(M_n\times \rho_n)\right]+\frac{1}{2}\sum_{j=1}^\infty G^\prime_{jn}\la M_n\ra\left[G_{jn}\la M_n\ra\right]\Bigg\}\ud s
\end{aligned}
\end{equation}
Because $M_n$ satisfies \eqref{eq:SVis 3.29R}, we have
\[\xi_n(t)=\sum_{j=1}^\infty\int_0^tG_{jn}\la M_n\ra\ud W_j(s).\]
  Then the proof will consists in three steps:
\begin{trivlist}
\item[Step 1]:
 We are going to find some $\til{\xi}$ as a limit of $\til{\xi}_n$ which are similar to $\xi_n$ defined  in \eqref{eq:xinsol} as $n\rightarrow\infty$.
\item[Step 2]:
  We will show the second "$=$" in \eqref{eq:xinsol} holds for the limit process $\til{\xi}$, but with $\til{M}$ instead of $M_n$ and $\til{W}_j$ instead of $W_j$, etc.
  \item[Step 3]:
  We will get rid of the auxiliary function $\psi$ and finish the proof.
 \end{trivlist}
\subsection{Step 1}
Let us denote
\setlength\arraycolsep{2pt}{
  \begin{eqnarray*}
\til{\xi}_n(t)&:=&\til{M}_n(t)-\til{M}_n(0)-\int_0^t\Bigg\{\pi_n\left[\lmd_1 \til{M}_n\times \til{\rho}_n\right]-\lmd_2 \pi_n\left[\til{M}_n\times(\til{M}_n\times \til{\rho}_n)\right]\nonumber\\
  &&+\frac{1}{2}\sum_{j=1}^\infty G_{jn}^\prime\la\til M_n\ra\left[G_{jn}\la\til M_n\ra\right]\Bigg\}\ud s\nonumber.
  \end{eqnarray*}}

\begin{lem}\label{lem:xi't}
  For each $t\in [0,T]$ the sequence of random variables $\til{\xi}_n(t)$ converges weakly in $L^2(\til{\Omega};\mathbb X^{-b})$ to the limit
  \setlength\arraycolsep{2pt}{
  \begin{eqnarray*}
\til{\xi}(t)&:=&\til{M}(t)-M_0-\int_0^t\Bigg\{\left[\lmd_1 \til{M}\times \til{\rho}\right]-\lmd_2 \left[\til{M}\times(\til{M}\times \til{\rho})\right]\nonumber\\
  &&+\frac{1}{2}\sum_{j=1}^\infty\gj\Bigg\}\ud s\nonumber.
  \end{eqnarray*}}
  as $n\rightarrow\infty$.
\end{lem}

\begin{proof}
The dual space of $L^2(\til{\Omega};\mathbb X^{-b})$ is $L^2(\til{\Omega};X^b)$. Let $t\in(0,T]$ and $U\in L^2(\til{\Omega};X^b)$. We have
\setlength\arraycolsep{2pt}{
  \begin{eqnarray*}
    &&\,\!_{L^2(\til{\Omega};\mathbb X^{-b})}\llangle \til{\xi}_n(t),U\rrangle_{L^2(\til{\Omega};X^b)}=\til{\mathbb{E}}\left[_{\mathbb X^{-b}}\llangle \til{\xi}_n(t),U\rrangle_{X^b}\right]\\
    &=&\til{\mathbb{E}}\Bigg\{\llangle \til{M}_n(t),U\rrangle_\mathbb{H}-\llangle M_n(0),U\rrangle_\mathbb{H}-\lmd_1\int_0^t\llangle \til{M}_n(s)\times\til{\rho}_n(s),\pi_n U\rrangle_{\mathbb{H}}\ud s\\
    &&+\lmd_2\int_0^t \llangle (\til{M}_n(s)\times(\til{M}_n(s)\times\til{\rho}_n(s))),\pi_n U\rrangle_\mathbb{H}\ud s
    -\frac{1}{2}\sum_{j=1}^\infty\int_0^t\llangle\gjn,\pi_nU\rrangle\ud s\Bigg\}
  \end{eqnarray*}}
  Next we are going to consider the right hand side of above equality term by term.

  By the Lemma \ref{prop:tilMn}, $\til{M}_n\longrightarrow \til{M}$ in $C([0,T];\mathbb X^{-b})$ $\til{\mathbb{P}}$-a.s., so
   \[\sup_{t\in[0,T]}\|M_n(t)-M(t)\|_{\mathbb X^{-b}}\longrightarrow0,\quad \til{\mathbb{P}}-a.s.\]
 and $\,\!_{\mathbb X^{-b}}\llangle \cdot,U\rrangle_{X^b}$ is a continuous function on $\mathbb X^{-b}$, hence
   \[\lim_{n\rightarrow\infty}\,\!_{\mathbb X^{-b}}\llangle \til{M}_n(t),U\rrangle_{X^b}=\,\!_{\mathbb X^{-b}}\llangle \til{M}(t),U\rrangle_{X^b},\quad \til{\mathbb{P}}-a.s.\]
   By \eqref{eq:S 3.10R'}, $\sup_{t\in[0,T]}|\til{M}_n(t)|_\mathbb{H}\leq|M_0|_{\mathbb{H}}$, so that we can find a constant $C$ such that
   \setlength\arraycolsep{2pt}{
  \begin{eqnarray*}
    &&\sup_n\til{\mathbb{E}}\left[\left|{}_{\mathbb X^{-b}}\llangle \til{M}_n(t),U\rrangle_{X^b}\right|^2\right]\leq\sup_n\til{\mathbb{E}}\|U\|^2_{X^b}\til{\mathbb{E}}\|\til{M}_n(t)\|^2_{\mathbb X^{-b}}\\
    &\leq&C\til{\mathbb{E}}\|U\|_{X^b}^2\til{\mathbb{E}}\sup_n\|\til{M}_n(t)\|^2_{\mathbb{H}}\leq C\til{\mathbb{E}}\|U\|_{X^b}^2\til{\mathbb{E}}\|M_0\|_{\mathbb{H}}^2<\infty.
  \end{eqnarray*}}
 Therefore, the sequence ${}_{\mathbb X^{-b}}\llangle \til{M}_n(t),U\rrangle_{X^{b}}$ is uniformly integrable.
   So the almost surely convergence and uniform integrability implies that
    $$\lim_{n\rightarrow \infty}\til{\mathbb{E}}[\,\!_{\mathbb X^{-b}}\llangle \til{M}_n(t),U\rrangle_{X^b}]=\til{\mathbb{E}}[\,\!_{\mathbb X^{-b}}\llangle \til{M}(t),U\rrangle_{X^b}].$$
   By \eqref{eq:Z1w},
   \[\lim_{n\rightarrow\infty}\til{\mathbb{E}}\int_0^t\llangle \til{M}_n(s)\times\til{\rho}_n(s),\pi_nU\rrangle_\mathbb{H}\ud s=\til{\mathbb{E}}\int_0^t\llangle Z_1(s), U\rrangle_\mathbb{H}.\]
   By \eqref{eq:4.16M}
   \[\lim_{n\rightarrow\infty}\til{\mathbb{E}}\int_0^t\,\!_{\mathbb X^{-b}}\llangle \pi_n(\til{M}_n(s)\times(\til{M}_n(s)\times\til{\rho}_n(s))),U\rrangle_{X^b}\ud s=\til{\mathbb{E}}\int_0^t\,\!\llangle Z_2(s),U\rrangle_{X^b}\ud s.\]
   By the H$\ddot{\textrm{o}}$lder's inequality,
   \[\|\til{M}_n(t)-\til{M}(t)\|_{\mathbb{L}^2}^2\leq\|\til{M}_n(t)-\til{M}(t)\|_{\mathbb{L}^4}^2.\]
   We will show that for any $U\in L^2\left(\til{\Omega};L^2(0,T;\mathbb{H}\right)$
  \begin{equation}\label{eq_lim}
   \lim_{n\to\infty}\E\sum_{j=1}^\infty\int_0^t\llangle\gjn,\pi_nU\rrangle_{\mathbb H}\ud s=\E\sum_{j=1}^\infty\int_0^t\llangle\gj,U\rrangle_{\mathbb H}\ud s\,.
   \end{equation}
  Using \eqref{eq:S 3.10R'} we can prove
  \[\left|\llangle\gjn,\pi_nU\rrangle_{\mathbb H}\right|^2\le C\left\|h_j\right\|^4_{\mathbb L^\infty}\|U\|_{\mathbb H}^2\,,\]
  it remains to show that
  \[\lim_{n\to\infty}\llangle\gjn,\pi_nU\rrangle_{\mathbb H}=\llangle\gj,U\rrangle_{\mathbb H}\,.\]
  This follows imediately from the convergence of $\til M_n(t)$ to $\til M(t)$ for every $t\in[0,T]$ $\bb P$-a.s.
  
  Hence
  \setlength\arraycolsep{2pt}{
  \begin{eqnarray*}
    &&\lim_{n\rightarrow\infty}\,\!_{L^2(\til{\Omega};\mathbb X^{-b})}\llangle \til{\xi}_n(t),U\rrangle_{L^2(\til{\Omega};X^b)}\\
    &=&\til{\mathbb{E}}\Big[\,\!_{\mathbb X^{-b}}\llangle \til{M}(t), U\rrangle_{X^b}-\,\!_{X^{-\beta}}\llangle M_0,U\rrangle_{X^b}-\lmd_1\int_0^t\llangle Z_1(s),U\rrangle_\mathbb{H}\ud s\\
    &&+\lmd_2\int_0^t\,\!_{\mathbb X^{-b}}\llangle Z_2(s),U\rrangle_{X^b}\ud s-\frac{1}{2}\sum_{j=1}^\infty  \llangle G_j^\prime\la\til M\ra\left[G_j\la\til M\ra\right],U\rrangle\ud s
  \end{eqnarray*}}
  Since by Lemma \ref{lem:4.5M} and Lemma \ref{lem:4.8M}, we have $Z_1=\til{M}\times \til{\rho}$ and $Z_2=\til{M}\times(\til{M}\times\til{\rho})$. Therefore for any $U\in L^2(\til{\Omega};X^b)$,
  \[\lim_{n\rightarrow\infty}\,\!_{L^2(\til{\Omega};\mathbb X^{-b})}\llangle \til{\xi}_n(t),U\rrangle_{L^2(\til{\Omega};X^b)}=\,\!_{L^2(\til{\Omega};\mathbb X^{-b})}\llangle \til{\xi}(t),U\rrangle_{L^2(\til{\Omega};X^b)}.\]
  This concludes the proof.
\end{proof}

\subsection{Step 2}
In this step we are going to show that
\begin{equation}
	\til \xi(t)=\sum_{j=1}^\infty\int_0^tG_j^\psi(\til M(s))\ud \til W_j(s).
\end{equation}
We will finish this step by the approximation method. To do this, we need the next Lemma for preparation. Let us define, for each $m\in\mathbb{N}$, a partition $\left\{s_i^m:=\frac{iT}{m},\;i=0,\ldots,m\right\}$ of $[0,T]$.  It will be convenient to define on $[0,T]$ processes
  \begin{equation}\label{sigma1}
  \sigma_{jn}(s)=G_{jn}\la\til M_n(s)\ra\,,
  \end{equation}
  and
  \begin{equation}\label{sigma2}
  \sigma_{jn}^m(s)=\sum_{i=0}^{m-1}G_{jn}\la\til M_n\la s_i^m\ra\ra 1_{(s_i^m,s_{i+1}^m]}(s)=\sum_{i=0}^{m-1}\sigma_{jn}\la s_i^m\ra 1_{(s_i^m,s_{i+1}^m]}(s)\,.
  \end{equation}

\begin{lem}\label{lem:4 ineqM}
   For any $\eps>0$, We can choose $m\in\mathbb{N}$ sufficiently large such that:
      \begin{trivlist}
        \item[(i)]
       \[
  \limsup_{n\rightarrow\infty}\til{\mathbb{E}}\Bigg\|\sum_{j=1}^\infty\int_0^t\left[\sigma_{jn}(s)-\sigma^m_{jn}(s)\right]\ud \til{W}_{jn}(s)\Bigg\|^2_{\mathbb X^{-b}}<\frac{\eps^2}{4};
  \]
        \item[(ii)]
    \[
         \lim_{n\rightarrow\infty}\til{\mathbb{E}}\left\|\sum_{j=1}^\infty\int_0^t\sigma_{jn}^m(s)\ud \til W_{jn}(s)-\sum_{j=1}^\infty \int_0^t\sigma_{jn}^m(s)\ud \til W_{j}(s)\right\|^2_{\mathbb X^{-b}}=0\,,\]
        \item[(iii)]
   \[ \limsup_{n\rightarrow\infty}\til{\mathbb{E}}\left\|\sum_{j=1}^\infty\la \sigma^m_{jn}(s)-
 \sigma_{jn}(s)\ra\ud\til W_j(s)
        \right\|^2_{\mathbb X^{-b}}<\frac{\eps^2}{4}\,,
        \]

        \item[(iv)]
        \[\lim_{n\rightarrow\infty}\til{\mathbb E}\left\|
      \int_0^t\sum_{j=1}^\infty\la \sigma_{jn}(s)-G^\psi_j\la\til M(s)\ra\ra\ud\til W_j(s)
       \right\|^2_{\mathbb X^{-b}}=0\,. \]
      \end{trivlist}
    \end{lem}
\begin{proof}
\begin{trivlist}
  \item[(i)] By It\^o isometry, our assumptions on $\psi$ and $h_j$, there exists some constants $C>0$, such that
  \[\begin{aligned}
  &\til{\mathbb{E}}\Bigg\|\sum_{j=1}^\infty\int_0^t\left[\sigma_{jn}(s)-\sigma^m_{jn}(s)\right]\ud \til{W}_{jn}(s)\Bigg\|^2_{\mathbb X^{-b}}\\
  =&\sum_{j=1}^\infty\til{\mathbb E}\int_0^t\left\|\sigma_{jn}(s)-\sigma_{jn}^m(s)\right\|^2_{\mathbb X^{-b}}\ud s
  = \sum_{j=1}^\infty
  \sum_{i=1}^{m-1}\til{\mathbb E}\int_{s_i^m}^{s_{i+1}^m}\left\|G_{jn}\la\til M_n(s)\ra-G_{jn}\la\til M_n\la s_i^m\ra\ra\right\|^2_{\mathbb X^{-b}}\ud s \\
  \le& C \la\sum_{j=1}^\infty\left\|h_j\right\|_{\mathbb L^\infty}^2\ra \la\sum_{i=1}^{m-1}\til{\mathbb E}\int_{s_i^m}^{s_{i+1}^m}\left\|\til M_n(s)-\til M_n\la s_i^m\ra\right\|^2_{\mathbb X^{-b}}\ud s\ra\\
  \le& C\,\sum_{i=1}^{m-1}\til{\mathbb E}\int_{s_i^m}^{s_{i+1}^m}\left\|\til M_n(s)-\til M\la s\ra\right\|^2_{\mathbb X^{-b}}\ud s+C\,\sum_{i=1}^{m-1}\til{\mathbb E}\int_{s_i^m}^{s_{i+1}^m}\left\|\til M(s)-\til M\la s_i^m\ra\right\|^2_{\mathbb X^{-b}}\ud s\\
  &\,\,\,\,\,\,+C\,\sum_{i=1}^{m-1}\til{\mathbb E}\int_{s_i^m}^{s_{i+1}^m}\left\|\til M\la s_i^m\ra-\til M_n\la s_i^m\ra\right\|^2_{\mathbb X^{-b}}\ud s\\
  =&C\la I_1(n)+I_2(m)+I_3(n,m)\ra\,.
 \end{aligned} \]
  By the estimate \eqref{eq:EMn'L4}, $\lim_{n\to\infty}I_1(n)=0$.
   
   Since $\til M\in C([0,T];\bb X^{-b})$, for every $\eps>0$ we can find $m_0$ such that
  \[I_2(m)<\frac{\eps^2}{4 C},\quad\mathrm{for}\quad m>m_0\,.\]
For any $m\ge 1$
\[I_3(n,m)\le T \E\sup_{s\in[0,T]}\left\|\til M(s)-\til M_n(s)\right\|^2_{\mathbb X^{-b}}\,.\]
By Lemma \ref{prop:tilMn} (ii), 
\[\lim_{n\rightarrow\infty}\sup_{s\in[0,T]}\left\|\til M(s)-\til M_n(s)\right\|^2_{\mathbb X^{-b}}=0,\qquad \til{\bb P}-a.s.,\]
so by the dominated convergence theorem, 
 $\lim_{n\to\infty}I_3(m,n)=0$ for every $m$. Finally, combining these facts together we obtain (i). \\
 (ii) By the estimate \eqref{eq:S 3.10R'}, remark \ref{rem:tilWh} and Jensen inequality we have
 \[\begin{aligned}
 &\til{\mathbb{E}}\left\|\sum_{j=1}^\infty\int_0^t\sigma_{jn}^m(s)\ud \til W_{jn}(s)-\sum_{j=1}^\infty \int_0^t\sigma_{jn}^m(s)\ud \til W_{j}(s)\right\|^2_{\mathbb X^{-b}}\\
 \le&\E\left[\sum_{j=1}^\infty\sum_{i=1}^{m-1}\left\|G_{jn}\left(\til M_n\left( s_i^m\right)\right)\right\|_{\mathbb X^{-b}}\left| \til W_{jn}\left( t\wedge s_{i+1}^m\right)-\til W_{jn}\left( t\wedge s_{i}^m\right)-\left(  \til W_{j}\left( t\wedge s_{i+1}^m\right)-\til W_{j}\left( t\wedge s_{i}^m\right)\right)\right|\right]^2\\
 \le& C\E\left[\sum_{j=1}^\infty \left\|h_j\right\|_{\mathbb L^\infty}\sum_{i=1}^{m-1}\left| \til W_{jn}\left( t\wedge s_{i+1}^m\right)-\til W_{jn}\left( t\wedge s_{i}^m\right)-\left(  \til W_{j}\left( t\wedge s_{i+1}^m\right)-\til W_{j}\left( t\wedge s_{i}^m\right)\right)\right|\right]^2\\
\le & C\sum_{j=1}^\infty\left\|h_j\right\|_{\mathbb L^\infty}^2 \E\la \sum_{i=1}^{m-1} \left| \til W_{jn}\left( t\wedge s_{i+1}^m\right)-\til W_{jn}\left( t\wedge s_{i}^m\right)-\left(  \til W_{j}\left( t\wedge s_{i+1}^m\right)-\til W_{j}\left( t\wedge s_{i}^m\right)\right)\right|\ra^2\\
\le  &C_m\sum_{j=1}^\infty\left\|h_j\right\|_{\mathbb L^\infty}^2\E\sum_{i=1}^{m-1}\left| \til W_{jn}\left( t\wedge s_{i+1}^m\right)-\til W_{jn}\left( t\wedge s_{i}^m\right)-\left(  \til W_{j}\left( t\wedge s_{i+1}^m\right)-\til W_{j}\left( t\wedge s_{i}^m\right)\right)\right|^2
  \end{aligned}\]
  For $m$ fixed we have
\[\sup_{n,j}\E\la \sum_{i=1}^{m-1}\left| \til W_{jn}\left( t\wedge s_{i+1}^m\right)-\til W_{jn}\left( t\wedge s_{i}^m\right)-\left(  \til W_{j}\left( t\wedge s_{i+1}^m\right)-\til W_{j}\left( t\wedge s_{i}^m\right)\right)\right|^2\ra^2<\infty\,.\]
Therefore, we can pass with $n$ to the limit under the sum and expectation above and since $\til W_n$ converges to $\til W$ in $C\la[0,T];\R^{\mathbb N}\ra$ we obtain
\[\lim_{n\to\infty}\til{\mathbb{E}}\left\|\sum_{j=1}^\infty\int_0^t\sigma_{jn}^m(s)\ud \til W_{jn}(s)-\sum_{j=1}^\infty \int_0^t\sigma_{jn}^m(s)\ud \til W_{j}(s)\right\|^2_{\mathbb X^{-b}}=0\,.\]

  \item[(iii)] The proof of (iii) is same as the proof of (i).
  \item[(iv)] By It\^{o} isometry, we have
    \[
    \til{\mathbb E}\left\|
      \int_0^t\sum_{j=1}^\infty\la \sigma_{jn}(s)-G^\psi_j\la\til M(s)\ra\ra\ud\til W_j(s)
       \right\|^2_{\mathbb X^{-b}}=\sum_{j=1}^\infty\E\int_0^t\left\|G_{jn}\la\til M_n(s)\ra-G^\psi_j\la\til M(s)\ra\right\|^2_{\mathbb X^{-b}}\ud s.
       \]
 By our assumption on $h_j$, the estimates \eqref{eq:S 3.10R'} and \eqref{eq:S 4.11}, we have
 \[\sup_n\sum_{j=1}^\infty\E\int_0^t\left\|G_{jn}\la\til M_n(s)\ra-G^\psi_j\la\til M(s)\ra\right\|^4_{\mathbb X^{-b}}\le\sup_n\sum_{j=1}^\infty C\left\|h_j\right\|^4_{\mathbb L^\infty}\E\int_0^t\la \left\|\til M_n(s)\right\|^4_{\mathbb X^{-b}}+\left\|\til M(s)\right\|^4_{\mathbb X^{-b}}\ra\ud s<\infty
 \]
 and
 \[\begin{aligned}
 \left\|G_{jn}\la\til M_n(s)\ra-G^\psi_j\la\til M(s)\ra\right\|^2_{\mathbb X^{-b}}&\le 2\left\|G_{jn}\la\til M_n(s)\ra-G_{jn}\la\til M(s)\ra\right\|^2_{\mathbb X^{-b}}+2\left\|G_{jn}\la\til M(s)\ra-G^\psi_j\la\til M(s)\ra\right\|^2_{\mathbb X^{-b}}\\
 &\le C\left\|h_j\right\|^2_{\mathbb L^\infty}\left\|\til M_n(s)-\til M(s)\right\|^2_{\mathbb X^{-b}}+2\left\|G_{jn}\la\til M(s)\ra-G^\psi_j\la\til M(s)\ra\right\|^2_{\mathbb X^{-b}}.
 \end{aligned}\]
By the Lemma \ref{prop:tilMn}, $\til{M}_n\longrightarrow \til{M}$ in $C([0,T];\mathbb X^{-b})$ $\til{\mathbb{P}}$-a.s., therefore
 \[\lim_{n\to\infty} \left\|G_{jn}\la\til M_n(s)\ra-G^\psi_j\la\til M(s)\ra\right\|^2_{\mathbb X^{-b}}=0\,,\til{\mathbb{P}}-a.s..\]
 and (iv) follows by the uniform integrability.
\end{trivlist}
\end{proof}
After the above preparation, now we can finish the Step 2 by the following Lemma.
\begin{lem}\label{lem:xi'}
  For each $t\in [0,T]$, we have
  \[\til{\xi}(t)=\sum_{j=1}^\infty\int_0^tG_j^\psi\la\til M\ra\ud \til{W}_j(s),\]
  in $L^2(\til{\Omega};\mathbb X^{-b})$.
\end{lem}
 \begin{proof}
Firstly, we show that
\[\til{\xi}_n(t)=\sum_{j=1}^\infty\int_0^tG_{jn}\la\til M\ra\ud \til{W}_{jn}(s)\]
$\til{\mathbb{P}}$  almost surely for each $t\in[0,T]$ and $n\in\mathbb{N}$.\\
      Let us fix that $t\in[0,T]$ and $n\in\mathbb{N}$. For each $m\in\mathbb{N}$ we define the partition $\left\{s_i^m:=\frac{iT}{m},i=0,\ldots,m\right\}$ of $[0,T]$. By Lemma \ref{prop:tilMn} and Propositon \ref{prop:LLCn}, $(\til{M}_n,\til B_n,\til E_n \til{W}_{hn})$ and $(M_n, B_n,E_n,W_h)$ have same distribution on $C([0,T];\mathbb{H}_n)\times C([0,T];\mathbb{Y}_n)\times C([0,T];\mathbb{Y}_n)\times C\la [0,T];\HH\ra$, so for each $m$,  the $\mathbb{H}$-valued random variables:
      \[\xi_n(t)-\sum_{j=1}^\infty\sum_{i=0}^{m-1}G_{jn}\la M_n\la s_i^m\ra\ra
      \la W_j\la t\wedge s_{i+1}^m\ra-W_j\la t\wedge s_i^m\ra\ra\]
   and
    \[\til\xi_n(t)-\sum_{j=1}^\infty\sum_{i=0}^{m-1}G_{jn}\la \til M_n\la s_i^m\ra\ra
      \la \til W_j\la t\wedge s_{i+1}^m\ra-\til W_j\la t\wedge s_i^m\ra\ra\]
      have the same distribution. For each $j$, we have
      \[\lim_{m\to\infty}\E\left\|\sum_{i=0}^{m-1}G_{jn}\la M_n\la s_i^m\ra\ra
      \la W_j\la t\wedge s_{i+1}^m\ra-W_j\la (t\wedge s_i^m\ra\ra-\int_0^tG_{jn}\la M_n(s)\ra \ud W_j(s)\right\|^2_{\mathbb H}=0\]
      and
\[\lim_{m\to\infty}\E\left\|\sum_{i=0}^{m-1}  G_{jn}\la \til M_n\la s_i^m\ra\ra
      \la\til W_j\la t\wedge s_{i+1}^m\ra-\til W_j\la (t\wedge s_i^m\ra\ra -
      \int_0^tG_{jn}\la\til M_n(s)\ra\ud\til W_{jn}(s)\right\|^2_{\mathbb H}=0,\]
      so
      \[\int_0^tG_{jn}\la M_n(s)\ra\ud W_{j}(s)\quad\mathrm{and}\quad \int_0^tG_{jn}\la\til M_n(s)\ra\ud\til W_{jn}(s)\]
      have the same distribution. Hence
  \[\sum_{j=1}^\infty\int_0^tG_{jn}\la M_n(s)\ra\ud W_{j}(s)\quad\mathrm{and}\quad \sum_{j=1}^\infty\int_0^tG_{jn}\la\til M_n(s)\ra\ud\til W_{jn}(s)\]
 have the same distribution. Therefore
      \[\xi_n(t)-\sum_{j=1}^\infty\int_0^tG_{jn}\la M_n(s)\ra\ud W_{j}(s)\]
      and
      \[\til{\xi}_n(t)-\sum_{j=1}^\infty\int_0^tG_{jn}\la \til M_n(s)\ra\ud \til{W}_{jn}(s)\]
have the same distribution. But
\[\xi_n(t)=\sum_{j=1}^\infty\int_0^t G_{jn}\la M_n(s)\ra\ud W_{j}(s),\quad \mathbb{P}-a.s.\]
 and thereby
\[\til{\xi}_n(t)=\sum_{j=1}^\infty\int_0^tG_{jn}\la \til M_n(s)\ra\ud \til{W}_{jn}(s),\quad \til{\mathbb{P}}-a.s.\]
 We will show that $\til{\xi}_n(t)$ converges in $L^2(\til{\Omega};\mathbb X^{-b})$ to
  $$\sum_{j=1}^\infty\int_0^tG_j^\psi\la\til M(s)\ra\ud \til{W}_j(s)$$
   as $n\rightarrow\infty$. Indeed, using notation \eqref{sigma1}, and \eqref{sigma2} we obtain for a certain $C>0$
  \[\begin{aligned}
  &C\,\til{\mathbb{E}}\left\|\til{\xi}_n(t)-\sum_{j=1}^\infty\int_0^tG_j^\psi\la\til M(s)\ra\ud \til{W}_j(s)\right\|_{\mathbb X^{-b}}^2\\
 =&\til{\mathbb{E}}\,\Bigg\|\sum_{j=1}^\infty\int_0^t\sigma_{jn}(s)\ud \til{W}_{jn}(s)-\sum_{j=1}^\infty\int_0^tG_j^\psi\la\til M(s)\ra\ud \til{W}_j(s)\Bigg\|_{\mathbb X^{-b}}^2\\
  \leq&\til{\mathbb{E}}\,\left\|\sum_{j=1}^\infty\int_0^t\la \sigma_{jn}(s)-\sigma_{jn}^m(s)\ra \ud \til{W}_{jn}(s)\right\|_{\mathbb X^{-b}}^2\\
&\,\,\,\,\,\,+\til{\mathbb{E}}\,\left\|\sum_{j=1}^\infty\int_0^t\sigma^m_{jn}(s)\ud \til{W}_{jn}(s)-\sum_{j=1}^\infty\int_0^t\sigma^m_{jn}(s)\ud \til{W}_{j}(s)\right\|_{\mathbb X^{-b}}^2\\
&\,\,\,\,\,\,+\til{\mathbb{E}}\,\left\|\sum_{j=1}^\infty\int_0^t\la\sigma^m_{jn}(s)-\sigma_{jn}(s)\ra\ud \til{W}_{j}(s)\right\|_{\mathbb X^{-b}}^2
+\til{\mathbb{E}}\,\left\|\sum_{j=1}^\infty\int_0^t\la\sigma_{jn}(s)-G^\psi_j\la\til M(s)\ra\ra\ud \til{W}_{j}(s)\right\|_{\mathbb X^{-b}}^2,\\
 \end{aligned}\]
 and invoking Lemma \ref{lem:4 ineqM}, we conclude the proof.
\end{proof}
\begin{cor}
  \setlength\arraycolsep{2pt}{
  \begin{eqnarray}\label{eq:Meqwithpsi}
\til{M}(t)&=&M_0+\int_0^t\Bigg\{\left[\lmd_1 \til{M}\times \til{\rho}\right]-\lmd_2 \left[\til{M}\times(\til{M}\times \til{\rho})\right]\nonumber+\frac{1}{2}\sum_{j=1}^\infty\la G^\psi_j\la\til M\ra\ra^\prime\left[G^\psi_j\la\til M\ra\right]\Bigg\}\ud s
\nonumber\\
  &&+\sum_{j=1}^\infty\int_0^tG^\psi_j\la\til M\ra\ud \til{W}_j(s),
  \end{eqnarray}}
  in $L^2(\til{\Omega};\mathbb X^{-b})$.
\end{cor}
\begin{proof}
The corollary follows immediately from Lemma \ref{lem:xi't} and Lemma \ref{lem:xi'}.
\end{proof}

\subsection{Step 3}
In order to get rid of the auxiliary function $\psi$ from the equation \eqref{eq:Meqwithpsi} and finish the proof of the existence of the weak solution,
now we  need to prove  the constraint condition of $\til M$, i.e. condition (iii) of the main Theorem \ref{thm:mainthm}.
\begin{lem}\label{lem:M=1}
  Let $\til{M}$ be a process defined in Lemma \ref{prop:tilMn}. Then for each $t\in [0,T]$, we have $\til{\mathbb{P}}$-almost surely
  \begin{equation}\label{eq:M=1}
    |\til{M}(t,x)|=1,\quad\textrm{ for a.e. }x\in \cl D\,.
  \end{equation}
\end{lem}
\begin{proof}
We will use a version of the It\^o formula proved in Pardoux's paper \cite{Pardoux}, see Lemma \ref{lem:Pardoux}.
  Let $\eta\in C_0^\infty(D,\R)$ and let $\gamma$ denote a function
  \[\gamma:\mathbb{H}\ni M\longmapsto \llangle M,\eta M\rrangle_{\mathbb{H}}\in \R\,.\]
 Then  $\gamma\in C^2\mathbb H)$, $\gamma^\prime(M)=2\eta M$ and $\gamma^{\prime\prime}(M)(v)=2\eta v$ for $M,v\in \mathbb{H}$.\\
In view of  definition of the problem and  \eqref{eq:SVis 3.34R'}, \eqref{eq:SVis 3.35R'} and \eqref{eq:SLLG 3.14R'}, all the assumptions of Lemma \ref{lem:Pardoux} are satisfied. Therefore,

 Lemma \ref{lem:Pardoux} yields for  $t\in[0,T]$ $\til{\mathbb{P}}$-a.s.
  \setlength\arraycolsep{2pt}{
  \begin{eqnarray*}
    &&\llangle \til{M}(t),\eta \til{M}(t)\rrangle_{\mathbb{H}}-\llangle M_0,\eta M_0\rrangle_{\mathbb{H}}\\
    &=&\int_0^t{}_{\mathbb X^{-b}}\llangle \lmd_1\til{M}\times \til{\rho}-\lmd_2\til{M}\times(\til{M}\times\til{\rho})+\frac{1}{2}\sum_{j=1}^\infty\gj,2\eta \til{M}(s)\rrangle_{X^b}\ud s\\
    &&+\sum_{j=1}^\infty\int_0^t\llangle 2\eta \til{M}(s),G^\psi_j\la\til M\ra\rrangle_{\mathbb{H}}\ud \til{W}_j(s)
   +\sum_{j=1}^\infty\int_0^t\llangle \eta G_j^\psi\la\til M\ra,G_j^\psi\la\til M\ra\rrangle_{\mathbb{H}}\ud s=0.
  \end{eqnarray*}}
  Hence we have
  \setlength\arraycolsep{2pt}{
  \begin{eqnarray*}
    &&\llangle \eta,|\til{M}(t)|^2-|\til{M}_0|^2\rrangle_{L^2(\cl D;\R)}=\llangle \til{M}(t),\eta \til{M}(t)\rrangle_{\mathbb{H}}-\llangle \til{M}_0,\eta \til{M}_0\rrangle_{\mathbb{H}}=0.
  \end{eqnarray*}}
  Since $\eta$ is arbitrary and $|M_0(x)|=1$ for almost every $x\in \cl D$, we infer that $|\til{M}(t,x)|=1$ for almost every $x\in \cl D$ as well.
\end{proof}

 Note that, if
$\left|\til M(t,x)\right|=1$ then $\psi\la \til M(t,x)\ra=1$, so we can get rid of it, which means that now we have the following equalities:
\[G^\psi_j\la\til M(t,x)\ra=G_j\la\til M(t,x)\ra,\quad \la G^\psi_j\ra^\prime\la\til M(t,x)\ra\left[G^\psi_j\la\til M(t,x)\ra\right]=G_j^\prime\la\til M(t,x)\ra\left[G_j\la\til M(t,x)\ra\right].\]
Hence we have the following result.
\begin{lem}\label{thm:solSDE}
  The process $(\til{M},\til{E},\til{B})$ is a weak martingale solution of Problem \ref{S-LLG}, that is, $(\til{M},\til{E},\til{B})$ satisfies \eqref{eq:Vis 3.5}, \eqref{eq:Vis 3.7} and \eqref{eq:Vis 3.6}.
\end{lem}
\begin{proof}[Proof of \eqref{eq:Vis 3.5}]
By Lemma \ref{lem:xi'} and Lemma \ref{lem:M=1}, we have $\psi(\til{M}(t))\equiv 1$ for $t\in[0,T]$. Hence we deduce that for $t\in [0,T]$, the following equation holds in $L^2(\til{\Omega};\mathbb X^{-b})$.
  \setlength\arraycolsep{2pt}{
  \begin{eqnarray*}
\til{M}(t)&=&M_0+\int_0^t\Bigg\{\left[\lmd_1 \til{M}\times \til{\rho}\right]-\lmd_2 \left[\til{M}\times(\til{M}\times \til{\rho})\right]+\frac{1}{2}\sum_{j=1}^\infty G_j^\prime\la\til M\ra G_j\la\til M\ra\Bigg\}\ud s\nonumber\\
  &&+\sum_{j=1}^\infty\int_0^t G_j\la\til M\ra\ud \til{W}_j(s)\\
  &=&M_0+\int_0^t\left[\lmd_1 \til{M}\times \til{\rho}-\lmd_2 \til{M}\times(\til{M}\times \til{\rho})\right]\ud s+\sum_{j=1}^\infty\left\{\int_0^tG_j\la\til M\ra\circ\ud \til{W}_j(s)\right\}.
  \end{eqnarray*}}
Then \eqref{eq:Vis 3.5} follows from our explaination of $\til M\times \til\rho$, see Lemma \ref{lem:4.5M}.
\end{proof}

\begin{proof}[Proof of \eqref{eq:Vis 3.7}]
By Lemma \ref{prop:tilMn} and the equation \eqref{eq:SVis 3.29R}, we have

\begin{equation}\label{eq:Bn't}
   \til{B}_n(t)-\til{B}_n(0)=-\int_0^t\pi_n^\mathbb{Y}[\nabla\times \til{E}_n(s)]\ud s,\qquad \til{\mathbb{P}}-a.s.
  \end{equation}

We also have
\begin{trivlist}
    \item[(a)] $\til{E}_{n}\longrightarrow \til{E}$ in $L^2_w(0,T;\mathbb{L}^2(\R^3))$ $\til{\mathbb{P}}$ almost surely, and
    \item[(b)] $\til{B}_{n}\longrightarrow \til{B}$ in $L^2_w(0,T;\mathbb{L}^2(\R^3))$ $\til{\mathbb{P}}$ almost surely.
  \end{trivlist}

Hence for any $u\in H^1(0,T;\mathbb{Y})$,
\setlength\arraycolsep{2pt}{
  \begin{eqnarray*}
    &&\int_0^t\llangle \til{B}(s),\frac{\ud u(s)}{\ud s}\rrangle_{\mathbb{L}^2(\R^3)}\ud s=\lim_{n\rightarrow\infty}\int_0^t\llangle \til{B}_n(s),\frac{\ud u(s)}{\ud s}\rrangle_{\mathbb{L}^2(\R^3)}\ud s\\
    &=&-\lim_{n\rightarrow\infty}\int_0^t\llangle \frac{\ud \til{B}_n(s)}{\ud s},u(s)\rrangle_{\mathbb{L}^2(\R^3)}\ud s=\lim_{n\rightarrow\infty}\int_0^t\llangle\pi_n^\mathbb{Y}\big[\nabla\times \til{E}_n(s)\big],u(s)\rrangle_{\mathbb{L}^2(\R^3)}\ud s\\
    &=&\lim_{n\rightarrow\infty}\int_0^t\llangle \nabla\times \til{E}_n(s),\pi_n^\mathbb{Y}u(s)\rrangle_{\mathbb{L}^2(\R^3)}\ud s=\lim_{n\rightarrow\infty}\int_0^t\llangle  \til{E}_n(s),\nabla\times\pi_n^\mathbb{Y}u(s)\rrangle_{\mathbb{L}^2(\R^3)}\ud s
  \end{eqnarray*}}

Since
\setlength\arraycolsep{2pt}{
  \begin{eqnarray*}
    &&\lim_{n\rightarrow\infty}\left|\int_0^t\llangle  \til{E}_n(s),\nabla\times\pi_n^\mathbb{Y}u(s)\rrangle_{\mathbb{L}^2(\R^3)}-\llangle  \til{E}(s),\nabla\times u(s)\rrangle_{\mathbb{L}^2(\R^3)}\ud s\right|\\
    &\leq&\lim_{n\rightarrow\infty}\int_0^t\left|\llangle \til{E}_n(s),\nabla\times\left(\pi_n^\mathbb{Y}u(s)-u(s)\right)\rrangle_{\mathbb{L}^2(\R^3)}\right|\ud s+\lim_{n\rightarrow\infty}\left|\int_0^t\llangle \til{E}_n(s)-\til{E}(s),\nabla\times u(s)\rrangle_{\mathbb{L}^2(\R^3)}\ud s\right|\\
    &\leq&\lim_{n\rightarrow\infty}\left(\int_0^t\|\til{E}_n(s)\|_{\mathbb{L}^2(\R^3)}^2\ud s\right)^\frac{1}{2}\left(\int_0^t\|\pi_n^\mathbb{Y}u(s)-u(s)\|^2_{\mathbb{Y}}\ud s\right)^\frac{1}{2}+0=0,\qquad \til{\mathbb{P}}-a.s.,
  \end{eqnarray*}}

we have
\[\lim_{n\rightarrow\infty}\int_0^t\llangle  \til{E}_n(s),\nabla\times\pi_n^\mathbb{Y}u(s)\rrangle_{\mathbb{L}^2(\R^3)}\ud s=\int_0^t\llangle  \til{E}(s),\nabla\times u(s)\rrangle_{\mathbb{L}^2(\R^3)}\ud s.\]
Therefore
\[\int_0^t\llangle \til{B}(s),\frac{\ud u(s)}{\ud s}\rrangle_{\mathbb{L}^2(\R^3)}\ud s=\int_0^t\llangle  \til{E}(s),\nabla\times u(s)\rrangle_{\mathbb{L}^2(\R^3)}\ud s,\]
for all $u\in H^1(0,T;\mathbb{Y})$.

  Hence for $t\in[0,T]$,
  \[\til{B}(t)=B_0-\int_0^t\nabla\times \til{E}(s)\ud s,\quad\in \mathbb{Y}^*,\;\til{\mathbb{P}}-a.s..\]
 \end{proof}

\begin{proof}[Proof of \eqref{eq:Vis 3.6}]
  Similar as in the proof of \eqref{eq:Vis 3.7}.
  Let $p=q=2$ in Lemma \ref{prop:tilMn}, we have
\begin{trivlist}
    \item[(a)] $\til M_{n}\longrightarrow \til{M}$ in $L^2(0,T;\mathbb{H})$ $\til{\mathbb{P}}$ almost surely,
    \item[(b)] $\til{E}_{n}\longrightarrow \til{E}$ in $L^2_w(0,T;\mathbb{L}^2(\R^3))$ $\til{\mathbb{P}}$ almost surely, and
    \item[(c)] $\til{B}_{n}\longrightarrow \til{B}$ in $L^2_w(0,T;\mathbb{L}^2(\R^3))$ $\til{\mathbb{P}}$ almost surely.
  \end{trivlist}
  Hence by \eqref{eq:SVis 3.29R} we have for all $u\in H^1(0,T;\mathbb{Y})$,
  \setlength\arraycolsep{2pt}{
  \begin{eqnarray*}
    &&\int_0^t\llangle \til{E}(s),\frac{\ud u(s)}{\ud s}\rrangle_{\mathbb{L}^{2}(\R^3)}\ud s=\lim_{n\rightarrow\infty}\int_0^t\llangle \til{E}_n(s),\frac{\ud u(s)}{\ud s}\rrangle_{\mathbb{L}^{2}(\R^3)}\ud s\\
    &=&-\lim_{n\rightarrow\infty}\int_0^t\llangle \pi_n^\mathbb{Y}[1_{\cl D}(\til{E}_n(s))+f(s)]-\pi_n^\mathbb{Y}[\nabla\times( \til{B}_n(s)-\pi_n^\mathbb{Y}(\overline{\til{M}}_n(s)))],u(s)\rrangle_{\mathbb{L}^{2}(\R^3)}\ud s\\
    &=&\int_0^t\llangle \til{B}(s)-\overline{\til{M}}(s),\nabla\times u(s)\rrangle_{\mathbb{L}^{2}(\R^3)}-\llangle 1_{\cl D}\til{E}(s)+f(s),u(s)\rrangle_{\mathbb{L}^{2}(\R^3)}\ud s.
   \end{eqnarray*}}

  Hence for $t\in[0,T]$,
  \[\til{E}(t)=E_0+\int_0^t\nabla\times[\til{B}(s)-\overline{\til{M}}(s)]\ud s-\int_0^t[1_{\cl D}\til{E}(s)+\overline{f}(s)]\ud s,\quad\in \mathbb{Y}^*,\;\til{\mathbb{P}}-a.s.\]
  
  Therefore the proof of Lemma \ref{thm:solSDE} is complete. 
\end{proof}

\section{Some further regularities of the weak solution}
Next we will show some further regularity of $\til{M}$.
\begin{lem}\label{thm:MinH5}
For $t\in[0,T]$ the following equation holds in $L^2(\til{\Omega};\mathbb{H})$.
  \setlength\arraycolsep{2pt}{
  \begin{eqnarray}
\til{M}(t)&=&M_0+\int_0^t\Bigg\{\left[\lmd_1 \til{M}\times \til{\rho}\right]-\lmd_2 \left[\til{M}\times(\til{M}\times \til{\rho})\right]\nonumber\\
  &&+\frac{1}{2}\sum_{j=1}^\infty G_{j}^\prime\la \til M\ra\left[G_{j}\la \til M\ra\right]\Bigg\}\ud s+\sum_{j=1}^\infty G_{j}\la\til M\ra\ud \til{W}_j(s)\label{eq:tilML2}\\
  &=&M_0+\int_0^t\left\{\lmd_1 \til{M}\times \til{\rho}-\lmd_2 \til{M}\times(\til{M}\times \til{\rho})\right\}\ud s+\sum_{j=1}^\infty\left\{\int_0^tG_j\la\til M\ra\circ\ud \til{W}_j(s)\right\}.\nonumber
  \end{eqnarray}}
\end{lem}

\begin{proof}
	We will only show the following two terms of \eqref{eq:tilML2} are in $L^2(\til{\Omega};\mathbb{H})$, the other terms can be dealt with similarly.
	
	Firstly, we consider the term $\int_0^t\til M\times(\til M\times \til \rho)\ud s$. Making use of Jensen's inequality, \eqref{eq:M=1} and Remark \ref{rem:MxrhoinL2}, we have
	\setlength\arraycolsep{2pt}{
		\begin{eqnarray*}
			&&\left\|\int_0^t \til M\times(\til M\times \til \rho)\ud s\right\|_{L^2(\til \Omega;\HH)}^2=\til\E \left\|\int_0^t \til M\times(\til M\times \til \rho)\ud s\right\|_\HH^2\\
			&\le&C\til\E \int_0^t\int_{\cl D} |\til M\times(\til M\times \til \rho)|^2\ud x\ud s\le C\til{\mathbb{E}}\|\til{M}\times \til{\rho}\|^2_{L^2(0,T;\mathbb{H})}<\infty.
	\end{eqnarray*}}
	So for all $t\in[0,T]$, $\int_0^t\til M\times(\til M\times \til \rho)\ud s\in L^2(\til \Omega;\HH)$.
	
	Secondly, we consider the term $\sum_{j=1}^\infty\int_0^t\til M\times (\til M\times h_j)\ud W_j$. Making use of Burkholder-Davis-Gundy inequality, Jensen's inequality, \eqref{eq:M=1} and our assumption on $h_j$, we have
	\setlength\arraycolsep{2pt}{
		\begin{eqnarray*}
			&&\left\|\sum_{j=1}^\infty\int_0^t\til M\times (\til M\times h_j)\ud W_j\right\|_{L^2(\til \Omega;\HH)}\le \sum_{j=1}^\infty\left\|\int_0^t\til M\times (\til M\times h_j)\ud W_j\right\|_{L^2(\til \Omega;\HH)}\\
			&=&\sum_{j=1}^\infty\left(\til\E\int_{\cl D}\left|\int_0^t\til M\times(\til M\times h_j)\ud W_j\right|^2\ud x\right)^\frac{1}{2}\le C\sum_{j=1}^\infty\left(\int_{\cl D}\til\E\int_0^t|\til M\times(\til M\times h_j)|^2\ud s\ud x\right)^\frac{1}{2}\\
			&\le&C\sum_{j=1}^\infty\|h_j\|_\HH<\infty.
	\end{eqnarray*}}
	So for all $t\in[0,T]$, $\sum_{j=1}^\infty\int_0^t\til M\times (\til M\times h_j)\ud W_j\in L^2(\til \Omega;\HH)$.
	The proof is complete.
\end{proof}

\begin{lem}\label{thm:tilMregt}
  The process $\til{M}$ introduced in Lemma \ref{prop:tilMn} satisfies: 
  \[\til{M}\in C^\theta(0,T;\mathbb{H}),\qquad  \til{\mathbb{P}}-a.s.,\;\theta\in [0,\frac{1}{2}).\]
\end{lem}
\begin{proof}
By Lemma \ref{thm:MinH5}, we have
\[\begin{aligned}
\til{M}(t)-\til{M}(s)=&\int_s^t\Bigg\{\lmd_1 \til{M}\times \til{\rho}-\lmd_2\til{M}\times(\til{M}\times \til{\rho})+\frac{1}{2}\sum_{j=1}^\infty G_{j}^\prime\la \til M\ra\left[G_{j}\la \til M\ra\right]\Bigg\}\ud \tau\\
 &+\sum_{j=1}^\infty\int_s^tG_{j}\la \til M\ra\ud \til{W}_j(\tau)\\
 &=\int_s^t F(s)\ud s+\sum_{j=1}^\infty\int_s^tG_{j}\la \til M\ra\ud \til{W}_j(\tau)
  \end{aligned}\]
  for $0\leq s<t\leq T$. By the constraint \eqref{eq:M=1}, the estimate \eqref{eq:S 4.10} and the regularity of $\til M\times\til \rho$ as in  \eqref{eq:Z1w}, we have
  \[\int_0^T\|F(t)\|_{\mathbb H}^2\ud t<\infty,\quad \til{\mathbb P}-a.s.,\]
hence the process
\[V(t)= \int_0^t F(s)\ud s\in\mathbb H\]
has trajectories in $C^{1/2}([0,T];\mathbb{H})$. Putting
\[N(t)=\sum_{j=1}^\infty\int_0^tG_{j}\la \til M\ra\ud \til{W}_j(\tau)\]
and invoking the Burkholder-Davis-Gundy inequalities we obtain for any $p\ge 2$
\[\E\left\|N(t)-N(s)\right\|^{2p}\le C_p\E\la\sum_{j=1}^\infty\int_s^t\left\|G_j\la\til M\ra\right\|^2_{\mathbb H}\ud \tau\ra^p\le C\la\sum_{j=1}^\infty\left\|h_j\right\|^2_{\mathbb L^\infty}\ra^p(t-s)^p\,.\]
Then the Kolmogorov continuity test, see Lemma \ref{lem:Kol con}, yields
  \[N\in C^\theta([0,T];\mathbb{H}), \quad \theta\in \left(0,\frac{1}{2}\right).\]
  since $\til M(t)=V(t)+N(t)$, the lemma follows.
\end{proof}

We can also prove that $\Delta \til M\in L^1(\til \Omega;L^1(0,T;\bb L^1))$. To do that, we need the following Corollary of Lemma \ref{lem:M=1}.
\begin{cor}
	\begin{equation}\label{eq:MperpnablaiM}
		\nabla_i\til M(t,x)\perp \til M(t,x),\qquad a.e. (t,x)\in[0,T]\times \cl D,\,\til{\bb P}-a.s.,\,i=1,2,3.
	\end{equation}
\end{cor}
\begin{proof}
	By the equation \eqref{eq:M=1} and by the chain rule of weak derivatives, we have
	\[0=\frac{1}{2}\nabla_i|\til M(t,x)|^2=\llangle \nabla_i \til M(t,x),\til M(t,x)\rrangle,\]
	for almost every $(t,x)\in[0,T]\times \cl D$, $\til{\bb P}$-almost surely and $i=1,2,3.$ Hence the proof is complete.
\end{proof}

We will also need the following results. 
\begin{prop}\label{prop:MxD2MnCon}
	\begin{equation}\label{eq:MxD2MnCon}
		\til M\times \Delta \til M_n\longrightarrow \til M\times \Delta \til M \textrm{ weakly star in } L^2(\til\Omega;L^2(0,T;(\V\cap\bb L^\infty)^*)) \textrm{ as }n\rightarrow \infty.
	\end{equation}
\end{prop}
\begin{proof}
	For any $u\in L^2(\til\Omega;L^2(0,T;(\V\cap\bb L^\infty)))$, note that by the equation \eqref{eq:M=1} and \eqref{eq:S 4.12}, we also have $\til M\times u\in L^2(\til\Omega;L^2(0,T;(\V\cap\bb L^\infty)))$. Hence by \eqref{eq:4.17phi} we have
	 \setlength\arraycolsep{2pt}{
		\begin{eqnarray*}
			&&\left|\til\E\int_0^T\llangle \til M\times \Delta \til M_n-\til M\times \Delta \til M,u\rrangle_\HH\ud t\right|\\
			&=&\left|\til\E\int_0^T {}_{(\V\cap\bb L^\infty)^*}\Big\langle\Delta \til M_n- \Delta \til M,\til M\times u\Big\rangle_{\V\cap\bb L^\infty}\ud t\right|\\
			&\le&\sum_{i=1}^3\til\E\int_0^T\left|\llangle \nabla_i\til M_n-\nabla_i \til M,\nabla_i\til M\times u+\til M\times\nabla_i u\rrangle_\HH\right|\longrightarrow 0,\qquad \textrm{as }n\rightarrow \infty.
		\end{eqnarray*}}
	The proof is complete.
\end{proof}

\begin{prop}
	We have 
	\begin{equation}\label{eq:MxMxD2M=}
		-M\times (M\times\Delta M)=M\sum_{i=1}^3|\nabla_i M|^2+\Delta M\qquad\textrm{in the space}\quad L^2(\til\Omega;L^2(0,T;(\V\cap\bb L^\infty)^*)).
	\end{equation}
\end{prop}
\begin{proof}
	Let us arbitrarily fix $u\in L^2(\til\Omega;L^2(0,T;(\V\cap\bb L^\infty)))$. Similar as in the proof of Proposition \ref{prop:MxD2MnCon}, we also have $\til M\times u\in L^2(\til\Omega;L^2(0,T;(\V\cap\bb L^\infty)))$. By \eqref{eq:MxD2MnCon}, \eqref{eq:MperpnablaiM} and by \eqref{eq:4.17phi} we have the following equality:
	\setlength\arraycolsep{2pt}{
		\begin{eqnarray*}
			&&{}_{L^2(\til\Omega;L^2(0,T;(\V\cap\bb L^\infty)^*))}\llangle -\til M\times (\til M\times \Delta \til M),u\rrangle_{L^2(\til\Omega;L^2(0,T;(\V\cap\bb L^\infty)))}\\
			&=&\lim_{n\rightarrow\infty}\llangle \til M\times \Delta \til M_n,\til M\times u\rrangle_{L^2}\\
			&=&\lim_{n\rightarrow\infty}\sum_{i=1}^3\llangle\nabla_i\til M_n,\nabla_i \til M\times(\til M\times u)+\til M\times(\nabla_i\til M\times u)+\til M\times(\til M\times\nabla_i u)\rrangle_{L^2}\\
			&=&\lim_{n\rightarrow\infty}\sum_{i=1}^3\llangle \nabla_i\til M_n,\til M\langle\nabla_i\til M,u\rangle+\nabla_i\til M\langle \til M,u\rangle+\til M\langle \til M,\nabla_i u\rangle-\nabla_i u\rrangle_{L^2}\\
			&=&\sum_{i=1}^3\llangle\nabla_i\til M,\nabla_i \til M\langle \til M,u\rangle-\nabla_i u\rrangle_{L^2}={}_{L^2(\til\Omega;L^2(0,T;(\V\cap\bb L^\infty)^*))}\llangle \til M\sum_{i=1}^3|\nabla_i \til M|^2+\Delta \til M,u\rrangle_{L^2(\til\Omega;L^2(0,T;(\V\cap\bb L^\infty)))}
	    \end{eqnarray*}}
    The proof is complete.
\end{proof}

\begin{lem}\label{lem:D2MinL1}
	We have the following regularity result about $\Delta \til M$, 
	\begin{equation}
		\Delta \til M\in L^1(\til \Omega;L^1(0,T;\bb L^1)).
	\end{equation}
\end{lem}
\begin{proof}
	By the proof of Lemma \ref{thm:MinH5}, we have $\til M\times(\til M\times\Delta \til M)\in L^2(\til\Omega;L^2(0,T;\HH))$. And by \eqref{eq:S 4.12}, it is easy to check that
	$\til M\sum_{i=1}^3|\nabla_i \til M|^2\in  L^1(\til\Omega;L^1(0,T;\bb L^1))$. Hence by \eqref{eq:MxMxD2M=}, we have $\Delta \til M\in L^1(\til \Omega;L^1(0,T;\bb L^1))$. 
\end{proof}

\section{Proof of the main result}
Finally we are ready to finish the proof the main result (Theorem \ref{thm:mainthm}) of this paper.
\begin{proof}[Proof of Theorem \ref{thm:mainthm}]
	\begin{trivlist}
		\item[(i)] The results follows from Lemma \ref{thm:MregV}, Lemma \ref{lem:BEreg}, Lemma \ref{lem:4.5M} and Lemma \ref{lem:D2MinL1}.
		\item[(ii)] The results follows from Lemma \ref{thm:MinH5} and Lemma \ref{thm:solSDE}.
		\item[(iii)] The result follows from Lemma \ref{lem:M=1}.
		\item[(iv)] The result follows from Lemma \ref{thm:tilMregt}.
	\end{trivlist}
\end{proof}

\section{Acknowledgments}	
Zdzis{\l}aw Brze\'zniak was partially supported by the Australian Research Council Discovery Project DP200101866 and partially supported by  National Natural Science Foundation of China project project 12071433.\\
Beniamin Goldys was partially supported by the Australian Research Council Discovery Project DP200101866.\\
Liang Li was partially supported by  National Natural Science Foundation of China project 11901026 and project 12071433 and project 12171032 and Fundamental Research Funds for the Central Universities project ZY1913.

\section{Appendix: Proofs of some auxiliary results}
Here are some auxiliary results and their proofs. 

\begin{prop}\label{prop:sumWjhjCh}
	The sequence $\{\sum_{j=1}^n W_jh_j\}_n$ introduced in the Remark \ref{rem:QW} is a Cauchy sequence in $L^2(\Omega; C([0,T];\HH))$.
\end{prop}
\begin{proof}
		Let us introduce the notation $a_n:=\{\sum_{j=1}^nW_jh_j\}_n$, and consider $a_{n+k}-a_n$.
	\begin{align*}	
		&\|a_{n+k}-a_n\|_{L^2(\Omega;C([0,T];\HH))}
		=\left\|\sum_{j=n+1}^{n+k}W_jh_j\right\|_{L^2(\Omega;C([0,T];\HH))}\\
		\le&\sum_{j=n+1}^{n+k}\left\|W_jh_j\right\|_{L^2(\Omega;C([0,T];\HH))}
		= \sum_{j=n+1}^{n+k}\left[\E\left(\sup_{t\in[0,T]}\|W_j(t)h_j\|_\HH^2\right)\right]^\frac{1}{2}.
	\end{align*}
	By the Doob's maximal inequality, we have
	\[\E\left(\sup_{t\in[0,T]}\|W_j(t)h_j\|_\HH^2\right) \le 4T\|h_j\|_\HH^2.\]
	Hence we have
	\[\|a_{n+k}-a_n\|_{L^2(\Omega;C([0,T];\HH))}\le 2\sqrt T\sum_{j=n+1}^{n+k}\|h_j\|_\HH.\]
	Therefore by our assumption on $h_j$ in the statement of Problem 2.7 (the equation (2.11)), $\{a_n\}$ is Cauchy in the space $L^2(\Omega;C([0,T];\HH))$.
	\end{proof}

The next result used in the proof of Lemma \ref{lem:StermWanorm}.
\begin{lem}\label{lem:FG2.1}\cite{Flandoli}
	Let $p\ge 2$ and $a\in[0,\frac{1}{2})$ be given. There exists a constant $C(p,a)>0$ such that for any progressively measurable process $\xi=\sum_{j=1}^\infty\xi_j\in L^p(\Omega\times[0,T];\HH)$ with $\sum_{j=1}^\infty\|\xi_j\|_\HH^2<\infty$, we have
	\[\E\left\|\sum_{j=1}^\infty\int_0^\cdot\xi_j(t)\ud W_j(t)\right\|^p_{W^{a,p}(0,T;\HH)}\le C(p,a)\E\int_0^T\left(\sum_{j=1}^\infty\left\|\xi_j(t)\right\|^2_\HH\right)^\frac{p}{2}\ud t.\]
\end{lem}

The next two results used in the proof of Lemma \ref{lem:LMntight}.
\begin{lem}\cite{Flandoli}\label{lem:compact emb 5}
  Let $B_0\subset B\subset B_1$ be Banach spaces, $B_0$ and $B_1$ being reflexive and the embedding $B_0\hookrightarrow B$ to be compact. Let $p\in (1,\infty)$ and $a\in (0,1)$ be given. Then the embedding
  \[L^p(0,T;B_0)\cap W^{a,p}(0,T;B_1)\hookrightarrow L^p(0,T;B)\]
  is compact.
\end{lem}

\begin{lem}\cite{Flandoli}\label{lem:compact emb 6}
  Assume that $B_1\subset B_2$ are two Banach spaces with compact embedding, and $a\in(0,1)$, $p>1$ satisfying $a>\frac{1}{p}$. Then the space $W^{a,p}(0,T;B_1)$ is compactly embedded into $C([0,T];B_2).$
\end{lem}

	The Lemma \ref{lem:UcomebY} and Definition \ref{def:Aldous} and Lemma \ref{lem:tightcri} are used in the proof of Lemma \ref{lem:LEBntight}.
\begin{lem}\label{lem:UcomebY}
	The Hilbert space 
	$$U=H^2(\R^3;\R^3)\cap \tilde{\mathbb Y}$$
	is compactly and densely embedded into $\bb Y$, where the weighted space $\tilde{\mathbb Y}$ is defined as
	\[\tilde{\mathbb Y}:=\left\{f\in \bb Y: \int_{\R^3}|f(x)|^2|x|^2\ud x+\int_{\R^3}|\nabla\times f(x)|^2|x|^2\ud x<\infty\right\}.\]
\end{lem}
\begin{proof}[Proof of that $U$  is compactly and densely embedded into $\bb Y$]
	The embedding is dense follows from that $C_0^\infty(\R^3;\R^3)$ is dense in $\bb Y$. So we only need to prove the embedding is compact.
	
	Let us take an arbitrary bounded set $V\subset U$. We want to show that for any $\eps>0$, there exists $n(\eps)\in \bb N$ and $y_1,\ldots, y_{n(\eps)}\in \bb Y$ such that 
	$$V\subset \bigcup_{i=1}^{n(\eps)}B_\eps(y_i).$$
	
	Now let us arbitrary fix an $\eps>0$, we claim that there exists $R>0$, such that for all $v\in V$, we have 
	\[\int_{B_R^c}\left(|v(x)|^2+|\nabla\times v(x)|^2\right)\ud x<\frac{\eps}{2}.\]
	In fact if it is not the case, then for any $R>0$, there exists $v\in V$ such that
	\[\int_{B_R^c}\left(|v(x)|^2+|\nabla\times v(x)|^2\right)\ud x\ge \frac{\eps}{2}.\]
	Then
	\[\|v\|^2_U\ge \int_{B_R^c}\left(|v(x)|^2+|\nabla\times v(x)|^2\right)|x|^2\ud x\ge \frac{\eps}{2}R^2,\]
	which contradict to the boundness of $V$ in $U$. 
	
	By the Rellich-Kondrachov Theorem, $H^2(B_R;\R^3)$ compactly embedded into $H^1(B_R;\R^3)$, so the embedding $ H^2(B_R;\R^3)\hookrightarrow \bb Y_R$ is also compact, where
	\[\bb Y_R:=\{y\in\bb Y:y(x)=0 \textrm{ for }|x|>R\}.\]
	Hence there exist $n(\eps)\in\bb N$ and $y_1,\ldots,y_{n(\eps)}\in \bb Y_R$ such that for any $v\in V$ there is some $j\in\{1,2,\ldots,n(\eps)\}$ such that
	\[\int_{B_R}\left(|v(x)-y_j(x)|^2+|\nabla\times(v(x)-y_j(x))|^2\right)\ud x<\frac{\eps}{2}.\]
	Since $y_j=0$ outside $B_R$, we actually have
	\begin{align*}
		\|v-y_j\|_{\bb Y}^2&=\int_{\R^3}\left(|v(x)-y_j(x)|^2+|\nabla\times(v(x)-y_j(x))|^2\right)\ud x\\
		&=\int_{B_R}\left(|v(x)-y_j(x)|^2+|\nabla\times(v(x)-y_j(x))|^2\right)\ud x\\
		&\qquad+\int_{B_R^c}\left(|v(x)|^2+|\nabla\times v(x)|^2\right)\ud x<\eps.
	\end{align*}
	Therefore the embedding $U\hookrightarrow \bb Y$ is compact and the proof is complete.
\end{proof}

\begin{defn}[Aldous condition]\label{def:Aldous}
  Let $(\Omega,\mathcal{F},\mathbb{P})$ be a probability space with a filtration $\mathbb{F}$. Let $(S,\rho)$ be a separable metric space with the metric $\rho$. We say that  $\{X_n(t)\}$, $t\in[0,T]$, of $S$-valued processes satisfies the \textit{Aldous condition} iff $\forall \eps>0$, $\forall \eta>0$, $\exists \delta>0$ such that for every sequence $\{\tau_n\}$ of $\mathbb{F}$-stopping times with $\tau_n\leq T$ a.s.  one has:
  \begin{equation}\label{eq:Ald con}
  \sup_{n\in\mathbb{N}}\sup_{0\leq\theta\leq \delta}\mathbb{P}\{\rho(X_n(\tau_n+\theta),X_n(\tau_n))\geq\eta\}\leq\eps.
  \end{equation}
\end{defn}
We will also need the following Tightness Criterion.
\begin{lem}[Tightness Criterion]\label{lem:tightcri}(\cite{ZB&EM}, Cor 3.10)
Let $(\Omega,\mathcal{F},\mathbb{P})$ be a probability space with the filtration $\mathbb{F}$. Let $H$ be a separable Hilbert space, $U$ be another Hilbert space such that the embedding $U\hookrightarrow H$ is compact and dense, $U^*$ be the dual space of $U$.
  Let $\{X_n(t)\}_{n\in\mathbb{N}}$, $t\in[0,T]$ be a sequence of continuous $\mathbb{F}$-adapted  $U^*$ valued process such that
  \begin{trivlist}
    \item[(a)] there exists a positive constant $C$ such that
    \[\sup_{n\in\mathbb{N}}\mathbb{E}\left[\sup_{s\in [0,T]}\left\|X_n(s)\right\|_{{H}}\right]\leq C.\]
    \item[(b)] $\{X_n\}_{n\in\mathbb{N}}$ satisfies the Aldous condition \eqref{eq:Ald con} in $U^*$.
  \end{trivlist}
  Then the laws of $X_n$ on $C([0,T];U^*)\cap L_w^2([0,T];{H})$ are tight.
\end{lem}

	The next Lemma states that the $\til W_h$ introduced in Lemma \ref{prop:tilMn} is a $Q$-Wiener process.
\begin{lem}\label{lem:distrWP}
	Let $\HH$ be a separable Hilbert space. Let $(\Omega, \cl F, \bb P)$ be a probability space and $W_h$ is a $\HH$-valued $Q$-Wiener process on it for some covariance operator $Q$. Let $(\til \Omega,\til{\cl F},\til{\bb P})$ be another probability space and $\til{W}_h$ is a $\HH$-valued adapted stochastic process on it. $W_h$ and $\til{W}_h$ have same distribution on $C([0,T];\HH)$ for some $T>0$. Then $\til{W}_h$ is also a $Q$-Wiener process. 
\end{lem}
\begin{proof}
	To show $\til{W}_h$ is a Wiener process, we will show that the following four conditions are satisfied:
	\begin{trivlist}
		\item[(i)] $\til{W}_h(0)=0$, $\til{\bb P}$-a.s.;
		\item[(ii)] $\til{W}_h$ has continuous trajectories,  $\til{\bb P}$-a.s.;
		\item[(iii)] $\til{W}_h$ has independent increments;
		\item[(iv)] $\mathscr{L}(\til{W}_h(t)-\til{W}_h(s))=N(0,(t-s)Q)$.
	\end{trivlist}
	Now let us prove them one by one. We will repeatedly use the fact that all the cylindrical sets in $C([0,T];\HH)$ are Borel sets and the assumption that $W_h$ and $\til{W}_h$ have same distribution on $C([0,T];\HH)$.
	
	\begin{trivlist}
		\item[(i)] $\til{W}_h(0)=0$, $\til{\bb P}$-a.s.;
		
		Note that $\{x\in C([0,T];\HH):x_0=0\}$ is a cylindrical set, so it is a Borel set, therefore we have the following equalities:
		\setlength\arraycolsep{2pt}{
			\begin{eqnarray*}
				&&\til{\bb P}(\til{W}_h(0)=0)\\
				&=&\til{\bb P}\big\{\til\omega\in \til\Omega:\til{W}_h(\cdot,\til\omega)\in\{x\in C([0,T];\HH):x_0=0\}\big\}\\
				&=&\mathbb{P}\big\{\omega\in \Omega:W(\cdot,\omega)\in\{x\in C([0,T];\HH):x_0=0\}\big\}\\
				&=&\mathbb{P}(W_h(0)=0)=1.
		\end{eqnarray*}}
		Hence the condition (i) is proved. 
		\item[(ii)] $\til{W}_h$ has continuous trajectories,  $\til{\bb P}$-a.s.;
		
		Since $W_h$ and $\til{W}_h$ have same distribution on $C([0,T];\HH)$, we have
		\[\til{\bb P}\{\til{W}_h\in C([0,T];\HH)\}=\bb P\{W_h\in C([0,T];\HH)\}=1.\]
		Hence the condition (ii) is proved.
		
		\item[(iii)] $\til{W}_h$ has independent increments;
		
		For $0\leq t_1<t_2\leq t_3<t_4\leq T$, any $A,B\in\cl{B}(\HH)$, we have the following equality:
		\setlength\arraycolsep{2pt}{
			\begin{eqnarray*}
				&&\{\til{W}_h(t_2)-\til{W}_h(t_1)\in A\}\cap\{ \til{W}_h(t_4)-\til{W}_h(t_3)\in B\}\\
				&=&\big\{\til\omega\in\til\Omega:\til{W}_h(\cdot,\til\omega)\in \{x:x_{t_2}-x_{t_1}\in A\}\cap\{x:x_{t_4}-x_{4_3}\in B\}\big\}.
		\end{eqnarray*}}
		$A,B\in\cl{B}(\HH)$, so $\{x:x_{t_2}-x_{t_1}\in A\}$, $\{x:x_{t_4}-x_{4_3}\in B\}$ are cylindrical sets and hence they are Borel sets, so $\{x:x_{t_2}-x_{t_1}\in A\}\cap\{x:x_{t_4}-x_{4_3}\in B\}$ is also a Borel set in $C([0,T];\HH)$. Since $\til{W}_h$ and $W_h$ have same law and $W_h$ has independent increments, we have
		\setlength\arraycolsep{2pt}{
			\begin{eqnarray*}
				&&\til{\bb P}\Big(\big\{\til\omega\in\til\Omega:\til{W}_h(\cdot,\til\omega)\in \{x:x_{t_2}-x_{t_1}\in A\}\cap\{x:x_{t_4}-x_{t_3}\in B\}\big\}\Big)\\
				&=&\mathbb{P}\Big(\big\{\omega\in\Omega:W_h(\cdot,\omega)\in \{x:x_{t_2}-x_{t_1}\in A\}\cap\{x:x_{t_4}-x_{4_3}\in B\}\big\}\Big).
		\end{eqnarray*}}
		Hence
		\setlength\arraycolsep{2pt}{
			\begin{eqnarray*}
				&&\til{\bb P}\big(\{\til{W}_h(t_2)-\til{W}_h(t_1)\in A\}\cap\{ \til{W}_h(t_4)-\til{W}_h(t_3)\in B\}\big)\\
				&=&\mathbb{P}\Big(\big\{\omega\in\Omega:W_h(\cdot,\omega)\in \{x:x_{t_2}-x_{t_1}\in A\}\cap\{x:x_{t_4}-x_{4_3}\in B\}\big\}\Big)\\
				&=&\mathbb{P}\big(\{W_h(t_2)-W_h(t_1)\in A\}\cap\{ W_h(t_4)-W_h(t_3)\in B\}\big)\\
				&=&\mathbb{P}\big(\{W_h(t_2)-W_h(t_1)\in A\}\big)\mathbb{P}\big(\{ W_h(t_4)-W_h(t_3)\in B\}\big)\\
				&=&\mathbb{P}\big(W_h\in\{x\in C([0,T];\HH):x_{t_2}-x_{t_1}\in A\}\big)\mathbb{P}(W_h\in\{x\in C([0,T];\HH):x_{t_4}-x_{t_3}\in B\}\big)\\
				&=&\til{\bb P}\big(\til{W}_h\in\{x\in C([0,T];\HH):x_{t_2}-x_{t_1}\in A\}\big)\til{\bb P}(\til{W}_h\in\{x\in C([0,T];\HH):x_{t_4}-x_{t_3}\in B\}\big)\\
				&=&\til{\bb P}\big(\{\til{W}_h(t_2)-\til{W}_h(t_1)\in A\}\big)\mathbb{P}\big(\{ \til{W}_h(t_4)-\til{W}_h(t_3)\in B\}\big)
		\end{eqnarray*}}
		Hence $\til{W}_h(t_2)-\til{W}_h(t_1)$ and $\til{W}_h(t_4)-\til{W}_h(t_3)$ are independent, so the condition (iii) is proved.
		
		\item[(iv)] $\mathscr{L}(\til{W}_h(t)-\til{W}_h(s))=N(0,(t-s)Q)$ for all $s,t\in[0,T]$.
		
		Similarly as before, since $\til{W}_h$ and $W_h$ have same law, we have
		\begin{align*}	
			&\til{\bb P}\{\til{W}_h(t)-\til{W}_h(s)\in A\}=\til{\bb P}\big\{\til\omega\in\til\Omega:\til{W}_h(\cdot,\til\omega)\in \{x:x_{t}-x_{s}\in A\}\big\}\\
			=&\bb P\big\{\omega\in\Omega:{W}_h(\cdot,\til\omega)\in \{x:x_{t}-x_{s}\in A\}\big\}=\bb P\{W_h(t)-W_h(s)\in A\}.
		\end{align*}
		for all $A\in\mathcal{B}(\HH)$ and $s,t\in [0,T]$. Therefore since $W_h$ is a Wiener process, we have $\mathscr{L}(\til{W}_h(t)-\til{W}_h(s))=\mathscr{L}(W_h(t)-W_h(s))=N(0,(t-s)Q)$ for all $s,t\in[0,T]$. Hence the condition (iv) is proved.
	\end{trivlist}
	Therefore the proof of Lemma \ref{lem:distrWP} is complete.
\end{proof}

The next two results are used in the proof of Proposition \ref{prop:LLCn}.
\begin{lem}[Kuratowski Theorem]\label{thm:kuratowski}\cite{key-1}
 Let $X_1,X_2$ be Polish spaces equipped with their Borel $\sigma$-field $\B(X_1),\B(X_2)$, and $\vp:X_1\longrightarrow X_2$ be a one to one Borel measurable map, then for any $E\in\B(X_1)$, $\vp(E)\in\B(X_2)$.
\end{lem}
\begin{lem}[\cite{Rudin Func}, Page 66, Thm 3.12]\label{lem:wcloeqclo}
  Suppose $E$ is a convex subset of a locally convex space $X$. Then the weak closure $\overline{E}_w$ of $E$ is equal to its original closure $\overline{E}$.
\end{lem}

The next result is a version of It\^o formula, which is used in the proof of Lemma \ref{lem:M=1}.
\begin{lem}\cite{Pardoux}(Th. 1.2)\label{lem:Pardoux}
  Let $V$ and $H$ be two separable Hilbert spaces, such that $V\hookrightarrow H$ continuously and densely. We identify $H$ with it's dual space. And let $M^2(0,T;\mathbb{H})$  denote the space of  $H$-valued measurable process with the filtered probability space $(\Omega,(\mathcal{F}_t)_{t\in [0,T]},\mathbb{P})$ which satisfy: $\vp\in M^2(0,T;\mathbb{H})$ if and only if
  \begin{trivlist}
    \item[(i)] $\vp(t)$ is $\mathcal{F}_t$ measurable for almost every $t$;
    \item[(ii)] $\mathbb{E}\int_0^t|\vp(t)|^2\ud t<\infty.$
  \end{trivlist}
  We suppose that
  \setlength\arraycolsep{2pt}{
  \begin{eqnarray*}
    &&u\in M^2(0,T;V),\quad u_0\in \mathbb{H},\quad v\in M^2(0,T;V'),\\
    &&\mathbb{E}\int_0^T\sum_{j=1}^\infty\|z_j(t)\|_{\mathbb{H}}^2\ud t<\infty,
  \end{eqnarray*}}
  with
  \[u(t)=u_0+\int_0^tv(s)\ud s+\sum_{j=1}^\infty\int_0^tz_j(s)\ud W_j(s).\]
  Let $\gamma$ be a twice differentiable functional on $H$, which satisfies:
  \begin{trivlist}
    \item[(i)] $\gamma$, $\gamma'$ and $\gamma''$ are locally bounded.
    \item[(ii)] $\gamma$ and $\gamma'$ are continuous on $H$.
    \item[(iii)] Let $\mathscr{L}^1(H)$ be the Banach space of all the trace class operators on $H$. Then $\forall Q\in \mathscr{L}^1(H)$, $Tr[Q\circ \gamma'']$ is a continuous functional on $H$.
    \item[(iv)] If $u\in V$, $\psi'(u)\in V$; $u\mapsto\gamma'(u)$ is continuous from $V$ (with the strong topology) into $V$ endowed with the weak topology.
    \item[(v)] $\exists k$ such that $\|\gamma'(u)\|_{V}\leq k(1+\|u\|_{V})$, $\forall u\in V$.
  \end{trivlist}
  Then $\mathbb{P}$ almost surely,
  \setlength\arraycolsep{2pt}{
  \begin{eqnarray*}
    \gamma(u(t))&=&\gamma(u_0)+\int_0^t{}_{V'}\llangle v(s),\gamma'(u(s))\rrangle_{V} \ud s+\sum_{j=1}^\infty\int_0^t{}{}_{H}\llangle\gamma'(u(s)),z_j(s)\rrangle_{H}\ud W_j(s)\\
    &&\qquad+\frac{1}{2}\sum_{j=1}^\infty\int_0^t{}_{H}\llangle\gamma''(u(s))z_j(s),z_j(s)\rrangle_{H}\ud s.
  \end{eqnarray*}}
\end{lem}

The last Lemma is used in the proof of Lemma \ref{thm:tilMregt}.
\begin{lem}[Kolmogorov continuity]\label{lem:Kol con}
  Let $\{u(t)\}_{t\in[0,T]}$ be a stochastic process with values in a separable Banach space $\bb X$, such that for some $C>0$, $\eps>0$, $\delta>1$ and all $t,s\in[0,T]$,
  \[\mathbb{E}\big\|u(t)-u(s)\big\|_{\bb X}^\delta\leq C|t-s|^{1+\eps}.\]
  Then there exists a version of $u$ with $\mathbb{P}$ almost surely trajectories being H$\ddot{\textrm{o}}$lder continuous functions with an arbitrary exponent smaller than $\frac{\eps}{\delta}$.
\end{lem}

\end{document}